\documentclass{amsart}

\title{Periods in equivariant and motivic contexts}
\author{Martin Gallauer}
\thanks{MG was supported by the Engineering and Physical Sciences Research Council under grant number UKRI3692}
\address{Martin Gallauer, Warwick Mathematics Institute, Coventry CV4 7AL, UK}
\email{martin.gallauer@warwick.ac.uk}
\urladdr{https://warwick.ac.uk/mgallauer}
\address{Ivo Dell'Ambrogio, Univ. Artois, UR 2462, Laboratoire de Mathématiques de Lens, F-62300 Lens, France}
\email{ivo.dellambrogio@univ-artois.fr}
\urladdr{https://idellambrogio.github.io}

\address{Sam K. Miller, Department of Mathematics, University of Georgia, Athens GA 30602, United States of America}
\email{sam.miller@uga.edu}
\urladdr{https://www.samkmiller.com/}

\makeatletter
\let\@wraptoccontribs\wraptoccontribs
\makeatother

\contrib[with an appendix by]{Ivo Dell'Ambrogio}
\contrib[and an appendix by]{Martin Gallauer and Sam K.\,Miller}

\makeatletter
\newcommand{\sectionauthor}[1]{{\parindent0pt\center\linespread{1.1}\large\scshape#1\par\nobreak\vspace*{35pt}}
  \@afterheading }
\makeatother

\usepackage[utf8]{inputenc}
\usepackage{amsmath,amssymb,amsthm}
\usepackage{mathtools}
\usepackage[usenames,dvipsnames]{xcolor}

\usepackage{hyperref}
\hypersetup{colorlinks=true,linkcolor={Brown},citecolor={Brown},urlcolor={Brown},bookmarksdepth=2}
\usepackage{float}
\usepackage[
backend = biber,
style = alphabetic,
citestyle = alphabetic,
maxbibnames = 99,
url = false]{biblatex}

\usepackage[compress]{cleveref}
\usepackage{thmtools, thm-restate}
\declaretheoremstyle[spaceabove=.8\baselineskip,spacebelow=.8\baselineskip,headfont=\bfseries,notefont=\normalfont,bodyfont=\itshape,postheadspace=.5em]{thms}

\declaretheoremstyle[spaceabove=.8\baselineskip,spacebelow=.8\baselineskip,headfont=\bfseries,notefont=\normalfont,bodyfont=\normalfont,postheadspace=.5em]{defn}

\numberwithin{equation}{section}\theoremstyle{thms}
\newtheorem{thrm}[equation]{Theorem}
\newtheorem{cor}[equation]{Corollary}

\newtheorem{lem}[equation]{Lemma}
\newtheorem{prop}[equation]{Proposition}

\theoremstyle{defn}
\newtheorem{cns}[equation]{Construction}
\newtheorem{defn}[equation]{Definition}
\newtheorem*{defn*}{Definition}
\newtheorem{exa}[equation]{Example}

\newtheorem{notn}[equation]{Convention}
\newtheorem*{notn*}{Notation}
\newtheorem{rmk}[equation]{Remark}

\crefname{prop}{Proposition}{Propositions}
\crefname{hyp}{Hypothesis}{Hypotheses}
\crefname{cor}{Corollary}{Corollaries}
\crefname{lem}{Lemma}{Lemmata}
  \crefname{cns}{Construction}{Constructions}
  \crefname{def}{Definition}{Definitions}
  \crefname{rmk}{Remark}{Remarks}
  \crefname{notn}{Convention}{Conventions}
  \crefname{exa}{Example}{Examples}
  \crefname{thrm}{Theorem}{Theorems}
  \crefname{thrmprime}{Theorem'}{Theorem'}
\usepackage[shortlabels,inline]{enumitem}

\usepackage[T1]{fontenc} \usepackage[osf]{Baskervaldx} \usepackage[vvarbb]{newtxmath} \usepackage[cal=boondoxo]{mathalfa} \usepackage{bbold}\usepackage{xspace}

\usepackage{tikz}
\usetikzlibrary{arrows,shapes,positioning,backgrounds,cd}

\tikzset{
commutative diagrams/.cd,
arrow style=tikz,
diagrams={>=latex}}

\tikzset{-,>=stealth',shorten >=2pt,shorten <=2pt,
  main node/.style={circle,fill=blue!20,inner sep=1.6pt,font=\sffamily\tiny\bfseries},
  split node/.style={circle,fill=red!20,inner sep=1.6pt,font=\sffamily\tiny\bfseries},
  lower node/.style={circle,fill=orange!20,inner sep=1.6pt,font=\sffamily\tiny\bfseries}
}

\usepackage[all]{xy}

\newcommand{\nc}{\newcommand}

\nc{\loccit}{{\sl loc.\ cit.}\xspace}
\nc{\inv}{^{-1}}
\nc{\unit}{\mathbb{1}}
\nc{\cat}{\mathscr}
\nc{\Catex}{\mathrm{Cat}^{\mathrm{\footnotesize ex}}}
\nc{\Catperf}{\mathrm{Cat}^{\mathrm{\footnotesize perf}}}
\nc{\el}{\mathrm{ell}}
\nc{\gr}{\mathrm{gr}}
\nc{\Hm}{\mathrm{H}}
\nc{\iso}[1]{\cP_{\mathrm{iso}}(#1)}
\nc{\cP}{\cat{P}}
\nc{\cQ}{\cat{Q}}
\nc{\cO}{\cat{O}}
\nc{\cI}{\cat{I}}
\nc{\cJ}{\cat{J}}
\nc{\cK}{\cat{K}}
\nc{\cL}{\cat{L}}
\nc{\cM}{\cat{M}}
\nc{\cR}{\cat{R}}
\nc{\cS}{\cat{S}}
\nc{\cT}{\cat{T}}
\nc{\cV}{\cat{V}}
\nc{\gm}{\mathfrak{m}}
\nc{\gp}{\mathfrak{p}}
\nc{\gq}{\mathfrak{q}}
\nc{\pgeq}{\stackrel{p}{\geq}}
\nc{\pleq}{\stackrel{p}{\leq}}
\nc{\psim}{\stackrel{p}{\sim}}
\nc{\ZZ}{\mathbb{Z}}
\nc{\Aff}{\mathbb{A}}
\nc{\EE}{\mathbb{E}}
\nc{\FF}{\mathbb{F}}
\nc{\KK}{\mathbb{K}}
\nc{\PP}{\mathbb{P}}
\nc{\QQ}{\mathbb{Q}}
\nc{\RR}{\mathbb{R}}
\nc{\CC}{\mathbb{C}}
\nc{\resp}[1]{(resp.\xspace #1)}
\nc{\sto}{\rightsquigarrow}
\nc{\Weyl}[2]{{#1}/\!\!/{#2}}
\nc{\Orop}[1]{\operatorname{Or}^{\operatorname{op}}_{#1}}
\nc{\Ch}{\operatorname{Ch}}
\nc{\dmo}{\DeclareMathOperator}
\dmo{\calg}{CAlg}
\dmo{\car}{char}
\dmo{\cof}{cof}
\dmo{\colim}{colim}
\dmo{\comp}{comp}
\dmo{\hend}{end}
\dmo{\fmod}{mod}
\dmo{\Fun}{Fun}
\dmo{\Gal}{Gal}
\dmo{\gen}{gen}
\dmo{\ho}{ho}
\dmo{\Infl}{Infl}
\dmo{\Mod}{Mod}
\dmo{\id}{id}
\dmo{\per}{per}
\dmo{\Per}{Per}
\dmo{\Perf}{Perf}
\dmo{\Pic}{Pic}
\dmo{\QCoh}{QCoh}
\dmo{\rep}{rep}
\dmo{\Ind}{Ind}
\dmo{\Res}{Res}
\dmo{\SH}{SH}
\dmo{\Sp}{Sp}
\dmo{\DM}{DM}
\nc{\et}{\acute{e}t}
\dmo{\DMet}{DM_{\et}}
\dmo{\DATM}{DATM}
\dmo{\DAM}{DAM}
\dmo{\DPerm}{DPerm}
\dmo{\free}{free}
\dmo{\proj}{proj}
\dmo{\geom}{gm}
\dmo{\Db}{D_b}
\dmo{\D}{D}
\dmo{\Proj}{Proj}
\dmo{\Spc}{Spc}
\dmo{\Spec}{Spec}
\dmo{\Specgr}{Spec^{\gr}}
\dmo{\Spch}{Spc^{\bullet}}
\let\Spech\Spec
\let\Spch\Spc
\dmo{\stmod}{stmod}
\dmo{\supp}{supp}

\nc{\martin}[1]{{\color{brown}{#1}}}

\newcommand{\ie}{{i.e.\ }}
\dmo{\End}{End}
\dmo{\cone}{cone} \dmo{\2Pic}{2Pic} \nc{\Kth}{\mathrm{K}} \dmo{\Id}{Id}
\newcommand{\hmg}{\textrm{hmg}} \dmo{\Mor}{Mor}

\begin{document}
\maketitle{}

\begin{abstract}
We define the period as a multiplicative characteristic of stably symmetric monoidal $\infty$-categories, develop its basic properties, and study many examples, with a focus on `ordinary' equivariant and motivic homotopy theory.
We apply the findings to isotropic points in motivic tt-geometry.
\end{abstract}
\setcounter{tocdepth}{1}
\tableofcontents{}
\section{Introduction}
\label{sec:introduction}

\subsection{Characteristics}
\label{sec:characteristics}

Let $R$ be a (commutative, unital) ring. A basic invariant is its characteristic, the non-negative integer satisfying
\[
\langle\car(R)\rangle=\ker\left(\ZZ\xrightarrow{\mathrm{can}} R\right).
\]
Being defined in terms of addition and multiplication, it propagates along any ring morphism $R\to S$, meaning that $\car(S)\bigm\vert\car(R)$.
In particular, it leads to the prime characteristics $\gp\in\Spec(\ZZ)$ that partition fields into an infinite countable family.
The characteristic is a fundamental organizing principle in algebraic geometry, describing specific techniques and phenomena such as `mixed characteristic', `characteristic~$p$' or `(equi-)characteristic~$0$'.

Now let $\cK$ be an (essentially small) stably symmetric monoidal $\infty$-category, a categorification of a ring.
There are several invariants that could serve functions similar to that of~$\car(R)$.
One is to replace the initial ring~$\ZZ$ by the initial categorified ring~$\Sp^\omega$ of finite spectra and define the characteristic as the kernel of
\begin{equation}
\label{eq:initial-map}
\Sp^\omega\xrightarrow{\mathrm{can}}\cK.
\end{equation}
This kernel is a thick $\otimes$-ideal that can be identified with a subset of $\Spec(\Sp^\omega)$.
Here, we use that categorified rings possess an associated (affine) \emph{tt-scheme}: as a space this was defined by Balmer~\cite{balmer:spectrum} (for tensor triangulated categories); in~\cite{aoki2025higherzariskigeometry} it was enhanced to a structured space in the sense of~\cite{lurie:dagV}.
In particular, the prime characteristics are now pairs $(\gp,n)$ where $\gp\in\Spec(\ZZ)$ as before, and (if $\gp\neq 0$) $1\leq n\leq \infty$ is the chromatic height~\cite{MR1652975}.
Undoubtedly this is a useful and important organizing principle in tt-geometry.

\subsection{Periods}
\label{sec:periods}

In this article we are interested in another such basic invariant, of multiplicative rather than additive type, obtained from~\eqref{eq:initial-map} by passing to invertible objects:
\begin{defn}
The \emph{period} of~$\cK$ is the non-negative integer satisfying
\[
\langle\per(\cK)\rangle=\ker\left(\ZZ=\Pic(\Sp^\omega)\xrightarrow{\mathrm{can}}\Pic(\cK)\right).
\]
\end{defn}
In other words, it is the smallest $d>0$ such that $\Sigma^d\unit\cong\unit$ (or $0$ if no such~$d$ exists).
While the period of many familiar examples like spectra or derived categories is~$0$, one doesn't need to look far for periodic ones.
Periodic ring spectra are ubiquitous in stable homotopy theory and their perfect modules are obviously periodic; this applies, among other things, to most of the tt-fields for spectra (Morava $K$-theories).
Another source of examples are periodic derived categories, orbit categories and categories of matrix factorizations.
As we will expand on below, many tt-fields in modular representation theory and motives also are, or ought to be, periodic.

Being defined in terms of $\otimes$ and suspension, the period propagates along any morphism $\cK\to\cL$, meaning that $\per(\cL)\bigm\vert\per(\cK)$.
This leads to a stratification by periods of~$\Spec(\cK)$ and more generally, of any tt-scheme, just as for the characteristic.
We also believe that $\cK$ being periodic or non-periodic determines some of the important phenomena one encounters and techniques one employs.
For example, t-structures and weight structures are two of the most useful tools for getting to grips with~$\cK$.
But they can \emph{never} be of use if~$\cK$ is periodic.
On the other hand, in the latter case the periodicity itself offers a reduction in complexity of some form, something that is often exploited in stable homotopy theory.
We will not try to substantiate this dichotomy here, and instead illustrate the use of periods as an organizing principle in examples---to which we now turn.

\subsection{Motivic tt-geometry}
\label{sec:motivation}
At the origin of this article was the construction of iso\-tro\-pic realizations due to Vishik~\cite{MR4454343} and the subsequent discovery~\cite{MR4768634,MR4905541} of an associated large new class of \emph{isotropic} points on $\Spec(\SH(\FF)^\omega)$ and~$\Spec(\DM(\FF;\ZZ/\ell)^\omega)$ (for $\FF$ a field of characteristic zero, $\ell$ any prime number).
The analysis of these two spectra, associated to Morel--Voevodsky's $\Aff^1$-motivic spectra and Voevodsky's derived category of motives, respectively, constitutes arguably the major open problem in motivic tt-geometry and the cited works made significant inroad.

To give an idea of just how significant, we mention for example that finitely many points on~$\Spec(\DM(\RR;\ZZ/2)^\omega)$ had been known before, whereas the set of isotropic points (and indeed the whole space) have cardinality~$2^{\mathfrak{c}}$ ($\mathfrak{c}$ denoting the continuum).
A similar statement could be made about~$\SH(\RR)$ and over other fields.
Our goal was to understand these developments and place them in the context of what was already known.

For example, let $\iota\colon\DAM(\FF;\ZZ/\ell)\hookrightarrow\DM(\FF;\ZZ/\ell)$ denote the inclusion of Artin motives (i.e., the localizing subcategory generated by zero-dimensional varieties).
In joint work with Balmer~\cite{MR4946248,MR4866349}, we determined the space underlying $\Spec(\DAM(\FF;\ZZ/\ell)^\omega)$ completely.
It is itself a typically complicated and high-dimen\-sio\-nal spectral space that can be described in terms of the subgroup structure of the absolute Galois group and the group cohomology of its subquotients.
We establish in \Cref{sec:isotropic-points} that under the induced map
\begin{equation}
\label{eq:DAM-vs-DM}
\iota^*\colon\Spec(\DM(\FF;\ZZ/\ell)^\omega)\to\Spec(\DAM(\FF;\ZZ/\ell)^\omega),
\end{equation}
all isotropic points get mapped to closed points.
(In fact, we give an explicit formula.)
Since it is known from general tt-principles that~\eqref{eq:DAM-vs-DM} is surjective, this shows that isotropic points are (very) far from exhausting $\Spec(\DM(\FF;\ZZ/\ell)^\omega)$.
For the vast majority of points in the target of~\eqref{eq:DAM-vs-DM}, we do not know any element in the fiber.
The main take-away is that despite the recent breakthroughs we seem to have barely scratched the surface of motivic tt-geometry (\Cref{rmk:armer-tor}).

At least for~$\ell=2$, \cite{vishik2025balmerspectrumvoevodskymotives} shows that the isotropic points themselves are closed so the fact that their images in~$\Spec(\DAM(\FF;\ZZ/2)^\omega)$ are as well may not be surprising.
(Although this doesn't seem to be a formal consequence.)
Our argument proceeds as follows: It is clear by construction that isotropic points are non-periodic hence, by the propagation of periodicity, they cannot map to periodic ones.
It then suffices to show that all non-closed points in~$\Spec(\DAM(\FF;\ZZ/\ell)^\omega)$ are periodic, and here we rely heavily on recent work of G\'omez and Miller~\cite{miller2025permutationtwistedcohomologyremixed,GM26} that in effect produces enough sections of line bundles on this space.

We do not claim that this is the simplest or most direct argument possible.
But it illustrates how periods can be effective in organizing tt-geometry in more complicated situations.
We imagine that their usefulness will only grow as explorations of the landscape in motivic tt-geometry lead us further afield.
And our result brings to the fore the challenge to exhibit periodic points of $\Spec(\DM(\FF;\ZZ/\ell)^\omega)$.
\subsection{Contents}
\label{sec:contents}

The first half of the article is straightforward: we simply lay the foundations to run the argument above as smoothly as possible.
This includes in \Cref{sec:periodicity:-rings,sec:periodicity:oo-rings,sec:periodicity} introducing periods for commutative ring spectra, stably symmetric monoidal $\infty$-categories, the corresponding period stratification on their associated spaces, and their basic properties.
We observe that these notions all depend on the underlying graded rings and tt-categories only, respectively, and we therefore work in this generality.
Determining periods in tt-geometry can be quite a bit harder than determining the characteristic of rings and so, in \Cref{sec:comparison,sec:transfer}, we establish some tools that aid in d\'evissage arguments.
One of the main ones is decategorification: comparing tt-categories to (graded) commutative rings arising as sections of families of line bundles.
Here we highlight the substantial contribution (in Appendix~\ref{sec:gener-comp-maps}) of Dell'Ambrogio (based on~\cite{MR3163513}) that provides a robust framework for talking about ample families of line bundles and divisorial tt-categories.
We hope this framework will be useful for tt-geometers in other contexts as well.

The second, more substantial, half of the article applies the theory, exhibiting the period stratification in many examples.
Some are easily pulled from the literature (\Cref{sec:examples}).
The ones that take more effort include notably the derived and stable categories in modular representation theory, as well as derived permutation modules (\Cref{sec:dperm}).
In collaboration with Miller we generalize an earlier result to give a complete description of the periodic locus in the latter case (\Cref{sec:peri-locus-perm}).
Through the Galois correspondence, these derived permutation modules (or, equivalently, modules for the Bredon cohomology spectrum) correspond to Artin motives.
We'll use this in \Cref{sec:motives,sec:isotropic-points} to explain our application to motives and isotropic points alluded to above.
At the same time we provide a bestiary of many of the known points in motivic tt-geometry together with their periods.
We hope this will be useful for anyone interested in exploring this area of research.

\subsection{Acknowledgments}
\label{sec:acknowledgments}

I would like to thank Ivo Dell'Ambrogio for writing the appendix and for being so accommodating.
I am also thankful to Sam Miller for making~\cite{miller2025permutationtwistedcohomologyremixed} available at that point, and for the collaboration in \Cref{sec:peri-locus-perm} at a later point.
Dave Benson, Drew Heard and Greg Stevenson helpfully answered questions related to this work.
Finally, I am grateful to Paul Balmer, Yorick Fuhrmann, Beren Sanders and Greg Stevenson for insightful comments on a draft version.

\subsection{Conventions}
\label{sec:conventions}
We follow the conventions of~\cite{MR2522659,HA} for $\infty$-categories.
Let $\Catex$ denote the $\infty$-category of stable $\infty$-categories and exact functors (with its usual tensor product).
A \emph{stably symmetric monoidal $\infty$-category} is a commutative algebra object in~$\Catex$.
That is, it is a stable $\infty$-category with a symmetric monoidal structure such that the tensor product is exact in each variable.

For tt-categories we follow the conventions of~\cite{balmer:spectrum}.
In particular, a \emph{tt-category} is a triangulated category with a symmetric monoidal structure such that the tensor product is exact in each variable.

Idempotent completion is denoted by $(-)^\natural$.
Given $\cK\in\calg(\Catex)$ there is an associated structured space~$\Spec(\cK):=\Spec(\cK^{\natural})$ defined in~\cite{aoki2025higherzariskigeometry}.
Its underlying space~$\Spc(\cK)$ is canonically identified with the spectrum~$\Spc(\ho(\cK))$ of the underlying tt-category in the sense of~\cite{balmer:spectrum}.
We will (try to) consistently distinguish between spectra of various things as spaces (denoted $\Spc$) and with a `locally ringed space' structure (denoted $\Spec$).

\section{Periods of rings\ldots{}}
\label{sec:periodicity:-rings}

In this section we fix terminology concerning periodicity in ring spectra and graded rings.

\begin{notn}
\label{notn:graded-ring}
A \emph{graded ring} refers to a $\ZZ$-graded commutative ring in a general sense.
That is, it is a monoid $R=\oplus_{d\in\ZZ}R_d$ in $\ZZ$-graded abelian groups such that there exists $\epsilon\in R_0^\times$ satisfying $rs=\epsilon^{\deg(r)\deg(s)}sr$ for any homogeneous elements $r,s$.
\end{notn}

\begin{exa}
\label{exa:dirac-ring}
The most important case for us is when $\epsilon$ can be chosen to be~$-1$.
These are called \emph{Dirac rings} in~\cite{hesselholt-pstragowski:dirac1}.
This is, in other words, a commutative monoid in graded abelian groups, if one endows the latter with the symmetry constraint that respects the Koszul sign rule.
\end{exa}

\begin{defn}
Let $d>0$ be an integer.
A graded ring~$R$ is \emph{$d$-periodic} if there exists a unit $u\in R_d$ in degree~$d$.
We say $R$ is \emph{periodic} if it is $d$-periodic for some $d>0$.
The \emph{period} of~$R$, denoted $\per(R)$, is the smallest degree~$d>0$ of a unit in~$R$.
We also set $\per(R)=0$ and call $R$ \emph{non-periodic} if no unit of positive degree exists.
\end{defn}

\begin{rmk}
\label{rmk:E_n}
Let $R$ be a sufficiently commutative ring spectrum, by which we mean an $\EE_n$-algebra in spectra, for $2\leq n\leq \infty$.
Then $\pi_*R$ is a Dirac ring hence in particular a graded ring in the sense above.
We say that $R$ is $d$-periodic if $\pi_*R$ is.
Similarly, we will transfer notions like period and local period (to be introduced below) to such ring spectra.
Also, while in the sequel we focus almost exclusively on $\EE_\infty$-rings, much of it would apply more generally to sufficiently commutative ring spectra.
\end{rmk}

\begin{exa}
\label{exa:even-periodic}
An \emph{even periodic} $\EE_\infty$-ring is a $2$-periodic $\EE_\infty$-ring whose odd homotopy groups vanish.
These play an important role in stable homotopy theory due to their connection with complex orientations and formal group laws.
Examples include complex $K$-theory $KU$, Morava $E$-theory~$E(n)$ at height~$n$ (and an implicit prime~$p$), elliptic spectra (arising from \'etale classifying maps $\Spec(R)\to\mathcal{M}_{\mathrm{ell}}$ through Goerss--Hopkins--Miller).
\end{exa}

\begin{exa}
\label{exa:TMF}
(Compare \Cref{exa:TMF-periods}.) The $\EE_\infty$-ring of (`periodic') topological modular forms $\mathrm{TMF}=\Gamma(\mathcal{M}_{\mathrm{ell}},\cO_{\el})$ has period~$576$~(Hopkins, \cite{MR3590352}).
\end{exa}

\begin{exa}
A graded field~$k$ is a graded ring in which $0$ is the only graded prime ideal.
Then either $k=k_0$  or $k=k_0[t^{\pm 1}]$, where $k_0$ is a field in the ordinary sense and $t\in k_d$ for some $d>0$.
This is an easy exercise or see~\cite[Proposition~2.13]{hesselholt-pstragowski:dirac1} whose proof (in the Dirac case) goes through verbatim.
In the first case $\per(k)=0$, and in the second case $\per(k_0[t^{\pm 1}])=d$.
\end{exa}

The periods of graded rings are subject to the following basic constraint.
\begin{lem}
\label{sta:odd-period}
Let $R$ be a graded ring of odd period.
Then the commutativity constraint is $\epsilon=1$.
In particular, if $R$ is Dirac then $2=0$ in~$R$.
\end{lem}
\begin{proof}
Let $u\in R$ be a unit of odd degree.
By the graded commutativity, $u^2=\epsilon u^2$ and since $u$ is a unit we must have $1=\epsilon$ in $R$.
In the Dirac case $\epsilon=-1$ so that $2=0$.
\end{proof}

\begin{rmk}
\label{rmk:spech}
A graded ring~$R$ has an associated geometry, its \emph{graded spectrum} (or \emph{homogeneous spectrum}) of graded prime ideals~$\Spech(R)=(\Spc(R),\cO_R)$.
That is, the underlying topological space~$\Spc(R)$ has graded prime ideals as points, with the topology generated by the opens $D(f)=\{\gp\mid f\notin\gp\,\}$ for $f\in R$ homogeneous.
The structure sheaf of graded rings is essentially determined by sending~$D(f)$ to $R[\frac{1}{f}]$.

Now, assume~$R$ is periodic.
Then the association
\[
\Spch(R)\to \Spc(R_0),\qquad \gp\mapsto \gp\cap R_0,
\]
is easily seen to be a homeomorphism.
Under this identification, the structure sheaf on $\Spec(R_0)$ is (isomorphic to) the degree-0 part of the structure sheaf on~$\Spech(R)$.
\end{rmk}

\begin{defn}
\label{defn:ring-local-period}
\begin{enumerate}
\item Let $R$ be a graded ring and $\gp\in\Spech(R)$.
We define the \emph{(local) period} of~$R$ at~$\gp$ to be
\[
\per_R(\gp):=\per(R_{\gp}),
\]
the period of the local ring at~$\gp$.
\item 
For $d>0$ we also define the \emph{$d$-periodic locus} to be the following subset of $\Spch(R)$:
\begin{align*}
  \Per_d(R):=\ &\{\gp\in\Spch(R)\mid \per_R(\gp)\text{ divides } d\}\\
  =\ &\{\gp\in\Spch(R)\mid R_\gp\text{ is $d$-periodic}\}
\end{align*}
\item Finally, $\Per(R):=\cup_{d>0}\Per_d(R)$ is the \emph{periodic locus}.
\end{enumerate}
\end{defn}

\begin{rmk}
\begin{enumerate}
\item The subsets $\Per_d(R)$ (and therefore also $\Per(R)$) are open.
This is an easy exercise, or see the proof of \Cref{sta:d-periodic-spreadingout}.
\item Locally graded ringed spaces that are locally isomorphic to spectra of graded rings might be called \emph{graded schemes}.
They generalize Dirac schemes in~\cite{hesselholt-pstragowski:dirac1} (built locally out of graded spectra of Dirac rings).
It is clear that the notions in \Cref{defn:ring-local-period}, being defined locally, extend to graded schemes.
\end{enumerate}
\end{rmk}

\begin{exa}
Let $k$ be a field and consider the graded ring $R=k[x,y]$ with $\deg(x)=1$, $\deg(y)=2$.
The periodic locus $\Per(R)$ is a~$\PP_k^1$, and the complement is a single closed point corresponding to the ideal $\langle x,y\rangle$.
We have $\per_R(\langle x\rangle)=2$ while all other points in the periodic locus have period~$1$.
\end{exa}

\begin{exa}
Let $d>0$.
A \emph{weakly $d$-periodic} $\EE_\infty$-ring is an $\EE_\infty$-ring~$R$ such that the multiplication map
\[
R_n\otimes_{R_0}R_d\to R_{d+n}
\]
is an isomorphism for all $n\in\ZZ$.
(As usual, we abbreviate $R_n:=\pi_nR$.)
Such ring spectra come up naturally in stable homotopy theory, for essentially the same reason that even periodic ones do.
Of course, a $d$-periodic $\EE_\infty$-ring is also weakly $d$-periodic but the converse is not true.
(For example, because not every one-dimensional formal group over a commutative ring admits a global coordinate.)

We claim that
\[
\Per(R)=\Per_d(R)=\Spech(\pi_*R).
\]
Indeed, the condition implies that $R_d$ is an invertible $R_0$-module with inverse~$R_{-d}$.
Hence locally on~$\Spec(R_0)$, say after inverting some $f\in R_0$, it is a free module of rank~$1$ and any generator will be a unit for $R[\frac{1}{f}]$ of degree~$d$.
\end{exa}

We finish with some observations about actually computing the local periods of a graded ring.
This will be used in \Cref{sec:examples}.
\begin{prop}
\label{sta:graded-ring-local-period}
Let $R$ be a graded ring and $\gp\in\Spech(R)$.
Then
\[
\per_R(\gp)=\gcd\left(\deg(f)\ \bigm\vert\ f(\gp)\neq 0, \deg(f)\neq 0\right).
\]
(Here, we set $\gcd(\emptyset)=0$.)
\end{prop}
\begin{proof}
Straightforward.
\end{proof}

\begin{cor}
\label{sta:graded-ring-periodic-locus}
Let $R$ be a graded ring.
Then we have
\begin{equation}
\label{eq:graded-ring-periodic-locus}
\Per(R)=\bigcup_{\deg(f)\neq 0}D(f)
\end{equation}
and it's enough to let $f$ run through generators of $R/\sqrt{0}$ as an $R_0/\sqrt{0}$-algebra.
Moreover, for each homogeneous $f\in R$ with $\deg(f)\neq 0$, we have $D(f)\subseteq\Per_{\deg(f)}(R)$.\qed
\end{cor}

\section{\ldots{}and categories}
\label{sec:periodicity:oo-rings}

Let $\cK$ be a stably symmetric monoidal $\infty$-category (\Cref{sec:conventions}).
Being a commutative algebra object in $\Catex$, we may see these as a categorification of the notion of an $\EE_\infty$-ring (a commutative algebra object in~$\Sp$).
In this section and the next we categorify the notions of the previous section and make them applicable in this context.

Just as the notions in \Cref{sec:periodicity:-rings} only depended on the underlying graded ring of homotopy classes, so the notions to be introduced now only depend on the underlying homotopy category which has the structure of a tensor triangulated (or tt-)category (\Cref{sec:conventions}).

\begin{exa}
\label{exa:oo-ho-tt}
The homotopy category of any stably symmetric monoidal $\infty$-category is a tt-category.
(It is enough for it to be $\EE_3$-monoidal and for many purposes even $\EE_2$ would suffice.)
Any exact symmetric monoidal functor between such induces a tt-functor on the homotopy categories.
\end{exa}

\begin{exa}
Let $R$ be a commutative local ring with maximal non-zero ideal generated by~$2$ such that $4=0$.
(For example $R=\ZZ/4$.)
The category $\fmod^{\free}(R)$ of finitely generated free $R$-modules has an exotic triangulated structure: it is not the homotopy category of a stable $\infty$-category~\cite{MR2342636}.
One checks easily that the usual tensor product of free $R$-modules is exact so that $(\fmod^{\free}(R),\otimes_R)$ is a tt-category.
\end{exa}

\begin{exa}
Let $k\neq 0$ be a semisimple ring.
The category $\fmod_k$ of finitely generated $k$-modules admits a unique triangulated structure where the suspension is the identity functor.
The triangles are those forced by the axioms; equivalently, those that are exact in each spot.
If $k$ is also commutative (hence a product of fields) then $(\fmod_k,\otimes_k)$ is a tt-category. 

In fact, the triangulated category~$\fmod_k$ is the stable category of the Frobenius algebra $k[\epsilon]/\langle\epsilon^2\rangle$ and hence is the homotopy category of a stable $\infty$-category.\footnote{I thank Greg Stevenson for pointing this out.}
However, at least if $\car(k)\neq 2$, the tt-category $(\fmod_k,\otimes_k)$ itself is exotic: it is not the homotopy category of a stably symmetric monoidal $\infty$-category.
For example because the graded endomorphism ring of the unit is not Dirac (see \Cref{rmk:E_n,rmk:period-rings}).
\end{exa}

\begin{defn}
\label{defn:period}
Let $\cK$ be a  tt-category.
\begin{enumerate}
\item Let $d>0$ be an integer.
Then~$\cK$ is called \emph{$d$-periodic} if $\unit\cong\Sigma^d\unit$ in~$\cK$.
\item
The smallest $d>0$ such that $\cK$ is $d$-periodic is called its \emph{period} and is denoted by~$\per(\cK)$.
We also set $\per(\cK)=0$ to mean that $\cK$ is not $d$-periodic for any~$d>0$.
\item
We say that $\cK$ is \emph{periodic} if $\per(\cK)>0$, and \emph{non-periodic} otherwise.
\end{enumerate}
\end{defn}

\begin{rmk}
Let $\cK$ be a stably symmetric monoidal $\infty$-category.
We say that $\cK$ is $d$-periodic (etc.) if its homotopy category is.
The statements and arguments in this section translate immediately to this setting.
\end{rmk}

\begin{rmk}
\label{rmk:period-char}
As pointed out in the introduction, the period is the generator of the kernel of the monoid homomorphism to the Picard group
\begin{equation}
\label{eq:Z-Pic}
(\ZZ_{\geq 0},+)\to(\mathrm{Pic}(\cK),\otimes)
\end{equation}
that sends $1$ to $\Sigma\unit$, and is thereby similar in spirit to the characteristic of a ring.

\end{rmk}

\begin{rmk}
\label{rmk:period-rings}
For us, the most useful way of thinking about $d$-periodicity and periods is to directly relate them to the concepts in \Cref{sec:periodicity:-rings}. Consider the graded endomorphism ring
\[
R_{\cK}=\hom_{\cK}(\Sigma^\bullet\unit,\unit).
\]
This is a graded ring in the sense of \Cref{notn:graded-ring}, where the commutativity constraint is given by the central switch of~$\Sigma\unit$, that is, the element~$\epsilon\colon \unit\to\unit$ which corresponds to the switch of factors on~$\Sigma\unit\otimes\Sigma\unit$.

Now, a map $a\colon\Sigma^d\unit\to\unit$ in $\cK$ is an isomorphism iff the corresponding element in $(R_{\cK})_d$ is a unit.
It follows that $\per(\cK)=\per(R_\cK)$.

\end{rmk}

\begin{rmk}
\label{rmk:unitation}
Let $\cK_{\langle\unit\rangle}$ be the (small) unitation of~$\cK$ as studied in~\cite{sanders2025tensortriangulargeometryfully}.
This is the thick (automatically $\otimes$-)subcategory of~$\cK$ generated by the unit~$\unit$.
By definition, $\cK$ is $d$-periodic if and only if $\cK_{\langle\unit\rangle}$ is.

It follows that for a fully faithful embedding $\cK\hookrightarrow\cL$ one has $\per(\cK)=\per(\cL)$.
In particular, the idempotent completion $\cK\to\cK^{\natural}$ does not change the period.
\end{rmk}

The following is obvious.
\begin{lem}
\label{sta:d-periodic-basics}
Let $F\colon \cK\to\cL$ be a tt-functor.
\begin{enumerate}
\item If $d\bigm\vert d'$ and $\cK$ is $d$-periodic then $\cK$ is $d'$-periodic.
\item If $\cK$ is $d$-periodic then so is~$\cL$.
\item We have $\per(\cL)\bigm\vert\per(\cK)$.\qed{}
\end{enumerate}
\end{lem}

We also note the following restriction on possible periods.
\begin{lem}
\label{sta:period-characteristic}
Assume $\cK$ has odd period.
Then the central switch $\epsilon$ for~$\Sigma\unit$ is~$1$.
In particular, if~$\cK$ has a model then $2=0$ in~$\cK$.
\end{lem}
\begin{proof}
It is equivalent to prove this for the graded ring~$R_{\cK}$ (\Cref{rmk:period-rings}) and this is \Cref{sta:odd-period}.
\end{proof}

\begin{lem}
\label{sta:periodic-suspension}
Let $\cK$ be $d$-periodic.
Then there is a natural isomorphism between the two functors $\Sigma^d,\id\colon\cK\to\cK$.
\end{lem}
\begin{proof}
We have natural isomorphisms of functors
\[
\Sigma^d(-)\cong\Sigma^d\unit\otimes(-)\cong\unit\otimes(-)\cong\id.\qedhere{}
\]
\end{proof}

\begin{rmk}
Recall that a triangulated category~$\cK$ is called $d$-periodic if the two endofunctors $\Sigma^d,\id$ are naturally isomorphic.
So, \Cref{sta:periodic-suspension} says that the triangulated category underlying a $d$-periodic tt-category is $d$-periodic.
In particular, every object $x\in\cK$ satisfies $\Sigma^dx\cong x$.
This immediately implies the following:
\end{rmk}

\begin{prop}
\label{Prop:wt-structures-periodic}
Let $\cK\neq 0$ be a periodic triangulated category.
Then $\cK$ does not afford:
\begin{enumerate}
\item a bounded weight structure;
\item a bounded $t$-structure.\qed
\end{enumerate}
\end{prop}

\begin{exa}
\label{exa:ring-spectrum}
Let $R$ be a periodic $\EE_\infty$-ring. (That is, $\per(\pi_*R)>0$.)
Then the $\infty$-category of perfect $R$-modules $\Perf_R$ is periodic as well.
More precisely, $\per(\Perf_R)=\per(\pi_*R)$, cf.\,\Cref{rmk:period-rings}.

\end{exa}

\begin{exa}
\label{exa:periodic-derived-category}
Let $R$ be a commutative (discrete) ring and consider the $\EE_\infty$-ring $R[t_d^{\pm 1}]$ with $t_d$ in even degree~$d>0$.
(One way to construct $R[t_d^{\pm 1}]$ is as a cdga.
Another is to construct the commutative animated ring~$R[t_d]$ and invert~$t_d$.)
Then $\Mod(R[t_d^{\pm 1}])$ is $d$-periodic.
It can be identified with the \emph{$d$-periodic derived category} of~$R$~\cite{MR910167,MR1482975}.
For $d=2$ it (or rather its perfect modules) forms the natural enriching category for matrix factorizations, see~\cite[\S\,2.2]{MR3877165}.
\end{exa}

\begin{exa}
\label{exa:orbit-category}
Let $R$ again be a commutative (discrete) ring and let $\cK$ be a (stable) $R$-linear presentably symmetric monoidal $\infty$-category.
Then the base change $\cK\otimes_{\Mod_R}\Mod_{R[t_d^{\pm 1}]}$ is $d$-periodic.
It is one version of the \emph{orbit category} of~$\cK$.
\end{exa}

\section{Period stratification}
\label{sec:periodicity}

In \Cref{defn:ring-local-period} we used the spectrum of a graded ring to define local periods and the periodic locus.
Similarly, a tt-category~$\cK$ has an associated spectrum $\Spc(\cK)$~\cite{balmer:spectrum} that will allow us in this section to categorify these notions as well.
Recall that its points are prime $\otimes$-ideals of~$\cK$ and the topology is generated by the distinguished opens $U(x)=\{\cP\mid x\in\cP\}$ for $x\in\cK$.

Discussing the geometry of categorified rings (that is, tt-geometry) is more natural in the context of stably symmetric monoidal $\infty$-categories and we elaborate on this in \Cref{rmk:tt-geometry}.

\begin{notn}
\label{notn:small}
All tt-categories in this section will be assumed essentially small.
\end{notn}

\begin{defn}
\label{defn:periodic}
Let $\cK$ be an (essentially small) tt-category.
\begin{enumerate}
\item
Let $d>0$ be an integer.
A point $\cP\in\Spc(\cK)$ is called \emph{$d$-periodic} \resp{\emph{periodic}} if the associated local category $\cK/\cP$ is.
And its period is $\per_{\cK}(\cP):=\per(\cK/\cP)$.
The \emph{$d$-periodic locus} $\Per_d(\cK)\subseteq\Spc(\cK)$ is the subset of $d$-periodic points.
\item
The \emph{periodic locus} is the subset $\Per(\cK):=\cup_{d>0}\Per_d(\cK)\subseteq\Spc(\cK)$.
The tt-category $\cK$ is called \emph{locally periodic} if $\Spc(\cK)=\Per(\cK)$.
\end{enumerate}
\end{defn}

\begin{rmk}
\label{rmk:local-condition}
A locally periodic tt-category~$\cK$ need not be $d$-periodic for any~$d$.
Examples include most stable module categories, see \Cref{rmk:stmod-not-periodic}.
And even if $\cK$ is periodic and all points of $\Spc(\cK)$ share the same period~$d>0$, it does not follow that $\per(\cK)=d$.
An example is perfect modules over~$\mathrm{TMF}$, see \Cref{exa:TMF-periods}.

Nevertheless, it is true that having period at most $d$ is an open condition.
We now state this for future reference.
\end{rmk}

\begin{notn}
\label{notn:localization}
Let $\cK$ be a tt-category and let $Z\subseteq\Spc(\cK)$ be a Thomason subset (for example the complement of a quasi-compact open; in general an arbitrary union of such).
Recall that this corresponds to a tt-ideal $\cI_Z\subseteq\cK$: the subcategory of objects with support contained in~$Z$.
We denote by $\cK|_{Z^c}:=\cK/\cI_Z$ the localization, the \emph{category $\cK$ on~$Z^c$}.
\end{notn}

\begin{lem}
\label{sta:d-periodic-spreadingout}
Let $\cK$ be a tt-category and let $\cP\in\Spc(\cK)$ be a $d$-periodic point, some $d>0$.
Then there exists a quasi-compact open neighborhood $\cP\in U\subseteq\Spc(\cK)$ such that $\cK|_U$ is $d$-periodic.
\end{lem}
\begin{proof}
By assumption we have $\unit\cong\Sigma^d\unit$ in $\cK/\cP$.
This is witnessed by a fraction $\unit\xleftarrow{\alpha} x\xrightarrow{\beta} \Sigma^d\unit$ in~$\cK$ where both $\alpha$ and $\beta$ have cone contained in~$\cP$.
Then $U=\supp(\cone(\alpha))^c\cap\supp(\cone(\beta))^c$ does the job.
\end{proof}

\begin{cor}
\label{sta:d-periodic-open}
Let $\cK$ be a tt-category and $d>0$.
The subset $\Per_d(\cK)\subseteq\Spc(\cK)$ is open.
Hence also $\Per(\cK)$ is open.
\qed
\end{cor}

\begin{cor}
Let $\cK$ be a locally periodic tt-category.
There exists a non-empty quasi-compact open $U\subseteq\Spc(\cK)$ such that $\cK|_U$ is periodic.
\qed
\end{cor}

\begin{lem}
\label{sta:d-periodic-basics-geometric}
Let $F\colon \cK\to\cL$ be a tt-functor, with $f=\Spc(F)\colon\Spc(\cL)\to\Spc(\cK)$ the associated map on spectra.
\begin{enumerate}
\item If $d\bigm\vert d'$ then $\Per_d(\cK)\subseteq\Per_{d'}(\cK)$.
\item If $f(\cP)$ is $d$-periodic then so is~$\cP$. In other words,
\[
f\inv(\Per_d(\cK))\subseteq\Per_d(\cL).
\]
\item We have $\per_\cL(\cP)\bigm\vert\per_{\cK}(f(\cP))$.
\end{enumerate}
\end{lem}
\begin{proof}
Unwinding the definitions this is an immediate consequence of \Cref{sta:d-periodic-basics}.
\end{proof}

We can therefore think of $\Spc(\cK)$ as filtered by the various period loci.
We presently make this picture precise.

\begin{cns}
\label{Cons:stratification}
Consider the set $\ZZ_{\geq 0}$ with the divisibility relation.
We endow it with the Alexandrov topology.
Explicitly, a subset is open if and only if with $d\in\ZZ_{\geq 0}$ it contains all divisors of~$d$.
For $\cK$ a tt-category consider the map
\begin{equation}
\label{eq:stratification}
\per_{\cK}\colon\Spc(\cK)\to \ZZ_{\geq 0}
\end{equation}
that sends a point $\cP$ to~$\per_\cK(\cP)$.
\end{cns}

In the following statement we consider $\Spc(\cK)$ as a topological space, and also as a poset with the relation given by specialization: $\cP\sto \cQ$ iff $\cQ\in\overline{\{\cP\}}$.
\begin{prop}
\label{sta:period-map}
\begin{enumerate}
\item The map~\eqref{eq:stratification} is continuous.
In other words, $\Spc(\cK)$ is canonically stratified over $(\ZZ_{\geq 0},|)$.\footnote{The second sentence does not say more than the first one. But we like to think of $\Spc(\cK)$ as being decomposed (`stratified') into the fibers of this map (the `strata'). Below we will see that the strata are also particularly nice (\Cref{rmk:strata-locally-closed}).}
\item The map $\per_\cK\colon(\Spc(\cK),\sto)\to (\ZZ_{\geq 0},|)$ preserves the poset structure.
\end{enumerate}
\end{prop}
\begin{proof}
Let $d\in\ZZ_{\geq 0}$ and consider the open $U=\{e\mid e| d\}$.
If $d>0$ then $\per_\cK\inv(U)=\Per_d(\cK)$ which is open, by \Cref{sta:d-periodic-open}.
If $d=0$ then $U=\ZZ_{\geq 0}$ so that $\per_\cK\inv(U)=\Spc(\cK)$ is open as well.
This shows that $\per_\cK$ is continuous.

For the second statement let $\cP\sto\cQ$. This is equivalent to $\cQ\subseteq\cP\subseteq\cK$.
Thus a canonical functor $\cK/\cQ\to\cK/\cP$ and we conclude by \Cref{sta:d-periodic-basics}.
\end{proof}

\begin{defn}
The resulting stratification on~$\Spc(\cK)$ is called the \emph{period stratification}.
\end{defn}

\begin{rmk}
\label{rmk:strata-locally-closed}
The strata are locally closed subsets (the intersection of an open and a closed subset).
The stratum over~$d\in\ZZ_{\geq 0}$ is the set of points with period~$d$.
\end{rmk}

\begin{rmk}
\label{rmk:strata-nonfunctoriality}
Let $F\colon \cK\to \cL$ be a tt-functor, and denote by $f\colon\Spc(\cL)\to\Spc(\cK)$ the associated map on spectra.
It is not true in general that $f$ respects the period stratification.
Rather, the correct relation is given by \Cref{sta:d-periodic-basics-geometric}:
\[
\per_\cL(\cP)\bigm\vert\per_\cK(f(\cP))
\]
for every~$\cP\in\Spc(\cL)$.
\end{rmk}

\begin{rmk}
\label{rmk:tt-geometry}
Let $\Catperf$ denote the $\infty$-category of small idempotent-complete stable $\infty$-categories and exact functors between them, with the usual symmetric monoidal structure. Let $\cK\in\calg(\Catperf)$ be rigid, that is, every object in $\cK$ has a dual.

The association $U(x)\mapsto \cO_{\cK}(U(x)):=(\cK/\langle x\rangle)^{\natural}$ on distinguished open subsets $U(x)\subseteq\Spc(\cK)$ underlies a sheaf~$\cO_{\cK}$ with values in $\calg(\Catperf)$ which makes $(\Spc(\cK),\cO_{\cK})$ into an affine patch of tt-geometry.
This can be deduced essentially from the results in~\cite{MR2371464} but is also proven elegantly in~\cite{aoki2025higherzariskigeometry}, using Lurie's language of structured spaces~\cite{lurie:dagV}.
From this point of view, the periodicity of points is a condition on the stalks of the structure sheaf because for every point~$\cP\in\Spec(\cK)$,
\[
\cO_{\cK,\cP}=\colim_{x\in\cP}\cO_{\cK}(U(x))=\left(\cK/\cP\right)^\natural,
\]
where the colimit is computed in $\calg(\Catperf)$.
The local periods and period loci therefore extend to non-affine tt-schemes.
\end{rmk}

\begin{rmk}
\label{rmk:tt-fields}
While some of the properties of periods and the period stratification remind us of the characteristics of rings (\Cref{rmk:period-char}), we note that the analogue of a prime characteristic doesn't seem to have the same import.
For example, there is no guarantee that a tt-field, as currently understood, has period a prime number or~$0$.
(Indeed, the Morava $K$-theories don't in general.)
Moreover, odd periods are underrepresented because of \Cref{sta:period-characteristic}.
And finally, there are perfectly respectable non-zero examples that are $1$-periodic (see \Cref{exa:stmod-elab}).
\end{rmk}

\begin{rmk}
Because periodicity can be expressed solely in terms of the triangulated structure (\Cref{sta:periodic-suspension}), the considerations here could naturally be generalized to tt-categories acting on triangulated categories.
From that perspective, the discussion in~\cite[\S\,6.1]{MR3801492} implies that global complete intersections are locally $2$-periodic.
\end{rmk}

\section{Ample families}
\label{sec:comparison}

Let $\cK$ be an essentially small stably symmetric monoidal $\infty$-category.
We have in effect defined $\per(\cK)$ as $\per(R_{\cK})$, see \Cref{rmk:period-rings}.
But how do the local periods of $\cK$ relate to the local periods of~$R_\cK$?
One would hope that the canonical comparison map~\cite{balmer:sss,aoki2025higherzariskigeometry}
\[
\comp_{\cK}\colon\Spec(\cK)\to\Spech(R_{\cK})
\]
should be relevant for this question.
In general it gives the following information:

\begin{lem}
Let $\cP\in\Spec(\cK)$ and $\gp:=\comp_{\cK}(\cP)$.
Then
\[
\per_\cK(\cP)\bigm\vert\per_{R_\cK}(\gp).
\]
\end{lem}
\begin{proof}
By \Cref{sta:graded-ring-local-period}, it suffices to prove that $\per_{\cK}(\cP)\bigm\vert \deg(f)$ for every homogeneous $f\in R_{\cK}\backslash\gp$.
The latter condition means exactly that $f\colon\Sigma^{\deg(f)}\unit\to\unit$ has cone contained in~$\cP$.
The claim follows.
\end{proof}

In this section we will study situations where the comparison map (or a variant thereof) provides more precise information.
While a discussion at the level of tt-schemes would be desirable we have resisted the temptation to do so here (but see \Cref{rmk:2-graded-rings-geometry}).
In any case, in the sequel, only statements at the level of underlying spaces will be required.
Hence, from now on, $\cK$ denotes a tt-category, and throughout this section \Cref{notn:small} remains in place.

\begin{rmk}
Consider a particularly favourable case:
assume that the canonical comparison map $\comp_{\cK}\colon\Spc(\cK)\to\Spch(R_\cK)$ is a homeomorphism.
In that case, the local rings~$R_{\cK/\cP}$ are simply the stalks of $R_\cK$~\cite[Proposition~6.11]{balmer:sss}.
In fact, as we will show later (\Cref{sta:ample-local-rings}), to get this conclusion it is enough for the comparison map to be a homeomorphism onto its image (without being surjective).
Even this is in practice a restrictive assumption and to relax it we take a page from algebraic geometry (see for example~\cite{MR1970862}; of course, similar ideas are also familiar in other contexts, see for example abstract ``multigraded stable homotopy theories'' in~\cite{MR1388895}).
\end{rmk}

\begin{notn}
Let $\cK$ be a tt-category and let $x\in\cK$ be some object.
A \emph{global section} of~$x$ is a morphism $s\colon\unit\to x$ in~$\cK$.
More generally, a \emph{global section of (cohomological) degree~$d$} of~$x$ is a global section of $\Sigma^{d}x$. We denote by $U(s):=\supp(\cone(s))^c\subseteq\Spc(\cK)$ the (open) locus where~$s$ is invertible.
\end{notn}

\begin{defn}
\label{defn:ample}
A submonoid $M\subseteq\Pic(\cK)$ is called \emph{ample} if the $U(s)$ form a basis for the topology of~$\Spc(\cK)$ where $s$ runs through all global sections of all~$x\in M$.
A family of invertible objects in $\cK$ is called \emph{ample} if the monoid generated by them is.
\end{defn}

This almost recovers the notion of ampleness in algebraic geometry (\Cref{rmk:ample-AG}), where ample families of line bundles give rise to embeddings into projective varieties.
The tt-analogue is made possible in this generality by the comparison maps introduced in~\cite{MR3163513} that Dell'Ambrogio makes more widely applicable in Appendix~\ref{sec:gener-comp-maps} and which we summarize (for our purposes) in the following remark.

\begin{rmk}
\label{rmk:comparison-map-summary}
Let $M\subseteq\Pic(\cK)$ be a submonoid.
There is an associated topological space and a continuous spectral comparison map
\begin{equation}
\label{eq:comp_M}
\comp_M\colon\Spc(\cK)\to\Spch(\cR_M)
\end{equation}
where one would like to think of $\cR_M$ as the `multi-graded ring' $\oplus_{m\in M}\hom_{\cK}(\unit,m)$ of sections, and of the target of~\eqref{eq:comp_M} as its Zariski spectrum of graded prime ideals.
A moment's reflection will convince the reader that this conception is beset by coherence issues, and the solution of~\cite{MR3163513} and Appendix~\ref{sec:gener-comp-maps} is to change the meaning of a `multi-grading'.
For us, it will suffice to know that the topology on~$\Spc(\cR_M)$ is generated by opens $D(s)$ associated to global sections of elements in~$M$, and that $\comp_M\inv(D(s))=U(s)$ (\Cref{thrm:comparison-map}).
Moreover, we mention that if the coherence issues alluded to can be resolved then $\Spch(\cR_M)$ is the space underlying the usual graded spectrum (\Cref{thrm:agreement}).
For example, this is the case when $M$ is generated by a single element (\Cref{exa:original-graded}).
\end{rmk}

\begin{lem}
The map~\eqref{eq:comp_M} has dense image.
\end{lem}
\begin{proof}
If the subset $\comp_M\inv(D(s))=U(s)$ just mentioned is empty for some $s\colon \unit\to m$ with $m\in M$, this means that $s$ is $\otimes$-nilpotent in~$\cK$ so that also $D(s)=\emptyset$.
This shows that $\comp_M$ has dense image.
\end{proof}

We can now formulate the expected result that follows directly from Appendix~\ref{sec:gener-comp-maps}.
\begin{prop}
\label{sta:M-ample-homeo}
For a submonoid $M\subseteq\Pic(\cK)$ the following are equivalent:
\begin{enumerate}[(i)]
\item \label{it:ample}
$M$ is ample.
\item \label{it:homeo}
The comparison map
\[
\comp_M\colon\Spc(\cK)\to\Spch(\cR_M)
\]
is a homeomorphism onto its image.
\end{enumerate}
\end{prop}
\begin{proof}
We first show \ref{it:ample}$\Rightarrow$\ref{it:homeo}.
Let $s$ be a global section of some $m\in \cM$.
As mentioned in \Cref{rmk:comparison-map-summary}, we have $\comp_M\inv(D(s))=U(s)$.
Since $\Spc(\cK)$ is a $T_0$-space, ampleness implies injectivity.
Again by ampleness, the map~$\comp_M$ is also open onto its image.

We now show the converse \ref{it:homeo}$\Rightarrow$\ref{it:ample}.
As $s$ runs through the global sections of all~$x\in M$, the subsets $D(s)$ define a basis for the topology on $\Spch(\cR_M)$ (\Cref{rmk:comparison-map-summary}).
By assumption then, $U(s)=\comp_M\inv(D(s))$ define a basis on~$\Spc(\cK)$.
\end{proof}

\begin{cor}
The comparison map
\[
\comp_{\cK}\colon\Spc(\cK)\to\Spch(R_\cK)
\]
is a homeomorphism onto its image if and only if $\{\Sigma^d\unit, d\in\ZZ\}$ is ample.
\end{cor}

\begin{rmk}
\label{rmk:ample-AG} Let $X$ be a qcqs scheme.
A line bundle~$\cL$ is ample in the algebro-geometric sense iff the invertibility loci of all sections of all powers $\cL^{\otimes n}$ for $n>0$ form a basis for the topology of~$X$.
In \Cref{defn:ample} we have reproduced this in tt-geometry except that we allow also $n=0$ to account for the fact that the $\mathrm{Proj}$ of a graded ring is replaced by~$\Spch$.
One then proves in algebraic geometry that an ample line bundle gives rise to a quasi-compact open immersion with dense image.
The only thing \Cref{sta:M-ample-homeo} does not reproduce is that the image is open.
We do not know whether this is true in general.
\end{rmk}

It seems that we haven't gained much in practice: While adding invertibles other than shifts makes it easier for a family~$M$ to become ample, the space $\Spc(\cR_M)$ seemingly gets more mysterious at the same time.
To finish this section we discuss a  situation that offers a good compromise.

\begin{defn}
Let $\cK$ be a tt-category.
\begin{enumerate}
\item An invertible object $x\in\Pic(\cK)$ is called a \emph{line bundle} if it is Zariski locally on~$\Spc(\cK)$ (isomorphic to) a shift of~$\unit$.
(The shift is allowed to vary.)
\item A line bundle~$x\in\Pic(\cK)$ is called \emph{base-free} if it admits global sections~$s_j$ of some degree such that $\Spc(\cK)=\cup_jU(s_j)$.
\end{enumerate}
\end{defn}

\begin{rmk}
\begin{enumerate}
\item In other words, $x\in\Pic(\cK)$ is a line bundle if it is locally a `trivial line bundle'.
And it is moreover base-free if these local identifications are defined globally.
\item 
Since $\Spc(\cK)$ is quasi-compact one could equivalently ask for \emph{finitely many} such sections~$s_j$.
\end{enumerate}
\end{rmk}

\begin{cns}
\label{cns:reduction-to-graded}
Let $\{x_i\}$ be a finite family of base-free line bundles and let $M\subseteq\Pic(\cK)$ be the submonoid generated by them.
It is convenient to assume that the trivial line bundles are contained: $\Sigma^d\unit\in M$ for all~$d\in\ZZ$.
For each $i$ choose finitely many sections $s_{ij}\colon\unit\to \Sigma^{d_{ij}}x_i$ such that $\cup_jU(s_{ij})=\Spc(\cK)$ ($j\in J(i)$).
For any choice $\underline{j}=(j(i)\in J(i))_i$, the intersection $D(s_{\underline{j}}):=\cap_iD(s_{ij(i)})\subseteq\Spch(\cR_M)$ sits in a pullback square of topological spaces (\Cref{cor:central_loc_in_context}):
\[
\begin{tikzcd}
\Spc(\cK)
\ar[r,"{\comp_M}"]
&
\Spch(\cR_M)
\\
\Spc(\cK[s_{\underline{j}}\inv])
\ar[r, "{\comp_{\{\Sigma^\ZZ\unit\}}}"]
\ar[u, hook]
&
D(s_{\underline{j}})
\ar[u, hook]
\end{tikzcd}
\]
Moreover, by \Cref{exa:original-graded}, the bottom horizontal arrow identifies canonically with the usual graded comparison map.
In particular, $D(s_{\underline{j}})$ canonically identifies with the graded spectrum of a graded ring.
We deduce that $\cup_{\underline{j}}D(s_{\underline{j}})\subseteq\Spch(\cR_M)$ has a canonical structure of a graded scheme and the comparison map for $M$ factors through it.
Taking the union over all possible choices of sections yields an open $P\{x_i\}\subseteq\Spch(\cR_M)$ that admits a canonical structure of graded scheme, and a factorization
\[
\Spc(\cK)\xrightarrow{\comp_M} P\{x_i\}.
\]
(The last step is intended only to remove the dependency on the sections.
In practice, it might be more convenient to work with a single $\cup_{\underline{j}}D(s_{\underline{j}})$ as above.)
\end{cns}

\begin{defn}
A tt-category is called \emph{divisorial} if it admits a finite ample family of base-free line bundles.
\end{defn}

\begin{notn}
Recall the structure sheaf $\cO^{\gr}_K$ of graded rings on~$\Spc(\cK)$.
It is a sheaf whose stalk at a point~$\cP\in\Spc(\cK)$ is the local graded endomorphism ring~$R_{\cK/\cP}$.
We denote by $\Specgr(\cK):=\left(\Spc(\cK),\cO_K^{\gr}\right)$ the associated locally graded ringed space~\cite[\S\,6]{balmer:sss}.
\end{notn}

\begin{prop}
\label{sta:ample-local-rings-general}
Let $\cK$ be a divisorial tt-category.
\begin{enumerate}
\item Any choice of finite ample family of base-free line bundles $\{x_i\}$ identifies  $\Spc(\cK)$ (via \Cref{cns:reduction-to-graded}) as a subspace of the graded scheme~$P\{x_i\}$.
\item For each $\cP\in\Spc(\cK)$ with image $\gp\in P\{x_i\}$ the local rings identify: $R_{\cK/\cP}\cong \cO_{P\{x_i\},\gp}$.
\item In particular, if the image of $\Spc(\cK)$ in~$P\{x_i\}$ is open then $\Specgr(\cK)$ is itself a graded scheme.
\end{enumerate}
\end{prop}
\begin{proof}
We already know from \Cref{sta:M-ample-homeo} that $\comp_M$ is a homeomorphism onto its image.
Therefore the first statement follows directly from \Cref{cns:reduction-to-graded}.
Also, the third statement follows from this and the second statement.
The second statement, finally, depends only on an open neighborhood of~$\gp$ and we reduce to the case where $\{\Sigma^d\unit, d\in\ZZ\}$ is ample.
For later use we state this separately in \Cref{sta:ample-local-rings}.
\end{proof}

\begin{cor}
\label{sta:ample-local-rings}
Assume that the comparison map $\comp_{\cK}\colon\Spc(\cK)\to\Spch(R_\cK)$ is a homeomorphism onto its image.
Then for each $\cP\in\Spc(\cK)$ with image $\gp\in\Spch(R_\cK)$ we have canonically $R_{\cK/\cP}\cong (R_{\cK})_{\gp}$.
\end{cor}

\begin{proof}
In this proof we abbreviate $R:=R_\cK$.
For any homogeneous~$f\in R\backslash\gp$, by definition of the comparison map, $\cP\in\comp_{\cK}\inv(D(f))=U(f)$.
The map $\comp_{\cK}$ being a homeomorphism onto its image implies that\footnote{We denote by $\gen(\cP)$ the set of generalizations of~$\cP$, that is, the set $\{\cQ\mid \cQ\sto\cP\}$. It is the complement of a Thomason subset.}
\[
\gen(\cP)=\cap_{f\notin\gp}\comp_{\cK}\inv(D(f))
\]
so that $\cP$ is generated by $\cone(f)$ for $f$ varying as above.
We conclude that
\[
\cK/\cP=\cK/\langle\cone(f)\mid f\notin\gp\rangle
\]
is a central localization.
It follows that
\[
R_{\cK/\cP}=R[\frac{1}{f}\mid f\notin\gp]=R_{\gp}.\qedhere
\]
\end{proof}

\begin{rmk}
\label{rmk:2-graded-rings-geometry}
The target of the generalized comparison maps of~\cite{MR3163513} and Appendix~\ref{sec:gener-comp-maps} is a topological space whereas the reader might have expected something more structured.
Indeed, if $\cK$ is a stably symmetric monoidal $\infty$-category and $M\subseteq\Pic(\cK)$ a submonoid, then the space $\Spch(\cR_M)$ should be endowed with a natural structure sheaf valued in certain symmetric monoidal $\infty$-categories.
It would be nice to obtain the comparison map from a comparison of geometries as in~\cite[\S\,4.D]{aoki2025higherzariskigeometry}.
We have refrained from exploring this here.
In any case, the natural translation of~\cite{MR1970862} would arguably involve developing multi-graded spectra at the level of tt-schemes.
\end{rmk}

\section{D\'evissage}
\label{sec:transfer}

As discussed in \Cref{rmk:strata-nonfunctoriality}, tt-functors are only weakly compatible with periods.
In this section we're interested in tools that nonetheless allow the transfer of information about periods across maps of spectra.

\subsection{Zariski descent}
\label{sec:open-immersions}

An example of a map that \emph{does} preserve the period strata is a (pro-)open immersion:
\begin{prop}
\label{sta:open-immersion-period}
Let $\cK$ be a tt-category and $U\subseteq\Spc(\cK)$ the complement of a Thomason subset.
The localization $\cK\to\cK|_U$ induces on spectra the inclusion
\[
U\cong\Spc(\cK|_U)\hookrightarrow\Spc(\cK)
\]
and for each $\cP\in U$ corresponding to $\bar{\cP}\in\Spc(\cK|_U)$ we have $\per_{\cK}(\cP)=\per_{\cK|_U}(\bar{\cP})$.
\end{prop}
\begin{proof}
The canonical functor $\cK/\cP\to (\cK|_U)/\bar{\cP}$ is an equivalence so that
\[
\per_{\cK}(\cP)=\per_{\cK|_U}(\bar{\cP}).\qedhere{}
\]
\end{proof}

\begin{rmk}
\label{rmk:etale-descent}
It follows from \Cref{sta:open-immersion-period} that periods descend along Zariski covers.
It is natural to ask about (weaker) descent properties along other types of covers, like \'etale descent.
Unfortunately, we were not able to prove anything of that sort.
\end{rmk}

\subsection{Comparison map}
\label{sec:comparison-map}

Let us be given a finite ample family of base-free line bundles $\{x_i\}$ of~$\cK$.
As discussed in \Cref{sec:comparison}, this produces a homeomorphism onto a subspace of a graded scheme~$P\{x_i\}$.
The reader may keep in mind as an example just shifts of the unit in which case the graded scheme in question is simply $\Spech(R_{\cK})$.
(We state this case as a corollary just below.)
\begin{prop}
\label{sta:period-comparison-general}
Assume $\{x_i\}\subseteq\Pic(\cK)$ is a finite ample family of base-free line bundles.
For each $\cP\in\Spc(\cK)$ with image $\gp\in P\{x_i\}$, we have
\[
\per_{\cK}(\cP)=\per_{P\{x_i\}}(\gp).
\]
\end{prop}
\begin{proof}
By \Cref{sta:ample-local-rings-general}, 
\begin{equation}
\label{eq:computation-localized-ring}
R_{\cK/\cP}\cong\cO_{P\{x_i\},\gp}
\end{equation}
and hence
\begin{align*}
  \per_{\cK}(\cP)&=\per(\cK/\cP)&&\text{by definition}\\
                 &=\per(R_{\cK/\cP})&&\text{by \Cref{rmk:period-rings}}\\
                 &=\per(\cO_{P\{x_i\},\gp})&&\text{by~\eqref{eq:computation-localized-ring}}\\
  &=\per_{P\{x_i\}}(\gp) &&\text{by definition}.\qedhere
\end{align*}
\end{proof}

\begin{cor}
\label{sta:period-comparison}
Assume $\comp\colon\Spc(\cK)\to\Spch(R_{\cK})$ is a homeomorphism onto its image.
For any $\cP\in\Spc(\cK)$ with image $\gp$ we have
\[
\per_{\cK}(\cP)=\per_{R_{\cK}}(\gp).
\]
\end{cor}

The following is just a translation of \Cref{sta:period-comparison}.
\begin{cor}
\label{sta:periodicity-comparison}
Under the same assumptions we have $\comp\inv(\Per_d(R_\cK))=\Per_d(\cK)$ and $\comp\inv(\Per(R_\cK))=\Per(\cK)$.
\end{cor}

\subsection{Fully faithful functors}
\label{sec:fully-faithf-funct}

For fully faithful functors the situation is more subtle (cf.~\Cref{rmk:unitation}).
\begin{prop}
\label{sta:ff-period}
Let $F\colon\cK\hookrightarrow\cL$ be a fully faithful tt-functor with $\cK$ rigid, and associated map $f$ on spectra.
For $\cP\in\Spc(\cK)$ and $d>0$, the following are equivalent:
\begin{enumerate}[(i)]
\item the point $\cP$ is $d$-periodic,
\item the localization $\cL|_{f\inv(\gen(\cP))}$ is $d$-periodic.
\end{enumerate}
\end{prop}
\begin{proof}
The point~$\cP$ is $d$-periodic if $\cK/\cP$ is.
The induced functor
\[
\cK/\cP=\cK|_{\gen(\cP)}\to\cL|_{f\inv(\gen(\cP))}
\]
is fully faithful by \Cref{sta:ff-quotient} below hence the claim follows from \Cref{rmk:unitation}.
\end{proof}

\begin{lem}
\label{sta:ff-quotient}
Let $F\colon \cK\to \cL$ be a fully faithful tt-functor with $\cK$ rigid, let $\cI_V\subseteq\cK$ be the tt-ideal corresponding to the Thomason subset $V\subseteq \Spc(\cK)$.
Let $W=f\inv(V)\subseteq\Spc(\cL)$ corresponding to the radical tt-ideal~$\cI_W$ (where $f$ is the map on spectra induced by~$F$).
Then:
\begin{enumerate}
\item The radical tt-ideal $\cI_W$ is generated by $F(\cI_V)$.
\item The induced functor $\overline{F}\colon\cK|_{V^c}\to\cL|_{W^c}$ is fully faithful.
\end{enumerate}
\end{lem}
\begin{proof}
For the first statement, let $x\in \cI_V$, that is, $\supp(x)\subseteq V$.
Thus $\supp(Fx)=f\inv(\supp(x))\subseteq f\inv(V)=W$ and we conclude that $F(x)\in\cI_W$.
This shows one inclusion.
For the converse, let $\cQ\in W$, that is, $f(\cQ)\in V$.
We find some $x\in \cI_V$ such that $f(\cQ)\in\supp(x)$ hence $\cQ\in f\inv(f(\cQ))\subseteq f\inv(\supp(x))=\supp(Fx)$.
This concludes the proof of the first statement.

If $F\colon \cK\to\cL$ is the restriction to compact objects of a geometric functor between big tt-categories then this statement is proven in~\cite[Lemma~5.1]{MR4866349}.
In fact, only rigidity of $\cK$ is used in this proof (for the projection formula).
So, this yields the lemma in particular whenever $F$ has a model.
In general, as pointed out to me by Greg Stevenson, one can repeat the same argument, replacing the big tt-categories by the module categories, as pioneered in~\cite{MR4064108}.
\end{proof}
\subsection{Filtered colimits}
\label{sec:filtered-colimits}

We come to the last d\'evissage tool in this section.
\begin{notn}
\label{notn:limit}
Let $(\cK_i)_i$ be a filtered diagram of tt-categories with colimit~$\cK$.
It was proved in~\cite[\S\,8]{MR3892970} that the canonical functors $\cK_i\to \cK$ induces a homeomorphism
\[
\Spc(\cK)\xrightarrow{\sim}\varprojlim_i\Spc(\cK_i).
\]
We therefore identify points $\cP\in\Spc(\cK)$ canonically with compatible families of points $\cP_i\in\Spc(\cK_i)$.
\end{notn}

\begin{lem}
\label{sta:period-limit-conditions}
With the setup in \Cref{notn:limit}, let $\cP=(\cP_i)_i\in\Spc(\cK)$ be a point.
For $d> 0$, the following are equivalent:
\begin{enumerate}[(i)]
\item
\label{it:the}
the point $\cP$ is $d$-periodic,
\item
\label{it:some}
some point $\cP_i$ is $d$-periodic,
\item
\label{it:all}
the points $\cP_i$ are eventually $d$-periodic.
\end{enumerate}
\end{lem}
\begin{proof}
The implications \ref{it:some}$\Rightarrow$\ref{it:all} and \ref{it:all}$\Rightarrow$\ref{it:the} both follow from \Cref{sta:d-periodic-basics}.
For the implication \ref{it:the}$\Rightarrow$\ref{it:some}, let $\unit\xleftarrow{\alpha}x\xrightarrow{\beta}\Sigma^d\unit$ be a roof in~$\cK$ with $\cone(\alpha),\cone(\beta)\in\cP$.
There exists some~$i$ and some $x',\alpha',\beta'$ lifting $x,\alpha,\beta$ along $F\colon\cK_i\to\cK$.
By design, $\cone(\alpha')\in F\inv(\cP)=\cP_i$ and similarly for $\beta'$.
It follows that $\cP_i$ is also $d$-periodic, as claimed.
\end{proof}

\begin{cor}
\label{sta:period-limit}

Given a point $\cP=(\cP_i)_i\in\Spc(\cK)\cong\varprojlim_i\Spc(\cK_i)$ we have
\[
\per_\cK(\cP)=\varinjlim_i\per_{\cK_i}(\cP_i).
\]
(It is part of the statement that the system of non-negative integers on the right is eventually constant.)
\end{cor}
\begin{proof}
If $\per_\cK(\cP)=0$ then, by \Cref{sta:period-limit-conditions} (the implication \ref{it:some}$\Rightarrow$\ref{it:some}), $\per_{\cK_i}(\cP_i)=0$ for all~$i$.
If $\per_\cK(\cP)=d>0$ then, again by \Cref{sta:period-limit-conditions}, the points $\cP_i$ are eventually $d$-periodic.
On the other hand, there exists no~$\cP_i$ of strictly smaller period, by the implication \ref{it:some}$\Rightarrow$\ref{it:the}.
It follows that $\per_{\cK_i}(\cP_i)$ is eventually constant with value~$d$, as claimed.
\end{proof}

\begin{rmk}
\Cref{sta:period-limit} remains true under the weaker assumptions on a diagram of tt-categories imposed in~\cite[\S\,8]{MR3892970}.
The same proof applies.
\end{rmk}

\section{First examples}
\label{sec:examples}

Here we describe periodic loci and period stratifications in some examples that are either easy to determine directly, or where we can invoke the literature.
The more involved examples are left for the subsequent sections.

\begin{exa}
\label{exa:Sp}
Let $\Sp$ denote the $\infty$-category of spectra.
The spectrum of its compact part was described in~\cite[Corollary~9.5]{balmer:sss} as an application of the thick subcategory theorem~\cite{MR1652975}.
The points $\cP(p,n)$ are indexed by primes~$p$ and $n\in\ZZ_{\geq 0}\cup\{\infty\}$ with an identification $\cP(p,0)=\cP(q,0)=:\cP(0)$ for all primes $p,q$.
This prime~$\cP(0)$ is the generic point of the spectrum and is given by the finite torsion spectra~$\cP_{\text{tor}}$.
The local category at this point,
\[
\Sp^{\omega}/{\cP_{\text{tor}}}\simeq \Perf_{\Hm\QQ},
\]
is the bounded derived category of rational vector spaces, which is patently not periodic (or it follows from \Cref{Prop:wt-structures-periodic}).
By the openness of the periodic locus we conclude that the periodic locus of~$\Sp^{\omega}$ is empty.
\end{exa}

\begin{exa}
\label{exa:SpG}
Let $G$ be a finite group and consider the $\infty$-category $\Sp_G$ of (genuine) $G$-spectra.
In~\cite{balmer-sanders:SHG-finite} it was shown that all primes in~$\Sp_G^{\omega}$ are pulled back from the primes in~$\Sp^{\omega}$ via geometric fixed-point functors $\Phi^H\colon\Sp_G\to\Sp$.
It follows from \Cref{exa:Sp} and \Cref{sta:d-periodic-basics-geometric} that the periodic locus of $\Sp_G$ is empty.
The same applies with the same argument to $G$ pro-finite~\cite{balchin2024profiniteequivariantspectratensortriangular} or compact Lie~\cite{MR4036448}.
\end{exa}

\begin{exa}
An even periodic $\EE_\infty$-ring $R$ is called \emph{regular noetherian} if $\pi_0(R)$ is.
Examples include the elliptic spectra as in \Cref{exa:even-periodic}.

It is shown in~\cite[Theorem~2.13]{MR3375530} that for regular noetherian even periodic $\EE_\infty$-rings~$R$, the comparison map $\comp\colon\Spc(\Perf_R)\to\Spch(\pi_*(R))$ is a homeomorphism.
(Note that the latter space is homeomorphic to~$\Spc(\pi_0(R))$, see \Cref{rmk:spech}.)
It follows from \Cref{sta:period-comparison} that $\Spc(\Perf_R)$ has a single period stratum corresponding to~$2\in\ZZ_{\geq 0}$.
\end{exa}

\begin{exa}
\label{exa:TMF-periods}
Let $\cM_{\el}$ denote the derived moduli stack of elliptic curves (see for example~\cite{MR2597740}).
Its ring of functions is $\mathrm{TMF}$ (\Cref{exa:TMF}) and~\cite{MR3356769} show that in fact $\QCoh_{\cM_{\el}}\simeq\Mod_{\mathrm{TMF}}$ via global sections.
There is a natural support theory for perfect complexes on~$\cM_{\el}$ valued in the underlying space~$|\cM_{\el}|$ and one of the main results in~\cite{MR3375530} is that this computes the spectrum:
\[
\Spc(\Perf_{\mathrm{TMF}})=\Spc(\Perf_{\cM_{\el}})=|\cM_{\el}|
\]
By construction, the structure sheaf $\cO_{\cM_\el}$ is locally $2$-periodic. (More precisely, its sections on any affine \'etale $\Spec(R)\to\cM_{\el}$ are even periodic.)
It follows that for all $E\in|\cM_{\el}|$,
\[
\per_{\Perf_{\mathrm{TMF}}}(E)=\per_{\Perf_{\cM_{\el}}}(E)=2.
\]
Observe that while the local periods are all~$2$, the category itself is only 576-periodic (\Cref{exa:TMF}).
\end{exa}

\begin{exa}
Let $X$ be a qcqs scheme and consider $\Perf_X$, the $\infty$-category of perfect complexes on~$X$.
The spectrum is canonically identified with the space underlying~$X$, see~\cite[Theorem~8.5]{MR2280286}.
For any affine open $\Spec(A)\cong U\subseteq X$ we have $(\Perf_X)|_U\simeq\Perf_U\simeq\Perf_{\Hm A}$ (up to idempotent completion).
Since $\Hm A$ has homotopy in degree~$0$ only we deduce that $(\Perf_X)|_U$ is not periodic.
It follows (\Cref{sta:d-periodic-spreadingout}) that $\Per(\Perf_X)=\emptyset$.
\end{exa}

\begin{exa}
\label{exa:rep}
Let $G$ be a finite group and $k$ a field.
Consider the $\infty$-category of finite-dimensional representations
\[
\rep(G;k):=\Fun(BG,\Perf_k)\simeq\Db(kG),
\]
also known as the bounded derived category of f.g.\,$kG$-modules.
By~\cite{MR1450996}, the comparison map
\[
\comp\colon \Spc(\rep(G;k))\xrightarrow{\sim}\Spch(R_{\rep(G;k)})=\Spch(\Hm^*(G;k))=:\cV_G
\]
is a homeomorphism onto the \emph{extended (cohomological) support variety}~$\cV_G$.
It follows from \Cref{sta:periodicity-comparison} that $\Per(\rep(G;k))$ identifies, under $\comp$, with $\Per(\Hm^*(G;k))$.
By \Cref{sta:graded-ring-periodic-locus}, the latter is equal to
\[
\cup_{\deg(f)>0}D(f)\subseteq\Spch(\Hm^*(G;k))
\]
which contains all points except the unique closed point corresponding to the irrelevant ideal~$\Hm^{>0}(G;k)$.
Translating back along~$\comp$ we deduce that the periodic locus of~$\rep(G;k)$ is the complement of the unique closed point (the $0$-ideal).
That is, the periodic locus is nothing but the \emph{(cohomological) support variety}~$V_G=\Proj(\Hm^*(G;k))$.
\end{exa}

\begin{exa}
\label{exa:stmod}
Following up on \Cref{exa:rep}, consider the stable module category
\[
\stmod(G;k):=\frac{\Fun(BG,\Perf_k)}{\langle kG\rangle}\simeq\frac{\Db(kG)}{\langle kG\rangle},
\]
that is, the singularity category of the ring~$kG$~\cite{MR1027750}.
As the support of~$kG$ in $\Spc(\rep(G;k))$ is the unique closed point, we deduce from the previous example that also
\[
\stmod(G;k)\simeq\rep(G;k)|_{\Per(\rep(G;k))}=\rep(G;k)|_{V_G}
\]
is precisely the `periodization' of~$\rep(G;k)$.
\end{exa}

\begin{rmk}
\label{rmk:stmod-not-periodic}
In particular, $\stmod(G;k)$ is locally periodic.
However, in the spirit of \Cref{rmk:local-condition}, in most cases $k\not\cong\Sigma^dk$ for any $d>0$ in $\stmod(G;k)$.
For example, \cite[Proposition~9.16]{MR506990} proves that if $G$ is a $p$-group which isn't cyclic nor (generalized) quaternion then $\stmod(G;k)$ is not periodic.\footnote{Also, the exceptions just mentioned can be explained by the fact that in these (and only in these) cases the tt-category~$\stmod(G;k)$ is itself local (in fact, it has a unique prime). See \Cref{exa:stmod-singleton}.}
This should be compared with earlier analogous results of Artin and Tate with integral coefficients~\cite[Theorem~11.6]{MR1731415}, related to actions without fixed points on spheres~\cite[pp.~357]{MR1731415}.
\end{rmk}

\begin{rmk}
\label{rmk:stmod-strategy}
For $\rep(G;k)$ and $\stmod(G;k)$ the next thing to ask is: how does the period stratification (of \Cref{Cons:stratification}) look like?
This should exhibit interesting connections with arguments employed in the long and rich history of determining $\Pic(\stmod(G;k))$, starting with Dade's article~\cite{MR506990} mentioned above and still ongoing today, see e.g.~\cite{MR2096798,MR2183283}.

By \Cref{sta:open-immersion-period}, we may compute all local periods in $\rep(G;k)$, and by \Cref{exa:rep} together with \Cref{sta:period-comparison}, we reduce to computing the local periods of the graded ring~$\Hm^*(G;k)$.
Then we can use \Cref{sta:graded-ring-local-period} which expresses the local period at $\gp\in V_G$ as the gcd of the degrees of all non-constant cohomology classes that don't vanish at~$\gp$.
For later use and illustration we discuss a couple of examples.
\end{rmk}

\begin{exa}
\label{exa:stmod-elab}
Let $k$ be a field of characteristic $p>0$ and $G$ an elementary abelian $p$-group of rank~$r$.
Then for each $\gp\in V_G\cong\PP^{r-1}_k$ we have:
\[
\per_{\stmod(G;k)}(\gp)=
\begin{cases}
  2&:p\text{ odd}\\
  1&:p=2
\end{cases}
\]
Indeed, since the generators (modulo nilpotents) of $\Hm^*(G;k)$ sit in (cohomological) degrees~$2$ (if $p$ odd) and $1$ (if $p=2$), respectively, the inequality~$\leq$ is clear.
But when $p$ is odd we cannot have a local period of~$1$ by \Cref{sta:period-characteristic}.
Hence also~$\geq$.
\end{exa}

\begin{exa}
\label{exa:stmod-singleton}
Let $k$ be a field of characteristic~$2$ and let us consider the $2$-groups~$G$ for which the support variety $V_G$ consists of a single point.
That is, $G$ is either cyclic or (generalized) quaternion.
By looking at the degree of the generators in cohomology (modulo nilpotents) one deduces that
\[
\per(\stmod(G;k))=
\begin{cases}
  1&:G=C_2\\
  2&:G=C_{2^n}, n>1\\
  4&:G=Q_{4n}, n>1
\end{cases}
\]
\end{exa}

\begin{exa}
\label{exa:rep-D8}
We continue with stable module categories.
For a slightly more interesting example consider $G=D_8$ in characteristic~$2$.
Its cohomology ring is 
\[
\Hm^*(D_8;k)=\frac{k[\alpha_0,\alpha_1,\beta]}{\langle{\alpha_0\alpha_1}\rangle}
\]
with $\deg(\alpha_i)=1$ and $\deg(\beta)=2$.
Therefore the spectrum has two irreducible components (defined by the~$\alpha_i$), each of which is a projective line.
They meet in an $\FF_2$-rational point~$\langle \alpha_0,\alpha_1\rangle$.
We find, by \Cref{rmk:stmod-strategy}, that
\[
\per_{\stmod(D_8;k)}(\gp)=
\begin{cases}
  2&:\gp=\langle \alpha_0,\alpha_1\rangle\\
  1&:\gp\neq\langle \alpha_0,\alpha_1\rangle
\end{cases}
\]
\begin{figure}[H]
\centering
\includegraphics[scale=0.15]{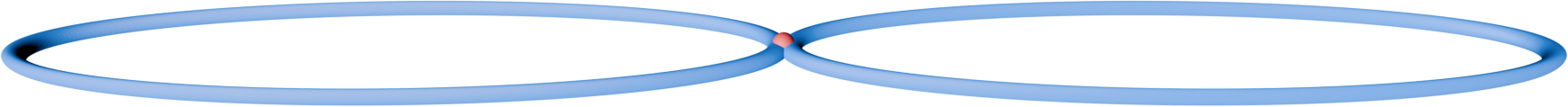}
\caption{The period stratification for $\stmod(D_8;k)$. In {\color{NavyBlue}blue} the stratum of period~$1$; in {\color{BrickRed}red} a single point with period~$2$.}
\label{fig:stmod-D8}
\end{figure}
\end{exa}

\begin{exa}
Let $G=M(11)$ be the Matthieu group of order 7920.
Its cohomology with coefficients in characteristic~$3$ is more complicated but once we quotient by the nilradical, it becomes
\[
\Hm^*(M(11);k)/\sqrt{0}=\frac{k[a,b,c]}{\langle{b^2+ac-a^3}\rangle}
\]
in which $\deg(a)=8, \deg(b)=12, \deg(c)=16$.
The associated projective scheme is a projective line~$\PP^1_k$ and its periods are~$4$ everywhere except at the (closed) point~$\langle a,b\rangle$ where the period is, naturally,~$16$.
\end{exa}

\begin{rmk}
Let $k$ be of characteristic~$p>0$.
Determining the cohomological support variety~$V_G$ as a space is made easier by Quillen's theorem~\cite{quillen:spec-cohomology} that gives $V_G=\colim_{E}V_E$ where $E\leq G$ runs through elementary abelian $p$-subgroups and conjugation in~$G$.
However, to determine the local periods one needs more precise information, namely the degrees of non-vanishing functions, and Quillen's theorem doesn't quite provide this information.
Indeed, this is in general a non-trivial task.
\end{rmk}

\section{Derived permutation modules}
\label{sec:dperm}
Throughout this section let $k$ be a field of characteristic $p>0$ and let $G$ be a finite group.
After recalling the \emph{derived category of permutation modules} $\DPerm(G;k)$, our goal here is to determine its local periods.
This we achieve partially.

\begin{rmk}
Let $\underline{k}$ denote the constant $G$-Green functor and $\Hm\underline{k}$ the associated $\EE_\infty$-algebra in $G$-spectra~$\Sp_G$ (the Bredon cohomology spectrum).
Then the derived $\infty$-category of permutation modules is
\[
\DPerm(G;k)=\Mod_{\Hm\underline{k}}(\Sp_G).
\]
Its subcategory of compact objects can be identified with bounded chain complexes of $k$-linear (finite dimensional) $G$-permutation modules, idempotent completed.
This category was studied in~\cite{MR4693637,MR4946248}.
For a homotopy theoretic treatment see \cite{fuhrmann2025modularfixedpointsequivariant}.
\end{rmk}

\begin{rmk}
Let $H\leq G$ be a $p$-subgroup and denote by $\Weyl{G}{H}=N_G(H)/H$ the Weyl group.
There is a symmetric monoidal colimit preserving \emph{modular fixed point functor}
\[
\Psi^H\colon \DPerm(G;k)\to\DPerm(\Weyl{G}{H};k)
\]
that is essentially determined by $k$-linearizing the $H$-fixed points functor on $G$-sets.
In other words, it sends a $k$-linearized $G$-set $k(X):=\Hm\underline{k}\wedge\Sigma^\infty_+X$ to $k(X^H)$.
(Such a functor can only exist for $p$-subgroups and $k$ of characteristic $p$ whence the adjective ``modular''.)
\end{rmk}

\begin{rmk}
\label{rmk:DPerm-Spc}
The main result in~\cite{MR4946248} was the computation of the spectrum of $\DPerm(G;k)^{\omega}$.
It has a stratification by locally closed subsets homeomorphic to (extended) support varieties (\Cref{exa:rep})
\[
\Spc(\DPerm(G;k)^{\omega})=\coprod_{(H)}\cV_{\Weyl{G}{H}},
\]
where $(H)$ runs through conjugacy classes of $p$-subgroups $H\leq G$.
There is one closed point~$\gm(H)$ for each conjugacy class $(H)$ of $p$-subgroups, namely the unique closed point of~$\cV_{\Weyl{G}{H}}$.

These strata are the (homeomorphic) images of the maps on spectra associated to the functors
\[
\check{\Psi}^H\colon\DPerm(G;k)^{\omega}\xrightarrow{\Psi^H}\DPerm(\Weyl{G}{H};k)^{\omega}\xrightarrow{\Upsilon}\rep(\Weyl{G}{H};k)
\]
where $\Upsilon$ sends a complex of permutation modules to itself viewed as an object in the bounded derived category of $k(\Weyl{G}{H})$-modules.
In the language of equivariant homotopy theory, this second functor is induced by Borel completion, in view of~\cite[Proposition~6.17]{mathew-naumann-noel:nilpotence-descent}.
We denote by $\cP(H,\gp)$ the point $(\check{\Psi}^H)\inv(\gp)\in\Spc(\DPerm(G;k)^{\omega})$ for $\gp\in\cV_{\Weyl{G}{H}}$.
\end{rmk}

\begin{rmk}
\label{rmk:conjecture}
In the original version of this article the following result (\Cref{sta:dperm-periodic-locus}) was proved for $p$-groups and conjectured to hold for arbitrary finite groups.
In \Cref{sec:peri-locus-perm} in collaboration with Sam Miller, we now establish this conjecture.
\end{rmk}

\begin{thrm}
\label{sta:dperm-periodic-locus}
The periodic locus of $\DPerm(G;k)^{\omega}$ is the complement of the closed points.
In other words,
\[
\Per(\DPerm(G;k)^{\omega})=\coprod_{(H)}V_{\Weyl{G}{H}}\ \subseteq\ \coprod_{(H)}\cV_{\Weyl{G}{H}}.
\]
\end{thrm}

At least for Dedekind groups (i.e.\ all subgroups are normal) we have a more precise result.
\begin{prop}
\label{sta:DPerm-period-normal}
Let $H\leq G$ be a \emph{normal} $p$-subgroup.
Then for every $\gp\in\cV_{\Weyl{G}{H}}$,
\begin{equation}
\label{eq:DPerm-period-normal}
\per_{\DPerm(G;k)^\omega}\left(\cP(H,\gp)\right)=\per_{\rep(\Weyl{G}{H};k)}(\gp).
\end{equation}
\end{prop}

\begin{cor}
\label{sta:dperm-d-periodic}
Let $d>0$ and $G$ a Dedekind group.
Then
\begin{equation}
\label{eq:dperm-d-periodic}
\pushQED{\qed} 
\Per_d(\DPerm(G;k)^{\omega})=\coprod_{(H)}\Per_d(\rep(\Weyl{G}{H};k))\ \subseteq\ \coprod_{(H)}\cV_{\Weyl{G}{H}}.\qedhere
\popQED
\end{equation}
\end{cor}

\begin{rmk}
Recall from \Cref{exa:rep,rmk:stmod-strategy} that we also have
\[
\Per_d(\rep(\Weyl{G}{H};k)) =\Per_d(\stmod(\Weyl{G}{H};k)) \subseteq V_{\Weyl{G}{H}}\subseteq \cV_{\Weyl{G}{H}}
\]
in the expression~\eqref{eq:dperm-d-periodic}.
Of course, $\Weyl{G}{H}=G/H$ since $H$ is normal but we prefer to state it in this way because~\eqref{eq:dperm-d-periodic} might well hold for arbitrary groups~$G$.
(As in \Cref{exa:DPerm-D8} below.)
\end{rmk}

\begin{proof}[Proof of \Cref{sta:DPerm-period-normal}]
Our goal is to show the equality (abbreviating $\check{\psi}^H=(\check{\Psi}^H)\inv$)
\begin{equation}
\label{eq:goal}
\per_{\rep(G/H;k)}(\gp)=\per_{\DPerm(G;k)^{\omega}}(\check{\psi}^H(\gp)).
\end{equation}
By \Cref{sta:d-periodic-basics-geometric} we have $\leq$.
For the converse we note that $\check{\psi}^H=\psi^H\circ\upsilon$ induced by the factorization as in \Cref{rmk:DPerm-Spc}.
But $\upsilon$ is an open immersion~\cite[Theorem~5.13]{MR4541331} so that (\Cref{sta:open-immersion-period})
\[
\per_{\rep(G/H;k)}(\gp)=\per_{\DPerm(G/H;k)^{\omega}}(\upsilon(\gp)).
\]
Finally we use that $\Psi^H\colon\DPerm(G;k)\to\DPerm(G/H;k)$ has a section given by inflation~$\Infl^{G/H}_G$~\cite[Corollary~5.16]{MR4946248}.
We deduce the other inequality $\geq$ in~\eqref{eq:goal}, again by \Cref{sta:d-periodic-basics-geometric}.
\end{proof}

\begin{exa}
Let $p=2$ and $G=Q_8$ the quaternion group.
This is a Dedekind group and \Cref{sta:dperm-d-periodic} applies.
The Weyl groups are $1, C_2, C_2^{\times 2}$ and $Q_8$.
We computed all the periods for the stable module category of these groups in \Cref{exa:stmod-elab,exa:stmod-singleton} and deduce that the period stratification (on the periodic locus) looks as in \Cref{fig:DPerm-Q8}, cf.\ also~\cite[Example~8.12]{MR4946248}.

\begin{figure}[H]
\centering
\includegraphics[scale=0.18]{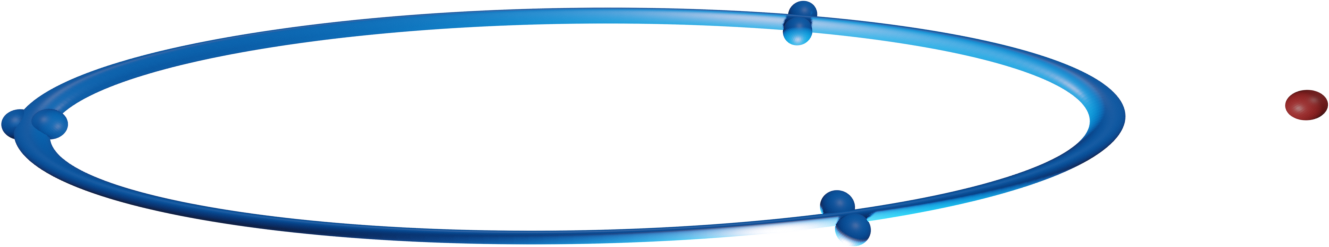}
\caption{The periodic locus for $\DPerm(Q_8;k)^{\omega}$. In {\color{NavyBlue}blue} the stratum of period~$1$: a $\PP^1_k$ with the $\ZZ/2$-rational points doubled; in {\color{BrickRed}red} a single point with period~$4$.}
\label{fig:DPerm-Q8}
\end{figure}
\end{exa}

\begin{exa}
\label{exa:DPerm-D8}
Let $p=2$ and $G=D_8$ the dihedral group.
We know by \Cref{sta:dperm-periodic-locus} that its periodic locus is the complement of the closed points.
And we may compute the local periods at $\cP(H,\gp)$ for all normal subgroups~$H\leq D_8$ via \Cref{sta:DPerm-period-normal}.
The associated Weyl groups are $1, C_2, C_2^{\times 2}$ and $D_8$ and we computed the relevant periods in \Cref{exa:stmod-elab,exa:rep-D8}.
There are two conjugacy classes of non-normal subgroups, each with Weyl group~$C_2$ and cohomological support variety a single point each.
We compute below that the local periods at these two points is~$1$ so that the period stratification (on the periodic locus) looks as in \Cref{fig:DPerm-D8}.
For more detailed explanations of how this space comes about we refer to~\cite[Example~18.17]{MR4946248}.
\begin{figure}[H]
\centering
\includegraphics[scale=0.11]{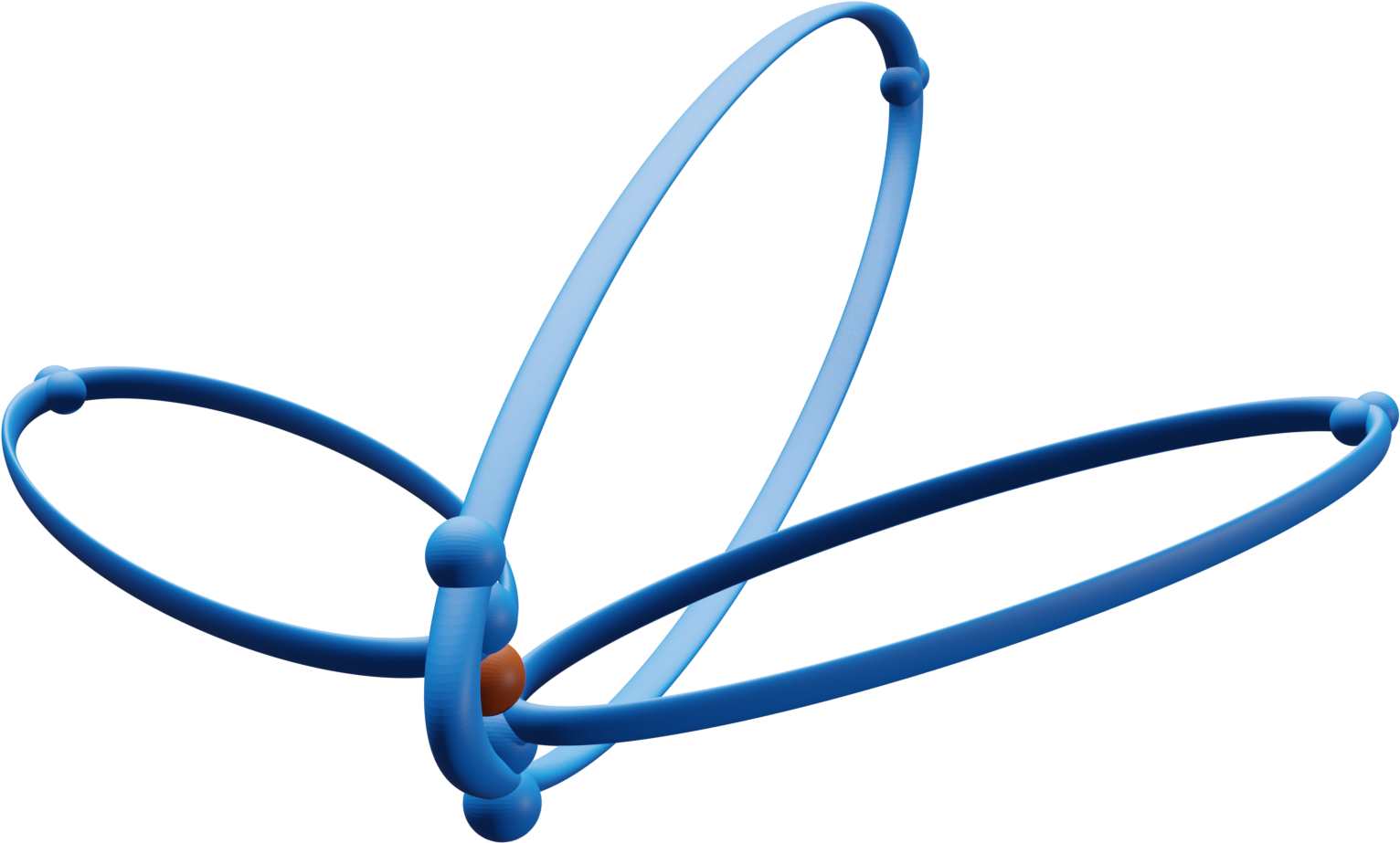}
\caption{The periodic locus for $\DPerm(D_8;k)^{\omega}$. {\color{NavyBlue}Blue}: period~$1$, {\color{BrickRed}red}: period~$2$.}
\label{fig:DPerm-D8}
\end{figure}

It remains to compute the two local periods alluded to.
Let $H\cong C_2$ be one of the non-normal subgroups with normalizer $V\cong C_2^{\times 2}$.
The corresponding stratum $V_{V/H}$ is a single point $\cP$ which we think of as one of the two doubled points at the far right in \Cref{fig:DPerm-D8}.
Recall (\Cref{exa:stmod-elab}) that the formula~\eqref{eq:DPerm-period-normal} (for normal subgroups) already suggests the period should be~$1$.
Or more geometrically, locally around the point, \Cref{fig:DPerm-D8} looks like an open in the stable module category for $V$ and so we would expect, again, the local period to be~$1$.\footnote{The fact that this by itself is \emph{not} a proof is due to the subtlety of these doubled points: It is not~$\cP$ but its double that belongs to the open subset formed by the stable module category. But once we zoom in enough and in particular remove its double, $\cP$ seems to `take the place' of its double.}

Consider now the fraction of canonical morphisms in $\DPerm(D_8;k)$ (with canonical augmentation differential):
\[
\begin{tikzcd}
k
\ar[r]
&
k(D_8/H)
\ar[d]
\\
&
k
&
k
\ar[l]
\end{tikzcd}
\]
The cone of the wrong-way morphism is (a shift of) $k(D_8/H)$ which becomes projective under $\Psi^{H}$ and therefore belongs to~$\cP$.
The cone of the first morphism becomes acyclic under $\Psi^H$ so again belongs to~$\cP$.
In summary, the fraction defines an explicit isomorphism $\unit\cong\Sigma\unit$ in $\DPerm(D_8;k)^{\omega}/\cP$.
Of course, by symmetry, the local period must be~$1$ at the other remaining point as well.
(There is a group automorphism of~$D_8$ that exchanges the two points.)
\end{exa}

\section{Motives}
\label{sec:motives}

In the last two sections we study periods in categories of motivic origin, at least at a few points that have been found so far.
In this section we discuss general features while in \Cref{sec:isotropic-points} we focus on a single but particularly large class of points.

\begin{notn}
Let $\FF$ and $k$ be fields.
\begin{enumerate}
\item We denote by $\SH(\FF)$ Morel-Voevodsky's $\infty$-category of $\Aff^1$-motivic spectra over~$\FF$~\cite{MR1813224,MR3281141}.
\item We denote by $\DM(\FF;k)$ Voevodsky's (big) $\infty$-category of motives over~$\FF$ with coefficients in~$k$.
Its subcategory of compact objects $\DM(\FF;k)^{\omega}$ identifies with ``geometric motives'', see~\cite{Voevodsky00}.
At least if $\car(\FF)=0$ or $\car(\FF)\in k^\times$, the category $\DM(\FF;k)$ can be described as modules (in~$\SH(\FF)$) over a motivic ring spectrum, bringing the setup more in line with the one in topology~\cite{MR2435654,MR3730515} and~\cite[Proposition~3.13]{MR3404379}.
\end{enumerate}
\end{notn}

\begin{exa}
\label{exa:SH-Betti}
Assume $\sigma\colon\FF\hookrightarrow\CC$ and consider the Betti realization $\Re^\sigma_B\colon\SH(\FF)\to\Sp$ that sends a smooth $\FF$-scheme to its $\CC$-points with the analytic topology~\cite[\S\,3.3.2]{MR1813224}, \cite[Theorem~1.4]{MR2045835}, \cite[\S\,A.7]{MR2597741}.
The induced map $\rho_B^\sigma\colon\Spc(\Sp^{\omega})\to\Spc(\SH(\FF)^\omega)$ is a homeomorphism onto its image and the image consists of \emph{Betti points}, by definition.
It follows from \Cref{exa:Sp,sta:d-periodic-basics-geometric} that these all have period~$0$.

Similarly, if $\alpha\colon\FF\hookrightarrow\RR$ there is a real Betti realization $\Re^\alpha_B\colon\SH(\FF)\to\Sp$ (\cite[\S\,3.3.2]{MR1813224}, \cite[Proposition~4.8]{MR3546793}).
The corresponding Betti points in $\Spc(\SH(\FF)^{\omega})$ are again of period~$0$ because of \Cref{exa:Sp}.
(There is also a $C_2$-equivariant Betti realization $\Re^\alpha_{B}\colon\SH(\FF)\to\Sp_{C_2}$~\cite[Proposition~4.8]{MR3546793}.
However, postcomposing with geometric fixed points $\Phi^{1}$ \resp{$\Phi^{C_2}$} yields the complex \resp{real} Betti realization so that this does not give rise to additional points, by \Cref{exa:SpG}.)
\end{exa}

\begin{exa}
\label{exa:DM-rational}
Assume $k$ is of characteristic zero.
Under standard motivic conjectures, there is a canonical equivalence $\DM(\FF;k)^{\omega}\simeq\Db(\mathcal{A})$ where $\mathcal{A}$ is a $k$-linear Tannakian category.
(Also, see~\cite[Theorem~4.22]{MR4971596} for a weakening of the conjectures when~$\FF$ is of characteristic zero.)
It was explained in~\cite{MR4224517} that this implies the spectrum has a single point, given by the $0$-ideal.
By \Cref{Prop:wt-structures-periodic}, we see that the period at this single point is~$0$.
\end{exa}

\begin{rmk}
It is because of \Cref{exa:DM-rational} that in the sequel we will restrict in $\DM(\FF;k)$ to coefficients~$k$ of positive characteristic.
From the tt-geometric point of view that's where the fun happens.
\end{rmk}

\begin{exa}
\label{exa:dmet}
Let $\ell$ be a prime invertible in~$\FF$ and consider $\DMet(\FF;\ZZ/\ell)$, the \'etale version of Voevodsky's category, and its full subcategory $\DMet(\FF;\ZZ/\ell)^\vee$ of rigid objects.
It was shown in~\cite[Proposition~4.4]{MR4224517} that the comparison map with respect to the Tate motive $\tilde{\mathbb{G}}_m=\unit(1)[1]$ is a homeomorphism
\[
\comp_{\tilde{\mathbb{G}}_m}\colon\Spc(\DMet(\FF;\ZZ/\ell)^\vee)\xrightarrow{\sim}\Spch(K^M(\FF)/\ell),
\]
using Suslin-Voevodsky rigidity and the solution to the Bloch--Kato conjecture.
The spectrum of mod-$\ell$ Milnor K-theory is known to have a unique closed point~$\langle[\FF^\times]\rangle$ and only additional points $\langle[P_\alpha]\rangle$ when $\ell=2$ and $\FF$ admits orderings~$\alpha$.
(Here, $P_\alpha$ denotes the positive cone of $\FF$ with respect to the ordering~$\alpha$.)
We refer to~\cite[Remark~4.2]{MR4224517} for a description of the space.
\end{exa}

\begin{exa}
\label{exa:etale-realization}
Continuing with the notation of \Cref{exa:dmet}, a point in the image of
\[
\upsilon_{\et}\colon \Spch(K^M(\FF)/\ell)\cong\Spc(\DMet(\FF;\ZZ/\ell)^\vee)\to\Spc(\DM(\FF;\ZZ/\ell)^{\omega})
\]
induced by \'etale sheafification is called an \emph{\'etale point}.
(This is sometimes just an open immersion, see~\cite{MR2182375}, slighly improved in~\cite[Theorem~C.1]{MR4010428}.)
We denote these prime ideals by $\cP_{\et}(\ell)\in\Spc(\DM(\FF;\ZZ/\ell)^{\omega})$ and $\cP_{\et}(2, \alpha)\in\Spc(\DM(\FF;\ZZ/2)^{\omega})$, respectively.

We now claim that
\begin{align*}
\per(\cP_{\et}(\ell))&=0,\\ \per(\cP_{\et}(2, \alpha))&=1.
\end{align*}
Compare with~\cite[Remark~4.6]{MR4224517} for a more useful description of these prime ideals which we'll use in the sequel.
The prime $\cP_{\et}(\ell)$ is the kernel of the \'etale realization functor (on compact objects)
\[
\DM(\FF;\ZZ/\ell)\to \DMet(\FF;\ZZ/\ell)\to \DMet(\bar{\FF};\ZZ/\ell)\simeq \Mod_{\Hm\ZZ/\ell}
\]
and hence its period is~$0$, by \Cref{sta:d-periodic-basics}.

For the remaining points, assume $\ell=2$ and fix an ordering~$\alpha$ of~$\FF$ with positive cone $P_\alpha\subseteq\FF^\times$ and with completion $\FF_\alpha$.
The \'etale sheafification $\DM(\FF;\ZZ/2)\to\DMet(\FF;\ZZ/2)$ inverts the motivic Bott element $\tau\colon \unit\to\unit(1)$ determined by $-1$ (a primitive square root of 1).
And if $\rho\colon\unit\to\unit(1)[1]$ corresponds to the point $-1\in\mathbb{G}_m$ then its image does not belong to $P_\alpha$ so $\rho$ is invertible at $\cP_{\et}(2,\alpha)\in\Spc(\DM(\FF;\ZZ/2)^{\omega})$.
Together we see that $\per\left(\cP_{\et}(2,\alpha)\right)=1$.
\end{exa}

\begin{rmk}
One can similarly define \'etale points in $\Spc(\SH(\FF)^\omega)$ as those that arise through \'etale sheafification.
We leave their study to a future work.
\end{rmk}

\begin{exa}
\label{exa:DAM}
Let $k$ be a field of characteristic~$p>0$ and consider Artin motives $\DAM(\FF;k)$, the localizing subcategory of Voevodsky motives generated by $\Spec(\FF')$ where $\FF'/\FF$ is finite and separable.
It is well-known (\cite[Proposition~3.4.1]{Voevodsky00}, or \cite[Corollary~710]{MR4693637} without the perfectness assumption) that
\[
\DAM(\FF;k)\simeq\DPerm(G_{\FF};k)
\]
where $G_{\FF}$ is the absolute Galois group of~$\FF$.
One has
\[
\DPerm(G_{\FF};k)=\colim_N\DPerm(G_\FF/N;k)
\]
where $N\leq G_{\FF}$ runs through open normal subgroups (so that $G_{\FF}/N$ is a finite group) and the colimit is taken in stable presentably symmetric monoidal $\infty$-categories.
Using such a (co)limit argument, it was shown in~\cite[Corollary~3.10]{MR4866349} that the stratification of \Cref{rmk:DPerm-Spc} also holds in this more general situation:
\begin{equation}
\label{eq:DAM-strata}
\Spc(\DPerm(G_{\FF};k)^{\omega})=\coprod_{(H)}\cV_{\Weyl{G_{\FF}}{H}},
\end{equation}
where $H$ runs through conjugacy classes of (closed)\footnote{Here and in the sequel all subgroups of profinite groups are implicitly assumed closed.} pro-$p$-subgroups.
We denote the points by $\cP(H,\gp)$ for $\gp\in\cV_{\Weyl{G_{\FF}}{H}}$.
Using \Cref{sta:period-limit} we have
\begin{equation}
\label{eq:period-DAM}
\per_{\DAM(\FF;k)^{\omega}}\left(\cP(H,\gp)\right)=\colim_N\per_{\DPerm(G_\FF/N:k)^{\omega}}\left(\cP(\bar{H},\bar{\gp})\right)
\end{equation}
where $\bar{H}=HN/N$ and $\bar{\gp}\in\cV_{\Weyl{G_{\FF}}{N}}$ the image of~$\gp$.
From this we deduce:
\begin{itemize}
\item For every closed point~$\gm(H)=\cP(H,\Hm^{>0})$ associated to a pro-$p$-subgroup~$H\leq G_{\FF}$ one has $\per(\gm(H))=0$;
\item
By \Cref{sta:dperm-periodic-locus}, these are the only non-periodic points.
\end{itemize}
\end{exa}

\begin{exa}
If $G_{\FF}$ is abelian (or pro-Dedekind) then we deduce from \Cref{sta:dperm-d-periodic} and~\eqref{eq:period-DAM} that
\begin{equation}
\label{eq:period-DAM-abelian}
\per_{\DAM(\FF;k)^{\omega}}\left(\cP(H,\gp)\right)=\colim_N\per_{\rep(\Weyl{(G_{\FF}/N)}{\bar{H}};k)}(\bar{\gp})=\per_{\rep(\Weyl{G_\FF}{H};k)}(\gp).
\end{equation}
For example, for $\FF=\FF_q$ a finite field (possibly $p\bigm\vert q$) we obtain the following period stratification (cf.\ \cite[Theorem~4.6]{MR4866349}):
\begin{align*}
\begin{tikzpicture}[inner sep=0pt, outer sep=0pt]
\node[main node] at (-6,2) (M0) {};
\node[split node] at (-5.5,1.4) (P1) {};
\node[main node] at (-5,2) (M1) {};
\node[lower node] at (-4.5,1.4) (P2) {};
\node[main node] at (-4,2) (M2) {};
\node[lower node] at (-3.5,1.4) (P3) {};
\node[main node] at (-3,2) (M3) {};
\node at (-2.5,1.4) (P4) {};
\node at (-2,1.6) {\ldots};
\node[main node] at (-1,2) (MI) {};
\path[line width=0.01pt]
(M0) edge (P1)
(P1) edge (M1)
(M1) edge (P2)
(P2) edge (M2)
(M2) edge (P3)
(P3) edge (M3)
(M3) edge (P4);
\end{tikzpicture}
\end{align*}
Here, the color $\begin{tikzpicture}
\node[main node] {};
\end{tikzpicture}$ (resp.\ $
\begin{tikzpicture}
\node[lower node] {};
\end{tikzpicture}$) indicates period~$0$ (resp.\ $2$), while the single point colored $\begin{tikzpicture}
\node[split node] {};
\end{tikzpicture}$ indicates period~$1$ if $p=2$ \resp{$2$ if $p>2$}.

To verify this note that the absolute Galois group identifies with the profinite integers~$\hat{\ZZ}=A\times \hat{\ZZ}_p$ where $A$ is the prime-to-$p$ part.
Because of this decomposition one may replace $\hat{\ZZ}$ by $\hat{\ZZ}_p$ in~\eqref{eq:period-DAM-abelian}.
Then all the finite quotients are cylic groups $C_{p^n}$ whose generators in cohomology (modulo nilpotents) live in degree~$1$ \resp{$2$} if $p^n=2$ \resp{$p^n>2$}.
\end{exa}

\begin{exa}
We now restrict to the localizing subcategory $\DATM(\RR;\ZZ/2)$ of Voevodsky motives over the real numbers (or any real-closed field), generated by the Artin-Tate motives $\Spec(K)(n)$ where $K\in\{\RR,\mathbb{C}\}$ and $n\in \mathbb{Z}$.
Its spectrum was computed in~\cite{MR4445121} as the following six-point space with specialization relations going upwards:
\begin{align}
\label{eq:ratm}
\Spc(\DATM(\RR;\ZZ/2)^{\omega})\quad =\qquad
\vcenter{
\hbox{\begin{tikzpicture}[inner sep=0pt, outer sep=0pt]
\node[lower node] at (-5,1.4) (S) {};
\node[main node] at (-4,2) (M) {};
\node[main node] at (-6,2) (W) {};
\node[main node] at (-6,1) (R) {};
\node[lower node] at (-5,0.4) (A) {};
\node[main node] at (-4,1) (E) {};
\path[line width=0.01pt]
(S) edge (M)
(S) edge (W)
(R) edge (W)
(E) edge (M)
(A) edge (S)
(A) edge (R)
(A) edge (E);
\end{tikzpicture}}
}
\end{align}
We indicate the period stratification in colors, with $
\begin{tikzpicture}
\node[main node] {};
\end{tikzpicture}$ (resp.\ $
\begin{tikzpicture}
\node[lower node] {};
\end{tikzpicture}$) having period~$0$ (resp.~$1$).

The top $V$-shaped layer is detected by a functor $\DATM(\RR;\ZZ/2)\to \DAM(\RR;\ZZ/2)$ hence we know by \Cref{exa:DAM} and \Cref{sta:d-periodic-basics-geometric} that the two closed points have period~$0$.
On the other hand, the middle point is generated by $\Spec(\mathbb{C})$ and the complex $\unit\to\Spec(\mathbb{C})\to\unit$, see~\cite[Theorem~7.9]{MR4445121}.
This immediately implies that this point has period~$1$.

A surprising fact~\cite[Theorem~6.18]{MR4445121} is that the category over the bottom $V$-shaped layer (cf.\ \Cref{notn:localization}) is another copy of $\DAM(\RR;\ZZ/2)$.
We conclude with \Cref{sta:open-immersion-period} that the periods in that bottom layer coincide with those of the latter category.
Because of \Cref{sta:DPerm-period-normal,exa:stmod-elab} we see that the pattern from the top layer repeats.
\end{exa}

\begin{rmk}
\label{rmk:etale-vs-artin-tate}
Under the canonical inclusion $\DAM(\FF;\ZZ/\ell)\hookrightarrow\DM(\FF;\ZZ/\ell)$ the \'etale points restrict precisely to the `cohomological' stratum in~\eqref{eq:DAM-strata}, that is, the following square commutes:
\[
\begin{tikzcd}
\Spc(\DM(\FF;\ZZ/\ell)^{\omega})
\ar[d, "\pi"]
&
\Spc(\DMet(\FF;\ZZ/\ell)^\vee)
\ar[d, "\sim"]
\ar[l, hook']
\\
\Spc(\DAM(\FF;\ZZ/\ell)^{\omega})
\ar[d, "="]
&
\Spc(\rep(G_\FF;\ZZ/\ell))
\ar[d, "="]
\ar[l, hook']
\\
\coprod_{(H)}\cV_{\Weyl{G_{\FF}}{H}}&
\cV_{\Weyl{G_\FF}{1}}
\ar[l, hook']
\end{tikzcd}
\]
This follows from~\cite[Corollary~7.15]{MR4693637}.

Similarly, under the canonical inclusion $\DATM(\RR;\ZZ/2)\hookrightarrow\DM(\RR;\ZZ/2)$, the two \'etale points are sent to the bottom middle and bottom right (or left, by symmetry) point in~\eqref{eq:ratm}, see~\cite[Equation~(10.8)]{MR4445121}.
\end{rmk}
\section{Isotropic points}
\label{sec:isotropic-points}

In this section we recall the construction of isotropic points due to Du-Vishik~\cite{MR4905541} and Vishik~\cite{MR4768634}, respectively.
We'll see that (essentially by construction) their periods are all~$0$ and, using this, we relate them to the other points considered in the preceding section.

\begin{notn}
We fix a field~$\FF$ of characteristic zero and a prime~$p$.
\end{notn}

\begin{exa}
\label{exa:isotropic-points}
Let us recall the \emph{isotropic points} in $\Spc(\DM(\FF;\ZZ/p)^{\omega})$ from~\cite{MR4768634}.
For this we start with a field extension $\EE/\FF$.
The associated isotropic prime $\iso{\EE}$ is the kernel of the composite (on compact objects)
\[
\DM(\FF;\ZZ/p)\xrightarrow{\otimes_{\FF}\EE(\PP^\infty)}\DM(\EE(\PP^\infty);\ZZ/p)\to\DM(\EE(\PP^\infty)/\EE(\PP^\infty);\ZZ/p)
\]
where the second functor is a localization, killing all $p$-anisotropic varieties.\footnote{Recall that a (locally of) finite type $\FF$-scheme is $p$-anisotropic if all its closed points are of $p$-power degree over~$\FF$.}
(The difficult part is showing that the kernel is prime.)

We now observe that the local isotropic category~$\DM(\EE(\PP^\infty)/\EE(\PP^\infty);\ZZ/p)$ admits a weight structure \cite[Proposition~5.7]{MR4768634} so that $\per_{\DM(\FF;\ZZ/p)^{\omega}}(\iso{\EE})=0$, by \Cref{Prop:wt-structures-periodic}.
It is also known that the isotropic points are all closed and hence there are no specialization relations among them~\cite{vishik2025balmerspectrumvoevodskymotives}.
\end{exa}

\begin{rmk}
\label{rmk:isotropic-equivalence-relation}
By~\cite[Theorem~1.3]{MR4768634}, the equality $\iso{\EE}=\iso{\EE'}$ holds iff $\EE\psim\EE'$ where the latter means equivalence with respect to the following partial order.
If $\EE=\FF(X)$ and $\EE'=\FF(X')$ for smooth projective varieties $X,X'/\FF$ then $\EE\pgeq\EE'$ iff there is a correspondence $X\dashrightarrow X'$ of degree prime-to-$p$.
(Equivalently, the pushforward on Chow groups modulo~$p$ induced by $X\times_{\FF}X'\to X$ is surjective.)
For arbitrary extensions write $\EE=\colim \EE_\alpha$ and $\EE'=\colim\EE'_\beta$ as filtered unions with $\EE_\alpha$, $\EE'_\beta$ finitely generated.
Then $\EE\pgeq \EE'$ if for every~$\beta$ there exists~$\alpha$ such that $\EE_\alpha\pgeq\EE'_\beta$.
\end{rmk}

\begin{exa}
Let us recall the \emph{isotropic points} in $\Spc(\SH(\FF)^{\omega})$ from~\cite{MR4905541}.
Let $\EE/\FF$ be a field extension.
The isotropic point~$\iso{\EE,p,n}$ for $p$ a prime and $n\in\ZZ_{\geq 0}\cup\{\infty\}$ is the kernel of the composite (on compact objects)
\[
\SH(\FF)\xrightarrow{\otimes_{\FF}\EE(\PP^\infty)}\SH(\EE(\PP^\infty))\to\SH(\EE(\PP^\infty)/\EE(\PP^\infty))
\]
where the second functor is a certain finite localization related to Morava $K$-theory at height~$n$ and at the prime~$p$.
We won't need to know the exact details but only note that by~\cite[\S\,7]{MR4905541}, the compact local isotropic category $\SH(\EE(\PP^\infty)/\EE(\PP^\infty))^{\omega}$ maps (conservatively) to an $\infty$-category with weight structure (specifically, a homotopy category of numerical Morava Chow motives).
By \Cref{sta:d-periodic-basics,Prop:wt-structures-periodic}, we deduce that $\per(\iso{\EE,p,n})=0$. 
\end{exa}

\begin{rmk}
Similarly to \Cref{rmk:isotropic-equivalence-relation} one can characterize when two isotropic points in~$\SH(\FF)$ coincide~\cite[Theorem~7.6]{MR4905541}.
In both $\SH(\FF)$ and $\DM(\FF;\ZZ/p)$, there is typically a large number of isotropic points.
We refer to~\cite[Example~8.9]{MR4905541} and \Cref{rmk:DM-many-isotropic-points} below.

For $n=\infty$, the isotropic point $\iso{\EE,p,\infty}$ is the preimage of $\iso{\EE}$ under motivic cohomology
\[
\SH(\FF)\xrightarrow{\Hm\ZZ/p}\DM(\FF;\ZZ/p),
\]
see~\cite[Theorem~7.6]{MR4905541}.

Finally, we observe that isotropic points are distinct from the \'etale and Betti points introduced in the preceding section.
Indeed, while $\tau$ is an isomorphism at the latter two, it vanishes at isotropic points, cf.~\cite[Remark~8.10]{MR4905541}.\footnote{Here, $\tau$ is as in \Cref{exa:etale-realization} determined by a primitive $p$-th root of unity if that exists. Otherwise it is defined after adjoining such, see also the proof of~\cite[Proposition~4.4]{MR4224517}.}
\end{rmk}

\begin{rmk}
Consider the inclusion $\DAM(\FF;\ZZ/p)\hookrightarrow\DM(\FF;\ZZ/p)$ and the induced map on spectra:
\begin{equation}
\label{eq:pi}
\pi\colon\Spc(\DM(\FF;\ZZ/p)^{\omega})\to\Spc(\DAM(\FF;\ZZ/p)^{\omega})
\end{equation}
Since we know all the points of the target (\Cref{exa:DAM}) we should be able to tell where the isotropic points map to.
We can do this just looking at their periods:
\end{rmk}

\begin{prop}
\label{sta:isotropic-closed-image}
The map~$\pi$ sends all isotropic points to closed points.
\end{prop}

\begin{proof}
By \Cref{exa:isotropic-points}, the period of an isotropic point~$\cP=\iso{\EE}$ is~$0$.
By \Cref{sta:d-periodic-basics-geometric} the period of~$\pi(\cP)$ must then be~$0$ too.
But the only non-periodic points in $\Spc(\DAM(\FF;\ZZ/p)^{\omega})$ are the closed ones, by \Cref{exa:DAM}.
\end{proof}

\begin{rmk}
\label{rmk:armer-tor}
Let $\cK$ be a rigid tt-category and $F\colon \cK\to\cL$ a faithful tt-functor.
By \cite[Corollary~1.8]{MR3829735}, the associated map on spectra $f\colon\Spc(\cL)\to\Spc(\cK)$ is surjective.
In particular, the map~$\pi$ of~\eqref{eq:pi} is surjective.
We just saw that all isotropic points map to only the closed points, which means that there must be (many) additional points in~$\Spc(\DM(\FF;\ZZ/p)^\omega)$.
In fact, we know that there must be (many) additional points lying in the periodic locus because of \Cref{exa:DAM}.
To our knowledge, the only such points known (for $\FF$ of any characteristic) are \'etale points as described in \Cref{exa:etale-realization}.
We note that they only exist if $\FF$ has orderings (in particular, is of characteristic zero) and $p=2$.
And even then they typically cover only a tiny bit of the spectrum, see \Cref{rmk:etale-vs-artin-tate}.

In other words, we seem to have barely scratched the surface of $\Spc(\DM(\FF;\ZZ/p)^\omega)$.
\end{rmk}

\begin{rmk}
\label{rmk:DM-many-isotropic-points}
As mentioned in \Cref{rmk:isotropic-equivalence-relation}, two isotropic points in~$\DM(\FF;\ZZ/p)$ coincide iff the corresponding fields are in relation~$\psim$.
For example, in~\cite[Example~5.14]{MR4768634} it is shown that for $p=2$ and $\FF=\RR$, there are $2^{\mathfrak{c}}$ equivalence classes~$(\EE)_{\psim}$ (where $\mathfrak{c}$ denotes the cardinality of the continuum).
Both as an illustration and for later use we now describe more explicitly this equivalence relation for \emph{algebraic} extensions.
\end{rmk}

\begin{notn}
Let $G$ be a profinite group and $H,H'\leq G$ two subgroups.
We define $H\pleq H'$ iff $H(p)\leq_G H'$ (subconjugate) where $H(p)\leq H$ is a pro-$p$-Sylow subgroup.
We also define $H\psim H'$ if $H\pleq H'$ and $H'\pleq H$.
\end{notn}

\begin{lem}
Let $\EE,\EE'/\FF$ be algebraic extensions corresponding to closed subgroups $H,H'\subseteq G_{\FF}$.
Then:
\begin{equation}
\label{eq:EpE'}
\EE\pgeq\EE'\qquad\Leftrightarrow\qquad H\pleq H'
\end{equation}
\end{lem}
\begin{proof}
We first consider the case where both $\EE$ and $\EE'$ are finitely generated hence finite.
Then $\EE\otimes_\FF\EE'\cong\prod_{g\in H\backslash G_{\FF}/H'}\bar{\FF}^{H\cap{}^gH'}$ is a product of finite extensions whose degrees over~$\EE=\bar{\FF}^H$ are given by~$[H\colon H\cap{}^gH']$.
We conclude that $\EE\pgeq\EE'$ iff there exists~$g\in G_{\FF}$ such that $p\nmid [H\colon H\cap{}^gH']$.
Letting $H(p)\leq H$ be a pro-$p$-Sylow, this is equivalent to the existence of~$g$ such that $p\nmid [H(p)\colon H(p)\cap{}^gH']$, or $H(p)\leq_G H'$.
This shows~\eqref{eq:EpE'} in this case.

More generally, write $\EE=\colim \EE_\alpha$ as the union of its finite subextensions, and let us still assume $\EE'/\FF$ finite.
It is then clear that~\eqref{eq:EpE'} still holds since a cofiltered intersection of subgroups is contained in an open subgroup iff one of them is.

Finally, in the general case consider the open normal subgroups~$N_\beta\leq G_\FF$, set $H_\beta'=N_\beta H'$ and note that $\EE'=\colim \bar{\FF}^{H_\beta'}$.
For each $\beta$ let $X_\beta\subseteq G_\FF/N_\beta$ be the (finite) subset of elements $gN_\beta$ such that $H(p)N_\beta/N_\beta\leq{}^gH_\beta'/N_\beta$.
This is an inverse subsystem of $(G_\FF/N_\beta)_\beta$ and its limit is non-empty iff each $X_\beta$ is.
But an element $g\in G_\FF=\varprojlim G_\FF/N_\beta$ in this inverse limit is characterized by satisfying $H(p)\leq{}^gH'$.
This completes the proof.
\end{proof}

\begin{rmk}
The equivalence classes of subgroups with respect to this relation are in bijection with pro-$p$-subgroups.
The reader will remember that these are also used to index the strata in~\eqref{eq:DAM-strata}.
It will be convenient in the sequel to agree on the following short-hand for the closed points in~$\Spc(\DPerm(G;\ZZ/p)^{\omega})$ for an arbitrary pro-finite group~$G$.
Given a subgroup~$H\leq G$, we set
\[
\gm(H):=\gm(H(p)),
\]
where $H(p)\leq H$ is a pro-$p$-Sylow subgroup.
This is rigged in such a way that $\gm(H)=\gm(H')$ iff $H\psim H'$.
We will see in a bit that, as one might guess, for an algebraic extension~$\EE/\FF$ corresponding to~$H\leq G_\FF$, one has
$\pi(\iso{\EE})=\gm(H)$.
\end{rmk}

The following result generalizes~\cite[Example~7.30]{MR4946248}.
\begin{lem}
\label{sta:very-closed-point}
Let $G$ be profinite.
The closed point $\gm(G)\in\Spc(\DPerm(G;\ZZ/p)^{\omega})$ is generated (as a tt-ideal) by the objects
\[
\ZZ/p(G/H),\qquad p\bigm\vert [G:H].
\]
\end{lem}
\begin{proof}
Let $\cI_G$ denote the tt-ideal in the statement.
Assume first that $G$ is finite and let $P$ be a $p$-Sylow.
By definition~\cite[7.26]{MR4946248}, $\gm(G)$ is obtained from~$\gm(P)\in\Spc(\DPerm(P;\ZZ/p)^{\omega})$ via restriction.
For the latter point we know the statement from~\cite[Example~7.30]{MR4946248}.
Now, for $X\in\DPerm(G;\ZZ/p)^{\omega}$,
\begin{align*}
  X\in\gm(G) &\Leftrightarrow \Res^G_P(X)\in\gm(P)\\
             &\Leftrightarrow \Res^G_P(X)\in\cI_P=\langle \ZZ/p(P/H), \ H\lneq P\rangle\\
             &\Leftrightarrow X\in\langle\Ind^G_P\ZZ/p(P/H), \ H\lneq P\rangle 
\end{align*}
where the last equivalence follows from the fact that $X$ is a direct summand of $\Ind^G_P\Res^G_P(X)$ and from the Mackey formula.
But the last ideal is just $\cI_G$.

Now, for the general case, consider the system of open normal subgroups $(N_\beta)_\beta$.
Under the identification of \Cref{notn:limit} (cf.\ \Cref{exa:DAM}) the point $\gm(G)$ corresponds to the system $(\gm(G/N_\beta))_\beta$.
Thus the claim follows from the finite case.
\end{proof}

\begin{cor}
\label{sta:DPerm-closed-characterization}
Let $H,H'\leq G$ be two subgroups with $H'$ open.
Then $\ZZ/p(G/H')\in\gm(H)$ iff $H\not\pleq H'$.
\end{cor}
\begin{proof}
It is immediate from \Cref{sta:very-closed-point} that $\ZZ/p(H/L)\notin\gm(H)\subseteq\DPerm(H;\ZZ/p)$ iff $p\nmid[H:L]$.
Therefore, 
$\ZZ/p(G/H')\notin\gm(H)=(\Res^G_{H})\inv\gm(H)$ iff there exists $g\in G$ such that $p\nmid[H:H\cap{}^gH']$, by the Mackey formula.
Which is equivalent to $H(p)\leq_G H'$, that is, $H\pleq H'$.
\end{proof}

To describe exactly the images of isotropic points in $\Spc(\DAM(\FF;\ZZ/p))$ we need a short interlude on the naturality of the identification between Artin motives and derived permutation modules (\Cref{exa:DAM}).
When $\EE/\FF$ is a finite extension we showed in~\cite[Corollary~8.11]{MR4693637} that the square
\begin{equation}
\label{eq:DAM-vs-DPerm}
\begin{tikzcd}
\DAM(\FF;\ZZ/p)
\ar[r, "-\otimes_\FF\EE"]
\ar[d, "\sim", leftrightarrow]
&
\DAM(\EE;\ZZ/p)
\ar[d, "\sim", leftrightarrow]
\\
\DPerm(G_{\FF};\ZZ/p)
\ar[r, "{\Res^{G_{\FF}}_{G_\EE}}"]
&
\DPerm(G_{\EE};\ZZ/p)
\end{tikzcd}
\end{equation}
commutes, in which we view $G_{\EE}\hookrightarrow G_{\FF}$ as a subgroup (canonical once one chooses an algebraic closure of~$\EE$).
Our goal is to extend this result to arbitrary extensions.

\begin{cns}
Let $\EE/\FF$ be an arbitrary extension and $\bar{\EE}/\EE$ an algebraic closure.
Let $\bar{\FF}\subseteq\bar{\EE}$ be the algebraic closure of~$\FF$ inside~$\bar{\EE}$ and consider the composite morphism
\begin{equation}
\label{eq:r}
 r_{\EE/\FF}\colon G_{\EE}=\Gal(\bar{\EE}/\EE)\hookrightarrow\Gal(\bar{\EE}/\FF)\twoheadrightarrow\Gal(\bar{\FF}/\FF)=G_{\FF}.
 \end{equation}
 We denote by $\Res^{G_{\FF}}_{G_\EE}\colon\DPerm(G_{\FF};\ZZ/p)\to\DPerm(G_{\EE};\ZZ/p)$ the induced $\ZZ/p$-linear symmetric monoidal left adjoint functor.
\end{cns}

\begin{prop}
\label{sta:DAM-vs-DPerm}
Let $\EE/\FF$ be an arbitrary extension.
Then the square of stable presentably symmetric monoidal $\ZZ/p$-linear $\infty$-categories in~\eqref{eq:DAM-vs-DPerm} commutes.
\end{prop}
\begin{proof}
Since all functors are symmetric monoidal left adjoints and preserve the weight hearts on compact objects~\cite[\S\,5]{fuhrmann2025modularfixedpointsequivariant} it suffices (in fact, is equivalent) to show the commutativity of $\ZZ/p$-linear symmetric monoidal functors on the weight hearts.
By $\ZZ/p$-linearity and the universal property of permutation modules~\cite[Lemma~4.8]{MR4693637} one further may restrict to $G_{\FF}$-sets.
Thus let $A$ be a finite \'etale $\FF$-algebra with corresponding $G_{\FF}$-set
\[
\hom_{\FF}(A,\bar{\FF})=\hom_{\FF}(A,\bar{\EE})=\hom_{\EE}(A\otimes_\FF\EE,\bar{\EE}).
\]
The right-most term is the $G_{\bar{\EE}}$-set corresponding to $A\otimes_\FF\EE$ and the action is induced via~$r_{\EE/\FF}$ by the action of $G_{\FF}$ on the left-most term.
This implies the claim.
\end{proof}

\begin{prop}
Let $\pi\colon\Spc(\DM(\FF;\ZZ/p)^\omega)\to\Spc(\DAM(\FF;\ZZ/p)^\omega)$ be the map on spectra induced by the obvious inclusion.
For any extension $\EE/\FF$ we have:
\[
\pi\left(\iso{\EE}\right)=\gm(r_{\EE/\FF}(G_\EE))\]
\end{prop}
\begin{proof}
Assume first that $\EE=\FF$.
A finite extension $\FF'/\FF$ belongs to $\iso{\FF}$ iff $\FF'\otimes_\FF\FF(\PP^\infty)$ is $p$-anisotropic which is equivalent to $p\bigm\vert [\FF':\FF]$, or $\FF\not\pgeq\FF'$.
Because of \Cref{sta:isotropic-closed-image} we know that $\pi(\iso{\FF})$ is closed.
But \Cref{sta:DPerm-closed-characterization} says that the closed points are characterized by which (indecomposable) permutation modules they contain.
Applying that result to $H=G_\FF$ yields the statement.

Now, let $\EE/\FF$ be an arbitrary extension.
A similar argument as in \Cref{sta:DPerm-closed-characterization} yields that $\ZZ/p(G_\FF/H')\in \gm(r_{\EE/\FF}(G_{\EE}))$ iff $\Res^{G_\FF}_{G_\EE}\ZZ/p(G_\FF/H')\in \gm(G_\EE)$.
By the previous case we know $\gm(G_\EE)=\pi(\iso{\EE})\in\Spc(\DPerm(G_\EE;\ZZ/p)^{\omega})$ so the statement follows from \Cref{sta:DAM-vs-DPerm}.
\end{proof}

\appendix{}

\section{Generalized comparison maps}
\label{sec:gener-comp-maps}
\sectionauthor{Ivo Dell'Ambrogio}

In \cite{MR3163513} Greg Stevenson and I developed a flexible generalization of Balmer's comparison map $\rho\colon \Spc(\cat K) \to \Spch(\End_\cat K^*(\unit))$ from the tt-spectrum of a tt-category~$\cat K$ to the Zariski spectrum of its graded central ring, where we allowed more variety of choice for the target space.
Some such variations on the basic theme were already possible in~\cite{balmer:sss}, where the grading of the central ring could be twisted by any invertible object~$\omega$ other than the canonical choice~$\omega=\Sigma(\unit)$. 
Morally one should also be able to use a collection of several invertible objects to obtain gradings by abelian groups  more general than~$\mathbb Z$, but it is unclear in what generality such `multigraded central rings' actually exist as ordinary rings, because one quickly runs into coherence problems.
(In the literature such problems are sometimes swept under the carpet.)
The solution of \cite{MR3163513} was to avoid coherence issues altogether by considering central graded commutative \emph{2-rings} of~$\cat K$, \ie full tensor subcategories whose objects are all invertible. 
It turns out that such 2-rings~$\cat R$ admit a Zariski spectrum of homogeneous prime ideals, defined essentially in the usual way, as well as a comparison map $\Spc(\cat K)\to \Spc(\cat R)$ enjoying all the basic properties of Balmer's.

Somehow, in \cite{MR3163513} we had neglected to verify that \emph{if} there exists a multigraded ring~$R$ which is suitably equivalent to~$\cat R$, then the resulting Zariski spectra and comparison maps actually agree.
The first goal of this appendix is to remedy that oversight; see \Cref{thrm:agreement} and \Cref{cor:multigraded-comparison}.

The second goal is to further generalize our comparison maps by allowing central 2-rings which are graded by a sub-2-monoid of the Picard 2-group; see \Cref{def:general-central-2-rings}.
This generalization is akin to considering (and in fact generalizes) an $\mathbb N$-graded central ring of~$\cat K$ such as the connective or coconnective part, $\End^{\leq 0}(\unit)$ or $\End^{\geq 0}(\unit)$, of the usual graded central ring.
We show the Zariski spectrum and comparison map of \cite{MR3163513} work equally well in this greater generality.
In particular, they are still compatible with central localization; see \Cref{thrm:central-loc} and \Cref{cor:central_loc_in_context}.

From now on, fix an essentially small tt-category~$\cat K$ with tensor $\otimes$ and unit~$\unit$.

\subsection{Comparison maps for central 2-rings of $\cat K$}

We recall the main tools of \cite{MR3163513}:

\begin{defn} \label{defn:2-ring}
A \emph{graded commutative 2-ring} is an essentially small preadditive ($=$~$Ab$-enriched) symmetric monoidal category~$\cat R$ in which every object is tensor-invertible. 
The tensor product is assumed to be additive in each variable.
By hypothesis, the maximal subgroupoid of~$\cat R$ is equal to its Picard category $\2Pic(\cat R)$ ($=$~its tensor subcategory of invertible morphisms and invertible objects); hence in particular it is an \emph{abelian 2-group} ($=$~a symmetric monoidal category in which all morphisms and objects are invertible) having the same objects as~$\cat R$.
We think of the 2-ring $\cat R$ as being graded by the abelian 2-group~$\2Pic(\cat R)$. 
In the following we will often drop the words ``graded commutative'' and ``abelian'', since no other kind of 2-rings or 2-groups will be considered.
\end{defn}

\begin{rmk} [Commutativity and translation]
\label{rmk:translate} 
We call such 2-rings ``graded commutative'' because the product of its elements (\ie the composition of its morphisms) actually commutes `up to~$\2Pic(\cat R)$', \ie up to taking twists by objects and composing with isomorphisms. 
More precisely, we say that a morphism $\tilde r \in \cat R$ is a \emph{translate} of another $r\in \cat R$ if 
$\tilde r$ can be obtained from $r$ by applying any finite sequence of twists and compositions with isomorphisms.
This is equivalent to having $\tilde r = v (g \otimes r) u$ for a single object $g$ and two isomorphisms $u,v$ (\cite[Lemma~2.7]{MR3163513}), and defines an equivalence relation (\emph{translation}) on the morphisms of~$\cat R$.
Then one can show that for any two composable morphisms $r$ and $s$ in~$\cat R$ we have $s \circ r = \tilde r \circ \tilde s$ for some translates $\tilde r$ of~$r$ and $\tilde s$ of~$s$  (\cite[Prop.\,2.9]{MR3163513}).
\end{rmk}

Based on this, \cite{MR3163513} proceeds to develop the affine algebraic geometry of a graded commutative 2-ring~$\cat R$, pretty much as if it were an ordinary graded ring.
A \emph{homogeneous ideal}, or just \emph{ideal}, $\cat I$ of $\cat R$ is a (two-sided) categorical ideal of morphisms which is closed under twists, \ie under tensoring morphisms with arbitrary objects (and hence under translation).
An ideal $\mathfrak p$ is \emph{prime} if it is proper ($\neq \cat R$) and if $s\circ r\in \mathfrak p $ implies $ s\in \mathfrak p$ or $r\in \mathfrak p$.
The set 
\[ \Spc(\cat R)
\]
of all prime ideals of $\cat R$ is endowed with the \emph{Zariski topology}, where closed subsets are of the form
$V(\cat I) = \{\mathfrak p\in \Spc(\cat R)\mid \cat I\subseteq \mathfrak p\}$ 
for some ideal $\cat I\subseteq \cat R$.
Maximal ideals (for inclusion) are prime, and we have $\Spc(\cat R)=\emptyset$ if and only if $\cat R\simeq 0$.
The resulting \emph{Zariski spectrum} is a spectral space in the sense of Hochster, and is functorial for morphisms $F\colon \cat R\to \cat R'$ of 2-rings (\ie additive tensor functors): they induce  spectral continuous maps 
\[
\Spc(F)\colon \Spc(\cat R')\to \Spc(\cat R)
\]
via $\Spc(F)(\mathfrak p)= F^{-1}\mathfrak p$, so that $\Spc(F_2 \circ F_1)= \Spc(F_1)\circ \Spc (F_2)$ as well as $\Spc(\Id_\cat R)=\Id_{\Spc(\cat R)}$.
More details on all this can be found in \cite[\S2]{MR3163513}.

\begin{defn} [Central 2-rings of $\cat K$]
\label{def:general-central-2-rings}
Now recall our tt-category~$\cat K$, and let $\cat M$ be any replete ($=$~closed under isomorphic objects) full sub-2-monoid ($=$~full symmetric monoidal subcategory) of its Picard category~$\2Pic(\cat K)$.
We associate to~$\cat M$ a graded commutative 2-ring $\cat R_\cat M$ defined as follows. 
The objects of~$\cat R_\cat M$ are those of the 2-group-completion $\overline{\cat M}$, the full sub-2-group of $\2Pic(\cat K)$ generated by~$\cat M$ (add the inverses of all objects). 
Its Hom groups are
\[
\cat R_\cat M (g,h) :=  
\left\{
\begin{array}{ll}
\cat K(g,h) & \textrm{if $h\otimes g^{-1} \in \cat M$ (\ie if $h\cong m\otimes g$ for some $m\in \cat M$)}\\
0 & \textrm{otherwise.}
\end{array}
\right.
\]
The composition and tensor structure are inherited from $\cat K$ in the evident way, and we easily check that the result is indeed a graded commutative 2-ring~$\cat R_\cat M$, graded by the 2-group~$\overline{\cat M}$.
(For the latter: if $r\in \cat K(g,h)$ is an isomorphism between any $g,h\in \overline{\cat M}$, then $h\otimes g^{-1}\cong g\otimes g^{-1}\cong \unit \in \cat M$, from which it follows that indeed $\2Pic(\cat R_\cat M)=\overline{\cat M}$.)
We say that $\cat R_\cat M$ is the \emph{central 2-ring} associated to~$\cat M$, or by a slight abuse of language, that it is the \emph{$\cat M$-graded} central 2-ring of~$\cat K$.
\end{defn}

\begin{rmk} \label{rmk:full}
We stress that $\cat R_\cat M$ is in general not a full subcategory of~$\cat K$ (because of the ``otherwise'' case). 
Indeed, it is full precisely when $\cat M$ is already a sub-2-group, that is when $\overline {\cat M}=\cat M$. 
Only full central 2-rings were considered in~\cite{MR3163513}.
\end{rmk}

\begin{rmk} \label{rmk:bijection}
By definition, however, we require central 2-rings to be replete in~$\cat K$, mostly for aesthetic reasons (similarly to thick subcategories etc.).
We also ask that our full sub-2-monoids $\cat M\subseteq \2Pic(\cat K)$ be replete. 
The latter implies that every such~$\cat M$ is uniquely determined by the ordinary submonoid $M=\Kth_0(\cat M)$ of the ordinary Picard group $\Pic(\cat K)$.
Consequently, we may view the central 2-rings of~$\cat K$ as being parametrized by the submonoids of the Picard group
\[
M \subseteq \Pic(\cat K) 
\quad \longmapsto \quad
\cat R_M := \cat R_\cat M \subset \cat K ,
\]
where $\cat M\subseteq \cat K$ is the full tensor subcategory of all objects whose isomorphism class lies in~$M$. 
(Note that \emph{a~priori} the above mapping need not be injective because~$M$ may contain classes $[g]$ such that $\cat K(\unit , g)=0$, so that a proper submonoid of $M$ omitting such~$[g]$ can still have the same associated central 2-ring.)
\end{rmk}

\begin{exa} [Minimal and maximal central 2-rings]
\label{exa:basic-full-2-rings}
The minimal and maximal sub-2-monoids, $\cat M_{\min}=\{\unit\}$ and $\cat M_{\max}=\2Pic(\cat K)$, are sub-2-groups and thus determine two full central 2-rings of~$\cat K$.
For all purposes, as we shall see, the minimal one $\cat R_{\{\unit\}}$ can be identified with  the ordinary central ring $\mathrm R_\cat K = \End(\unit)$.
\end{exa}

\begin{exa} [Cyclic central 2-rings]
\label{exa:basic-graded-2-rings}
For every invertible object $\omega$ of $\cat K$ we can always define the following three (full and replete) sub-2-monoids of~$\2Pic(\cat K)$
\[
\cat M_\omega:=\{\Sigma^n \omega \mid n\in \mathbb Z\}  , \quad
\cat M_\omega^+:=\{\Sigma^n \omega \mid n\in \mathbb N\}   ,\quad 
\cat M_\omega^-:=\{\Sigma^{-n} \omega \mid n\in \mathbb N\}
 \]
and therefore three central 2-rings $\cat R_\omega := \cat R_{\cat M_\omega}$, $\cat R_\omega^+ := \cat R_{\cat M_\omega^+}$ and $\cat R_\omega^- := \cat R_{\cat M_\omega^-}$, 
only the first of which is always full. 
(If $\cat R_\omega^+$ or $\cat R_\omega^-$ is full it must coincide with~$\cat R_\omega$.) 
They could be called, respectively, the \emph{full ~/ coconnective~/ connective $\omega$-twisted cyclic central 2-ring} of~$\cat K$. 
 The canonical choice $\omega = \Sigma(\unit)$ is always available, in which case we would drop the ``$\omega$-twisted'' qualification.
\end{exa}

\begin{thrm} [General comparison maps]
\label{thrm:comparison-map}
For every tt-category $\cat K$ and every full sub-2-monoid $\cat M\subseteq \2Pic(\cat K)$, the formula $\rho(\cat P) = \{r \in \cat R_\cat M \mid \cone(r)\not \in \cat P\} $ defines a spectral continuous map
\[
\rho = \rho_{\cat K, \cat M} \colon \Spc(\cat K) \to \Spc(\cat R_\cat M)
\]
from the tt-spectrum of $\cat K$ to the Zariski spectrum of its $\cat M$-graded central ring~$\cat R_\cat M$.
This map is natural in the pair $(\cat K, \cat M)$, in the sense that if $(\cat K', \cat M')$ is another such pair and $F$ is a tt-functor $\cat K \to \cat K'$ such that $F\cat M\subseteq \cat M'$, the resulting square is commutative:
\[
\xymatrix{
\Spc(\cat K') \ar[d]_{\rho_{\cat K', \cat M'}} \ar[r]^-{\Spc(F)} & \Spc(\cat K) \ar[d]^{\rho_{\cat K, \cat M}} \\
\Spc(\cat R_{\cat M'}) \ar[r]^-{\Spc(F)} & \Spc(\cat R_{\cat M})
}
\]
\end{thrm}

\begin{proof}
Consider the inclusion 
$
J\colon \cat R_\cat M \to \cat R_{\overline{\cat M}}
$
of the $\cat M$-graded central ring into the $\overline{\cat M}$-graded central ring, where as before $\overline{\cat M}$ is the 2-group completion of~$\cat M$ (inside $\2Pic(\cat K)$).
Thus $J$ is the inclusion of a wide sub-2-ring which Hom-wise is `full or zero'.
Since $J$ is a morphism of 2-rings, it induces a continuous map 
\[
\Spc(J) \colon \Spc(\cat R_{\overline{\cat M}}) \to \Spc(\cat R_\cat M)
\]
by restriction of ideals: $\Spc(J)(\mathfrak p)= \cat R_\cat M\cap \mathfrak p$. 
Since $\cat R_{\overline{\cat M}}$ is a \emph{full} central 2-ring of~$\cat K$, we know from \cite[Thm.\,3.10]{MR3163513} that there is a spectral continuous comparison map 
\[
\rho_{\cat K, \overline{\cat M}}\colon  \Spc(\cat K) \to \Spc(\cat R_{\overline{\cat M}})
\]
sending the tt-prime $\cat P$ to the prime ideal $\rho(\cat P) = \{r \in \cat R_{\overline{\cat M}} \mid \cone(r)\not \in \cat P\} $.
Now we simply define $\rho_{\cat K, \cat M}$ to be the composite 
\[
\xymatrix{
\Spc(K) \ar[r]^-{\rho_{\cat K,\overline{\cat M}}} & \Spc(\cat R_{\overline{\cat M}}) \ar[r]^-{\Spc(J)} & \Spc(\cat R_\cat M)
}
\]
which again is spectral and continuous, and which is given by the claimed formula:
\begin{align*}
\rho_{\cat K, \cat M} (\cat P) 
&= \Spc(J) \big( \rho_{\cat K,\overline{\cat M}} (\cat P) \big) 
= \cat R_{\cat M}\cap  \{ r\in \cat R_{\overline{\cat M}} \mid \cone(r)\not\in \cat P\} \\
&= \{ r\in \cat R_{\cat M} \mid \cone(r)\not\in \cat P\}.
\end{align*}
A similarly immediate calculation verifies the last claim of the theorem.
\end{proof}

\subsection{Comparison maps for general multigraded rings}

We now specify what we mean by an `ordinary multigraded commutative ring', in such generality as should suffice for all future purposes: 

\begin{defn} \label{defn:graded-ring}
Let $G$ be an abelian group, which we will write multiplicatively.
A \emph{$G$-graded} (or simply \emph{multigraded}) \emph{ring} is, of course, a ring together with a direct sum decomposition in homogeneous parts $R= \bigoplus_{x\in G}R_x$ such that $1\in R_1$ and $R_xR_y\subseteq R_{xy}$.
We say $R$ is a \emph{multigraded commutative} ring if it comes with a \emph{transposition rule}~$\tau$, that is a function $\tau\colon G\times G\to \mathrm Z(R)_1^\times$ into the abelian group of invertible central elements of degree one, satisfying:
\begin{itemize}
\item symmetry: $\tau(x,y)=\tau(y,x)$,
\item bilinearity, \ie written multiplicatively: $\tau(x,yz) = \tau(x,y)\tau(x,z)$.
\item transposition: $rs = \tau(|r|, |s|) sr$ for homogeneous $r\in R_{|r|}$ and $s\in R_{|s|}$.
\end{itemize}
\end{defn}

\begin{rmk}
The symmetry and bilinearity are analogous to the axioms of a symmetry for a monoidal category, and indeed they ensure that we may perform any sequence of transpositions (as in the third axiom) of adjacent elements in a product of several elements of~$R$, and the result will only depend on the overall permutation.
Typically $\tau$ is chosen to only take values in $\pm 1$ (as assumed for example in \cite{MR2995031}), but it does not have to be so.
However, since transposition implies both $\tau(|r|,|s|)^{-1} rs = sr$ and $\tau(|s|,|r|) rs = sr$, symmetry suggests we may want to assume from the start that $\tau(x,y)=\tau(x,y)^{-1}$, that is $\tau(x,y)^2=1$. 
\end{rmk}

If $R$ is any multigraded commutative ring as in \Cref{defn:graded-ring}, one can define homogeneous ideals (\ie two-sided ideals generated by homogeneous elements) and the Zariski spectrum 
\[
\Spch(R)
\]
of homogeneous primes precisely in the same way as for the more common $\mathbb Z$-graded commutative rings.
Note that the question whether a subset of $R$ is an ideal or a prime---and therefore the spectral space $\Spch(R)$---does not depend on the particular transposition rule~$\tau$ enjoyed by~$R$.

In the next definition (our most technical), one may wish to assume at first reading that $\pi$ is the identity map.

\begin{defn}[Tightening]
 \label{def:tightening}
Let $\cat R$ be any graded commutative 2-ring, let $\cat G=\2Pic(\cat R)$ be its grading 2-group, and let $\Kth_0(\cat G)$ be its abelian group of isomorphism classes $[g]$ of objects~$g$.
Let $\pi\colon G\twoheadrightarrow \Kth_0(\cat G)$ be a surjective morphism of abelian groups (for example the identity).
A \emph{tightening $R=(R, G, \pi, \{g_\gamma\}, \{\varphi_x\})$ of~$\cat R$} consists of a $G$-graded commutative ring $R=\bigoplus_{x\in G} R_x$, for some transposition rule~$\tau$ (\Cref{defn:graded-ring}), equipped with the choices of a complete set of representative objects $\{g_\gamma \}_{\gamma \in \Kth_0(\cat G)}$ for~$\Kth_0(\cat G)$ and of group isomorphisms
\[
\varphi_x \colon R_x \overset{\sim}{\to} \cat R (\unit, g_{\pi(x)}) \quad (x\in G),
\]
satisfying the following two conditions (where we write $\overline x:=\pi(x)$ for readability): 
\begin{enumerate}
\item
$g_1 = \unit$, and for every $x\in G$ and all $s \in R_1$ and $ r\in R_x$ we have
\[
\varphi_{x} (r s) = \varphi_{x} (r) \circ \varphi_1 (s) .
\]
In particular, with $x = 1$ we get a ring isomorphism $\varphi_1\colon R_1\overset{\sim}{\to} \End_\cat R(\unit)$.
\item
More generally, for any pair of homogeneous elements $r\in R_x$ and $s\in R_y$, the three morphisms of~$\cat R$ (the first and third hiding an unwritten unitor)
\[
(\varphi_{ x} (r) \otimes g_{\overline y}) \circ \varphi_{ y} (s)
\quad\quad
\varphi_{x}(r) \otimes \varphi_{y}(s) 
\quad\quad
(g_{\overline x} \otimes \varphi_{y}(s)) \circ \varphi_{x}(r)
\]
are all translates of $\varphi_{xy}(rs)$, as in \Cref{rmk:translate}. (In fact it suffices for one of them to be, since the three maps are all evident translates of one another.)
\end{enumerate}
\end{defn}

\begin{rmk}
Because of the asymmetry of the definition, these could more precisely be called \emph{right} tightenings, and of course there is an equivalent notion of left tightening. 
\end{rmk}

\begin{rmk}
The two compatibility conditions in \Cref{def:tightening} between the product of~$R$ and the composition of~$\cat R$ amount to the decent minimum one should ask if $R$ and $\cat R$ are to be deemed equivalent.
We could try to give a more structured definition, involving more choices and more axioms. 
We prefer to remain uncommitted so as to capture all possible solutions, and in any case this is enough for obtaining the expected comparison map of \Cref{cor:multigraded-comparison}.
\end{rmk}

\begin{rmk} \label{rmk:external-grading}
In our definition of a tightening,
the point of allowing morphisms $\pi$ which are \emph{not} the identity $G=\Kth_0(\2Pic(\cat R))$ is to capture how certain very common examples are usually treated, see \Cref{exa:original-graded}.
Besides, note that if $\pi\colon G\to \Kth_0(\cat G)$ is a group surjection and $R$ is a $\Kth_0(\cat G)$-graded ring, it is always possible to view $R$ as a $G$-graded ring by setting $R_x:=R_{\pi(x)}$ for all $x\in G$. 
Going the other way, however---to produce a $\Kth_0(\cat G)$-grading on a $G$-graded tightening---seems to necessarily involve choices and conditions which may fail.
Hence \emph{a~priori} our definition really is more general than if we only allowed $\pi=\Id$.
\end{rmk}

\begin{thrm} \label{thrm:agreement}
Let $R$ be a tightening of a graded commutative 2-ring~$\cat R$ as in \Cref{def:tightening}.
Then the evident extension and restriction of homogeneous ideals along the isomorphisms $\{\varphi_x\}_x$ of the tightening is a bijection and yields a homeomorphism
\[
\Spch(R) \cong \Spc(\cat R)
\]
between the Zariski spectra of the multigraded ring~$R$ and of the graded 2-ring~$\cat R$.
\end{thrm}

We begin with a couple of lemmas:

\begin{lem} \label{lem:magic}
In any tensor category, let $a,b \colon \unit \to g$ be two morphisms from the unit to the same invertible object~$g$. 
Then we have the equivalence
\[
\vcenter{
\xymatrix@R=5pt{
& g \ar[dd]^-w \\
\unit \ar[ur]^-a \ar[dr]_b & \\
& g
}}
\quad \Leftrightarrow \quad
\vcenter{
\xymatrix@R=5pt{
\unit \ar[dr]^-{a} &   \\
 & g \\
\unit \ar[ur]_-{b} \ar[uu]^{\tilde w} & 
}}
\]
for any endomorphism $w$ of $g$, with $\tilde w$ denoting the corresponding endomorphism 
\[
\xymatrix{
\tilde w \colon \unit \cong g^{-1} \otimes g
	 \ar[r]^-{g^{-1} \otimes w} & 
g^{-1} \otimes g \cong \unit
}
\]
of the tensor unit. 
Note also that $w$ is invertible iff $\tilde w$ is.
\end{lem}

\begin{proof}
This is a routine exercise in adjoint equivalences. 
(Alternatively: the equivalent claim that $wa = a \tilde w$ with $\tilde w = w \otimes g^{-1}$ follows precisely by the proof of \cite[Prop.\,2.9]{MR3163513} applied to $r:=(a\colon \unit \to g)$ and $s:=(w\colon g\to g)$. In fact it suffices to stare at the picture (2.10) in \emph{loc.\,cit.\ }keeping monoidal coherence in mind.)
\end{proof}

\begin{lem} \label{lem:translates-from-1}
Consider two morphisms $a\colon \unit \to g$ and $b\colon \unit \to h$ out of the unit in a graded commutative 2-ring~$\cat R$. 
Then $a$ and $b$ are translates of each other iff they are isomorphic as morphisms, iff there exists an isomorphism $w\colon g\overset{\sim}{\to} h$ with $wa=b$.
\end{lem}

\begin{proof}
Clearly it suffices to show that if $b$ is a translate of~$a$ then there exists  an isomorphism $w\colon g\overset{\sim}{\to} h$ with $wa=b$.
Any translate $b$ of $a$ is isomorphic to a single twist of~$a$, that is
\[ 
b = v \circ (k \otimes a) \circ u 
\]
for some object~$k$ and some isomorphisms $u\colon \unit \overset{\sim}{\to} k\otimes \unit$ and $v\colon k\otimes g \overset{\sim}{\to} h$.
The claim then follows immediately from the commutativity of the diagram
\[
\xymatrix{
& \unit
	 \ar@/_1ex/[dl]_{\tilde u} ^{\simeq}
	 \ar[r]^-{b} \ar[d]_u^\simeq  & 
h
	 \ar@{<-}[d]_\simeq^{v}
	 \ar@/^9ex/@{<-}[ddd]^{=: w} \\
k & k\otimes \unit
	 \ar[l]^-\simeq 
	 \ar[r]^-{k\otimes a} & 
k\otimes g \\
\unit
 	\ar@/_1ex/@{=}[dr] 
	\ar[u]^{\tilde u}_\simeq & 
\unit \otimes \unit 
	\ar[l]^-\simeq 
	\ar[d]^\simeq 
	\ar[u]^{\tilde u \otimes \unit}_\simeq 
	\ar[r]^-{\unit \otimes a} & 
\unit \otimes g
	 \ar[u]^-\simeq_{\tilde u \otimes g}
	 \ar[d]_\simeq \\
& \unit \ar[r]^-a & g
}
\]
where $\tilde u\colon \unit \overset{u}{\to} k \otimes \unit \cong k$.
\end{proof}

\begin{proof}[Proof of \Cref{thrm:agreement}.] This is similar to \cite[Lemma~4.1]{MR3163513} but more involved.
Write $R^\hmg= \coprod_{x \in G} R_x$ for the set of homogeneous elements of the multigraded ring~$R$, and denote by $\varphi:=\sqcup_{x} \varphi_{x\in G}\colon R^\hmg \to \cat R$ the (possibly noninjective) map provided with the tightening which identifies homogeneous elements with morphisms of~$\cat R$.

 If $\cat I\subseteq \cat R$ is a homogeneous ideal, its \emph{restriction} $\mathrm r(\cat I)\subseteq R$ is the additive closure in $R$ of the subset $\varphi^{-1} (\cat I) \subseteq R^\hmg$. 
Let us check that $\mathrm r(\cat I)$ is an ideal of~$R$. 
By multigraded-commutativity, it clearly suffices to see that the set $\mathrm r(\cat I)\cap R^\hmg = \coprod_{x \in G}\varphi^{-1}_{x}\cat I(\unit, g_{\overline x})$ is closed under multiplication on one side with arbitrary homogenous elements. 
Therefore consider $r \in \mathrm r(\cat I)\cap R_x$ and $s \in R_y$.
We want to show that $rs\in \mathrm r(\cat I)$, that is $\varphi(rs) \in \cat I$.
By the second axiom of a tightening, we know that $\varphi(rs)$ is a translate of $(g_{\overline x} \otimes \varphi(s))\circ  \varphi(r)$.
Since $\varphi(r)\in \cat I$ and since homogeneous ideals of~$\cat R$ are closed under translation and composition with arbitrary maps, we conclude that $\varphi(rs)\in \cat I$ as wished.

In the opposite direction, the  \emph{extension} $\mathrm e(I)$ of a homogeneous ideal $I\subseteq R$ is the homogeneous ideal of $\cat R$ generated by the set $\varphi(I \cap R^\hmg)$.

Clearly we have $\mathrm e(\mathrm r(\cat I)) \subseteq \cat I$.
For the other inclusion, note that every $a\in \cat I$ admits a translate of the form $\tilde a \colon \unit \to g_{\overline x}$ for some~$x\in G$, which still belongs to~$\cat I$ (since homogeneous ideals of $\cat R$ are translation-invariant) and therefore also to $\varphi(\mathrm r(\cat I)\cap R^\hmg)$ and thus to $\mathrm e(\mathrm r(\cat I))$. 
We deduce that $a\in \mathrm e(\mathrm r(\cat I))$, this time by the translation-invariance of~$\mathrm e(\mathrm r(\cat I))$. 
Hence $\mathrm e(\mathrm r(\cat I)) \supseteq \cat I$ and $\mathrm e(\mathrm r(\cat I)) = \cat I$.

On the other hand, starting with an ideal $I$ of~$R$, we always have $I\subseteq \mathrm r(\mathrm e(I))$ since $I$ is additively generated by its homogeneous elements.
For the reverse inclusion, let $r\in \mathrm r(\mathrm e (I))$. 
We need to show that $r\in I$, and for this we may assume that $r$ is homogenous.
Then by construction $\varphi(r) \in \mathrm e(I) (\unit , g_{\overline x})$ of some degree $x = |r|$.
Since by definition $\mathrm e (I)$ is the ideal of $\cat R$ generated by $\varphi(I\cap R^\hmg)$, we may use 
\cite[Lem.\,2.14 and Prop.\,2.15]{MR3163513} to give it a more precise description. 
To wit, $\mathrm e(I) (\unit , g_{\overline x})$ consists precisely of finite sums of morphisms of the form
\begin{equation*} \xymatrix{\unit \ar[r]^-b & \ell \otimes \unit \ar[r]^-{\ell \otimes a} & \ell \otimes g_{\overline y} \ar[r]^-{v}_-\simeq & g_{\overline x} }
\end{equation*}
where $a \in \varphi(I\cap R^\hmg)$---that is $a=\varphi(i)$ for some $i\in I\cap R_y$---and where $\ell$ is some object, $v$ some isomorphism, and $b$ some other morphism (not necessarily invertible).
Hence we may assume that $\varphi(r)$ has the above form.
Writing $\epsilon = [\ell]$, we may moreover suppose the composite has the following more specific form:
\begin{equation*} \xymatrix{\varphi(r) \colon \quad \unit \ar[r]^-b & g_\epsilon \otimes \unit \ar[r]^-{g_{\epsilon} \otimes a} & g_\epsilon \otimes g_{\overline y} \ar[r]^-{v}_-\simeq & g_{\overline x} }
\end{equation*}
We may also write $b$ as the composite $\unit \overset{\varphi(s)}{\longrightarrow} g_{\overline z} \cong g_{\overline z} \otimes \unit$ for some $s\in R_z$ with $\overline z=\epsilon$.
Therefore we can deduce from the second axiom of a tightening that the map $(g_\epsilon \otimes a)b=(g_{\overline z} \otimes \varphi(i) )\circ  \varphi(s)$ is a translate of $\varphi(si)$, with $si\in I\cap R_{zy}$ since $I\subseteq R$ is an ideal.
As both translates have source~$\unit$, by \Cref{lem:translates-from-1} there exists an isomorphism $u\colon g_{\overline z}\otimes g_{\overline y} \overset{\sim}{\to} g_{\overline{zy}}$ such that 
$u \circ (g_\epsilon \otimes a)b = \varphi(si)$. 
Altogether, we have the following commutative diagram
\[
\xymatrix@R=10pt{
& g_{\overline z} \otimes g_{\overline y} \ar[dr]_-\simeq^-{v} \ar[dd]_u^\simeq  & \\
\unit \ar[ur]^-{(g_{\overline z} \otimes a)b} \ar[dr]_{\varphi(si)} &  & g_{\overline x} \\
& g_{{\overline {zy}}} \ar@{-->}[ur]_-{vu^{-1}=:w}^-\simeq &
}
\]
where the top composite $\unit \to g_{\overline x}$ is $\varphi(r)$.
In particular, because of the isomorphism $w$ we see that ${\overline z}{\overline y}={\overline x}$ and therefore $g_{\overline z \overline y} = g_{\overline x}$, by the uniqueness of the chosen representatives for~$\Kth_0(\cat G)$.
Hence $\varphi(r)$ and $\varphi(si)$ have the same target object, and $w\circ \varphi(si) = \varphi(r)$. 
Therefore we deduce from \Cref{lem:magic} that $\varphi(si)\circ \tilde w = \varphi(r)$ for an automorphism $\tilde w$ of~$\unit$.
Since $\tilde w = \varphi(t)$ for some $t\in R_1$, we further deduce from the first axiom of a tightening that 
$\varphi(r)= \varphi(si) \varphi(t) = \varphi(sit)$.
Since $I $ is an ideal and $i\in I$, we conclude that $r= sit \in I$ too, as wished.

This completes the proof of the remaining inclusion $\mathrm r(\mathrm e(I)) \subseteq I$, establishing $\mathrm e(-)$ and~$\mathrm r(-)$ as mutually inverse inclusion-preserving bijections between the respective sets of homogenous ideals.

In order to prove the theorem, it remains to verify that the bijection restricts to prime ideals on both sides.
This can be proved by a similar element-wise (and tedious) argument.
Alternatively and more conceptually, we can exploit the fact that both spaces $\Spch(R)$ and $\Spc(\cat R)$ are the spectrum associated to a \emph{commutative ideal lattice} as in~\cite{MR2280286}. 
We have already established an isomorphism of the underlying ideal lattices, so we only need to show that it preserves the corresponding product operations, namely that $\mathrm e(IJ) = \mathrm e(I)\circ \mathrm e(J)$.   
Here $IJ$ is the usual ideal product in~$R$ and $\cat I\circ \cat J$ is the composition-product in~$\cat R$ (\cite[\S2.3]{MR3163513}).
By \cite[Lem.\,2.18]{MR3163513}, the latter is equal to $\langle \cat I \otimes \cat J\rangle$, the homogeneous ideal generated by all tensor products $a\otimes b$ with $a\in \cat I$ and $b\in \cat J$.
Now we can easily check the desired equality $\mathrm e(IJ) = \langle \mathrm e(I) \otimes \mathrm e(J)\rangle$ using the second axiom of tightenings (namely that $\varphi(rs)$ and $\varphi(r)\otimes \varphi(s)$ are translates) and translation-invariance of the homogeneous ideals of~$\cat R$. 
\end{proof}

\begin{cor} \label{cor:multigraded-comparison}
Let $\cat R_\cat M\subset \cat K$ be the central $\cat M$-graded 2-ring of~$\cat K$ associated with the sub-2-monoid $\cat M\subset \2Pic(\cat K)$.
Suppose the multigraded commutative ring $R$ is a tightening of $\cat R_\cat M$ in the sense of \Cref{def:tightening}, and write $\varphi$ for the map identifying the elements of $R$ as morphisms in $\cat R_\cat M \subset \cat K$. 
Under the homeomorphism of \Cref{thrm:agreement}, the comparison map of \Cref{thrm:comparison-map} yields a continuous spectral map 
\[
\rho=\rho_{\cat K, R} \colon \Spc(\cat K) \to \Spch(R)
\]
sending a tt-prime $\cat P$ of~$\cat K$ to the prime $\rho(\cat P) = \{r\in R \mid \cone(\varphi (r)) \not\in \cat P\}$ of~$R$.
\qed
\end{cor}

\begin{rmk}
Note that a tightening~$R$ of the central $\cat M$-graded 2-ring $\cat R_\cat M$ must have $R_x=0$ whenever $\pi(x)=[g]$ for $g\not\in \cat M$, so that~$R$ could actually be viewed as being graded by the submonoid $\pi^{-1}(\Kth_0(\cat M))$ of the abelian group~$G$.
\end{rmk}

\begin{exa} \label{exa:original-ungraded}
Obviously, the minimal central 2-ring $\cat R_{\{\unit\}}$ (\Cref{exa:basic-graded-2-rings}) always admits a unique tightening for which $\varphi_1$ is the identity morphism of $\End_\cat K(\unit)$.
The resulting comparison map is Balmer's original ungraded one.
\end{exa}

\begin{exa} \label{exa:original-graded}
Recall from \Cref{exa:basic-graded-2-rings} the full cyclic central 2-ring $\cat R_\omega$ twisted by an invertible object $\omega\in \cat K$. 
Its grading 2-group is $\cat G =\{\omega^{\otimes n}\mid n\in \mathbb Z\}$, for which $\Kth_0(\cat G)$ is a  (finite or infinite) cyclic group.
Balmer \cite{balmer:sss} defines a $\mathbb Z$-graded commutative ring $\mathrm R^\omega_\cat K=\End_\cat K^{*,\omega}(\unit)$ with homogeneous components 
 $ \mathrm R^\omega_n := \cat K(\unit , \omega^{\otimes n})$ $(n\in \mathbb Z)$
and transposition rule $\tau\colon \mathbb Z\times \mathbb Z\to \End_\cat K(\unit)^\times$ given by $\tau(n,m)= \epsilon^{nm}$, where $\epsilon\in \End(\unit)^\times$ is the \emph{central switch} of~$\omega$ (\ie the unique element whose action on $\omega\otimes \omega$ equals the tensor symmetry switch).
We can easily construct a tightening of $\cat R_\omega$ whose graded ring is $\mathrm R^\omega_{\cat K}$, in the following way. 
For the group morphism $\pi\colon \mathbb Z\twoheadrightarrow \Kth_0(\cat G)$ we map $n$ to~$\pi(n)=[\omega^{\otimes n}]$.
Now let $\ker(\pi)= p\mathbb Z$ where $p\geq 0$ is the periodicity of~$\omega$ (its order in~$\Kth_0(\cat G)$), and choose~(!) an isomorphism $\alpha \colon \omega^{\otimes p} \overset{\sim}{\to}  \unit$  (if $p=0$ we can simply take $\alpha = \Id$). 
For each $n\in \mathbb Z$, write $n=n' + pk$ with $0\leq n' < p$, define $g_{\pi(n)}:= \omega^{\otimes n'}$ to be the representative object of the class~$[\omega^{\otimes n}]$, and define the isomorphisms~$\varphi_n$ as
\[
\xymatrix{
\varphi_n\colon \mathrm (R^\omega_{\cat K})_n  = \cat K(\unit, \omega^{\otimes n'}\otimes (\omega^{\otimes p})^{\otimes k}) 
	 \ar[rr]^-{\Id \otimes \alpha^{\otimes k}}_-\simeq &&
\cat K(\unit, \omega^{\otimes n'}) = \cat K(\unit , g_{\pi(n)}) .
}
\]
That this defines a tightening is immediate from the formula for the product in~$\mathrm R^\omega_\cat K$.
We also directly see from the explicit formulas defining both maps that the tightening identifies the comparison map constructed in \cite{balmer:sss}, whose target is $\Spch(\mathrm R^\omega_\cat K)$, with the comparison map of \Cref{cor:multigraded-comparison} for the central 2-ring~$\cat R_\omega$.
This example has (among others) two evident variants, where on the ring side we take the connective or co\-conn\-ect\-ive subring of $\mathrm R^\omega_\cat K$ and on the 2-ring side the corresponding (co)conn\-ect\-ed cyclic central 2-ring $\cat R_\omega^-$ (resp.\,$\cat R_\omega^+$).
\end{exa}

\begin{rmk} [Free tightenings for all]
\Cref{exa:original-graded} can be generalized to any finite list of invertible objects $\omega_1,\ldots,\omega_r$ to produce a $\mathbb Z^r$-graded commutative ring whose homogeneous component at $(d_1,\ldots,d_r)\in \mathbb Z^r$ is $\cat K(\unit, \omega_1^{\otimes d_1} \otimes \ldots \otimes \omega_r^{\otimes d_r})$.
Taking colimits, this can be used to further show that every central graded 2-ring $\cat R_\cat M$ of~$\cat K$ admits a tightening whose grading group is (possibly infinite) free abelian; moreover, this construction essentially only depends on the choice of a generating set of objects for~$\cat M$ and a linear ordering on it.
(We omit the proof as it is quite lengthy and the result is mostly irrelevant to this article.)
We do not know, however, how to always find a tightening of $\cat R_\cat M$ with grading group~$G = \Kth_0(\overline{\cat M})$, and we suspect it may not be possible in general.
\end{rmk}

\subsection{Central localization}

To conclude this appendix, we briefly recall from \cite{MR3163513} the notion of central localization of $\cat K$ at multiplicative systems of central 2-rings, show how it also applies to \emph{2-monoids}-graded central 2-rings---as introduced here---and explain how it specializes the homonymous procedure introduced in~\cite{balmer:sss}.

Let $\cat R$ be any graded commutative 2-ring. 
A subset $S$ of morphisms of $\cat R$ is a \emph{(homogeneous) multiplicative system} of~$\cat R$ if it contains all isomorphisms and is closed under composition and twists; or equivalently, $S$ contains all identities and is closed under composition and translation (see~\Cref{rmk:translate}).
The localization $\cat R\to S^{-1}\cat R$ (in the sense of general categories) of $\cat R$ at a multiplicative system~$S$ enjoys a two-sided calculus of fractions and yields again a graded commutative 2-ring~$S^{-1}\cat R$. In particular, we can form the localization $\cat R_\mathfrak p := S_{\mathfrak p}^{-1}\cat R$ at any prime $\mathfrak p\in \Spc(\cat R)$ since $S_\mathfrak p= \mathrm{Mor}(\cat R)\smallsetminus \mathfrak p$ is a multiplicative system. This and more, generalizing familiar results from basic commutative algebra, can be found in \cite[\S2.5-6]{MR3163513}.

Let $\cat A$ be an \emph{$\cat R$-algebra}, by which we mean a preadditive (symmetric) tensor category $\cat A$ equipped with an additive tensor functor $F\colon \cat R\to \cat A$.
If  $S$ is a multiplicative system of~$\cat R$, its \emph{extension to~$\cat A$}, denoted~$S_\cat A$, is the smallest class of morphisms in $\cat A$ which contains $F(S)$ and all isomorphisms of $\cat A$ and which is closed under composition and under twists by arbitrary objects of~$\cat A$.
As in \cite{MR3163513}, we write $S^{-1}_\cat A\cat A$ for the resulting localized category.

\begin{defn} [Central localizations of $\cat K$]
\label{defn:central-locs-of-K}
Let $\cat R_\cat M$ be the central 2-ring associated to a full sub-2-monoid $\cat M\subseteq \2Pic(\cat K)$ (\Cref{def:general-central-2-rings}).
Let $S \subseteq \cat R_\cat M$ be any homogeneous multiplicative system as above. 
Then we can view $\cat K$ as an $\cat R_\cat M$-algebra, thanks to the inclusion functor $\cat R_\cat M\hookrightarrow \cat K$, and we can localize it at $S$ as such.
The resulting category $S^{-1}_\cat K\cat K$ is the \emph{central localization of $\cat K$ at~$S$}.
\end{defn}

Our next goal is to establish the good properties of~$S^{-1}_\cat K\cat K$.
In particular, it is a tt-category whose central 2-ring spanned by the objects of $\cat M$ is~$S^{-1}\cat R_\cat M$.

\begin{rmk} \label{rmk:enlarge-R}
If $\cat R \subseteq \cat R'$ is the inclusion of a sub-2-ring~$\cat R$ containing all objects and all isomorphisms of~$\cat R'$, each homogeneous multiplicative system of $\cat R$ is evidently also a homogeneous multiplicative system of~$\cat R'$.
In particular in \Cref{defn:central-locs-of-K} we may replace the 2-monoid~$\cat M$ with~$\overline{\cat M}$, and therefore the central 2-ring $\cat R_\cat M$ with $\cat R_{\overline{\cat M}}$, without affecting the resulting central localization~$S^{-1}_\cat K\cat K$.
\end{rmk}

\begin{thrm} \label{thrm:central-loc}
Let $\cat R_\cat M$ be the $\cat M$-graded central 2-ring of $\cat K$ associated to a sub-2-monoid $\cat M\subseteq \2Pic(\cat K)$.
Let $S$ be any homogeneous multiplicative system of $\cat R_\cat M$.
Then: 
\begin{enumerate}
\item
There is a canonical isomorphism of categories, which is the identity on objects, between the central localization of $\cat K$ at $S$ (\Cref{defn:central-locs-of-K}) and its Verdier quotient by the thick tensor ideal $\cat J_S$ generated by the cones of morphisms in~$S$:
\[
\xymatrix@R=10pt@C=10pt{
&\cat K \ar[dl]_-{\mathrm{loc.}} \ar[dr]^-{\mathrm{quot.}} & \\
S_\cat K^{-1} \cat K \ar[rr]^-{\simeq} && \cat K/ \cat J_{S}
}
\]
\item 
Write $\cat M_S \subseteq \2Pic(S^{-1}_\cat K \cat K)$ for the essential image of~$\cat M$, \ie the replete full sub-2-monoid spanned by the objects of~$\cat M$.
Then $\cat R_{\cat M_S}$, the $\cat M_S$-graded central 2-ring of $S^{-1}_\cat K\cat K$, is canonically equivalent to~$S^{-1}\cat R_\cat M$, the localization of~$\cat R_\cat M$ as a 2-ring.
\end{enumerate}
\end{thrm}

\begin{proof} For reference, recall the relationship between the categories we are dealing with before localizing:
\[
\xymatrix@C=10pt@R=10pt{
& \cat M
	 \ar@{}[r]|-{\textstyle{\subset}}
	 \ar@{}[d]|{\rotatebox{-90}{$\subset$}}
 & 
 \overline{\cat M} 
 	\ar@{}[r]|-{\textstyle{\subset}}
	\ar@{}[d]|{\rotatebox{-90}{$\subset$}} &
 \2Pic(\cat K)
 	\ar@{}[d]|{\rotatebox{-90}{$\subset$}} \\
S 
	\ar@{}[r]|-{\textstyle{\subset}} & 
\cat R_\cat M
	\ar@{}[r]|-{\textstyle{\subset}} &
 \cat R_{\overline{\cat M}}
 	\ar@{}[r]|-{\textstyle{\subset}} &
 \cat K
}
\]
By \Cref{rmk:enlarge-R}, $S$ is also a homogeneous multiplicative system of~$\cat R_{\overline{\cat M}}$. 
Note that the special case of both claimed statements when $\cat R_\cat M$ is a full central 2-ring (\ie when $\cat M$ is a 2-group) is precisely the content of \cite[Thm.\,3.6]{MR3163513}.  
Hence both statements hold true for the 2-ring $\cat R_{\overline{\cat M}}$ of~$\cat K$ and its multiplicative system~$S$. 

As the first statement is agnostic about whether we use $\cat R_{\overline{\cat M}}$ or $\cat R_{\cat M}$, there is nothing more to prove there.
For the second statement, we know from the already-proved full case that the \emph{full} central 2-ring of $S^{-1}_\cat K\cat K$ spanned by the objects of $\overline{\cat M}$ is $S^{-1}\cat R_{\overline{\cat M}}$.
We also have the following commutative square of morphisms of 2-rings, where $F$ is the functor induced by the universal property of localization:
\[
\xymatrix{
\cat R_\cat M
	 \ar[r]^-{\textrm{incl.}}
	 \ar[d]_{\textrm{loc.}} & 
\cat R_{\overline{\cat M}} 
	\ar[d]^{\textrm{loc.}} \\
S^{-1} \cat R_\cat M
	 \ar[r]^-F &
S^{-1} \cat R_{\overline{\cat M}}
}
\]
(All four functors are the identity on objects.) 
We claim that $F$, similarly to the inclusion at the top, is `full on $\cat M_S$ and zero elsewhere', meaning  that the mapping 
\begin{equation} \label{eq:F-on-Homs}
F\colon S^{-1}\cat R_\cat M(g,h) \longrightarrow S^{-1}\cat R_{\overline{\cat M}}(g,h)
\end{equation}
on Homs-groups is bijective if $h\otimes g^{-1} \in \cat M_S$ and otherwise $S^{-1}\cat R_\cat M(g,h)  = 0$.
This would immediately imply that $F$ provides an equivalence between $S^{-1}\cat R_\cat M$ and the $\cat M_S$-graded central 2-ring of~$S^{-1}_\cat K\cat K$, as per the second statement of the theorem.
It only remains to prove the above claim, for which we use the calculus of (right) fractions for localized 2-rings (\cite[Prop.\,2.31]{MR3163513}).
We are going to suppose that $S$ contains no zero morphisms, because otherwise all $S$-localized categories would be trivial and all claims would be trivially true.

To begin with, let us show the ``zero elsewhere'' part, \ie that $h \otimes g^{-1}\not\in \cat M_S$ implies $S^{-1} \cat R_\cat M(g,h)=0$.
Or rather,  let us prove the contrapositive:
\begin{equation} \label{eq:contrapositive}
S^{-1} \cat R_\cat M(g,h)\neq 0
\quad \Rightarrow \quad
h\otimes g^{-1} \in \cat M_S .
 \end{equation}
Thus suppose there exists a nonzero fraction
\begin{equation} \label{eq:fraction}
\xymatrix{ g & k \ar[l]_-s^-{\in S} \ar[r]^-a & h }
\end{equation}
with $s\in S$ and $a \in \cat R_{\cat M}$. 
Then \emph{a~fortiori}~$a$ is nonzero, implying by the definition of $\cat R_\cat M$ that $h\otimes k^{-1}\in \cat M$. 
Since $s$ is an isomorphism $k\overset{\sim}{\to}g$ in $S^{-1}\cat R_{\cat M}$, we deduce that $h\otimes g^{-1} \cong h \otimes k^{-1} \in \cat M_S$ as claimed.

Observe for later use that the functor $F$ is an isofibration, \ie that we can lift isomorphisms along~$F$. 
Indeed, an isomorphism $g\to h$ in $S^{-1}\cat R_{\overline{\cat M}}$ is represented by a fraction \eqref{eq:fraction} where both $s$ and $a$ are in~$S$, hence in particular $a\in \cat R_\cat M$ so that the same fraction already represents an isomorphism $g\to h$ in $S^{-1}\cat R_\cat M$.

Now suppose instead that $h\otimes g^{-1} \in \cat M_S$. 
This means that there is an isomorphism $h\otimes g^{-1}\cong m$ in $S^{-1}\cat R_{\overline{\cat M}}$ for some $m\in \cat M$;  since $F$ is an isofibration, it may be lifted to an isomorphism $h\otimes g^{-1}\cong m$ in $S^{-1}\cat R_\cat M$.
Therefore in showing \eqref{eq:F-on-Homs} is bijective we may and will suppose that $h\otimes g^{-1} \in \cat M$.

For injectivity, consider a fraction $as^{-1}\in S^{-1}\cat R_\cat M(g,h)$ as in \eqref{eq:fraction} such that $F(as^{-1})=0$, 
that is one with $s\in S(k,g)$, $a\in \cat R_\cat M(k, h)$, and such that $at =0 $ in $\cat R_\cat M$ for some $t\in S$.
Since $st\in S$, we then have an equivalence of fractions in~$\cat R_\cat M$
\[
\xymatrix@R=7pt{
& k \ar[dl]_s \ar[dr]^a& \\
g & & h \\
& \ell \ar[ul]^{st} \ar[uu]^t \ar[ur]_0 & 
}
\]
showing that $as^{-1}= 0$ in $S^{-1}\cat R_\cat M$. 
As $F$ is additive, this proves that \eqref{eq:F-on-Homs} is injective.
For surjectivity, suppose now that the fraction \eqref{eq:fraction} represents an arbitrary element of $S^{-1}\cat R_{\overline{\cat M}}(g,h)$, with $a\in \cat R_{\overline{\cat M}}(k,h)$ and $s\in S$. 
Since $S\subseteq \cat R_{\cat M}$ and $s\neq 0$, we must have $g\otimes k^{-1} \in \cat M$. 
Since $h \otimes g^{-1} \in \cat M$ by the current hypothesis and $\cat M$ is a 2-monoid, we have
\[
h \otimes k^{-1} \cong \underbrace{(h \otimes g^{-1})}_{\in \cat M} \otimes \underbrace{ (g \otimes k^{-1}) }_{\in \cat M} \; \in \; \cat M
\]
and therefore $a\in \cat R_\cat M$. 
Thus $as^{-1} \in S^{-1}\cat R_\cat M(g,h)$, which proves that  in this case~\eqref{eq:F-on-Homs} is also surjective and hence bijective.
\end{proof} 

\begin{cor} 
\label{cor:central_loc_in_context}
Let $M\subseteq \Pic(\cat K)$ be any submonoid and let $S$ be any set of morphisms between invertible objects whose isomorphism class lies in~$M$.
Let $\cat J_S$ be the tt-ideal of~$\cat K$ generated by $\{\cone(s)\mid s\in S\}$.
Then $\cat K/\cat J_S$ is a central localization in the above sense.
More precisely, writing $\cat R_M\subset \cat K$ for the corresponding central 2-ring (see \Cref{rmk:bijection}):
\begin{enumerate}
\item $\cat J_S = \cat J_{\tilde S}$, where $\tilde S$ is the homogeneous multiplicative system of~$\cat R_M$ generated by~$S$. 
\item The central 2-ring $\cat R_{q( M)}\subset \cat K/\cat J_S$ associated to the image $q(M)\subseteq \Pic(\cat K/\cat J_S)$ is equivalent to $\tilde{S}^{-1}\cat R_M$, the localization of $\cat R_M$ at~$\tilde S$ as a 2-ring.
\item Let  $q\colon \cat K\to \cat K/\cat J_S$ be the Verdier quotient. Then the naturality square 
\[
\xymatrix{
\Spc(\cat K/ \cat J_S) \;
	 \ar[d]_{\rho_{q (M)}}
	 \ar@{^{(}->}[rr]^-{\Spc(q)} && 
\Spc(\cat K)
	 \ar[d]^{\rho_M} \\
\Spch(\cat R_{q (M)})
	 \ar@{}[r]|-{\textstyle \cong} &
\Spch(\tilde{S}^{-1}\cat R_M) \;
	 \ar@{^{(}->}[r] & 
\Spch(\cat R_M)
}
\]
is a pullback of sets (and hence of spectral spaces), that is, $q$ induces a homeomorphism
$\Spc(\cat K/\cat J_S) \cong \{\cat P\in \Spc(\cat K)\mid  \rho_\cat M(\cat P) \cap S = \emptyset \}$.
\end{enumerate}
\end{cor}

\begin{proof}
Part~(1): Clearly $\cat J_S \subseteq \cat J_{\tilde{S}}$ since $S\subseteq \tilde S$. 
Note that the set of morphisms whose cone is contained in $\cat J_S$ includes all isomorphisms (because $0\in \cat J_S$), contains~$S$ (by definition), and is closed under composites (by the octahedral axiom) and twists ($\cat J_S$ is a tensor ideal).
On the other hand, $\tilde{ S}$ is obtained by adding to $S$ all isomorphisms and then closing in $\Mor(\cat R_\cat M)$  under composites of maps and twists by objects. 
We deduce the other inclusion $\cat J_S \supseteq \cat J_{\tilde{S}}$ too, which proves part~(1). 

Part~(2) is immediate from part~(1) together with part~(2) of \Cref{thrm:central-loc}.
To see this, write $\cat M = \cat M_M \subseteq \2Pic(\cat K)$ for the corresponding sub-2-monoid (so that $\cat R_M= \cat R_\cat M$), and note that $\cat R_{q(M)} = \cat R_{q(\cat M)}$ in the localized tt-category.

Part~(3) is easy and similar to \cite[Thm.\,5.4]{balmer:sss}).
Indeed, the naturality of the square is part of \Cref{thrm:comparison-map} and $q$ induces a homeomorphism $\Spc(\cat K/\cat J_S)\cong \{\cat P\in \Spc(\cat K)\mid \cat J_S\subseteq \cat P\}$ by general tt-geometry.
The latter condition $\cat J_S\subseteq \cat P$ on~$\cat P$ is equivalent to $\cone(s)\in \cat P$ for all $s\in S$ (by definition of~$\cat J_S$), which is equivalent to $S \cap \rho_\cat M(\cat P)=\emptyset$ (by definition of~$\rho_\cat M$ ), as claimed.
\end{proof}
\begin{center} $***$ \end{center}

Finally, let us compare our central localization with Balmer's. 
Let $\omega$ be any invertible object of~$\cat K$ and let $\cat R_\omega$ be the associated full $\omega$-twisted central 2-ring. 
Recall from \Cref{exa:original-graded} that Balmer's $\omega$-twisted graded central ring $\mathrm R^\omega_\cat K$ is part of a tightening of $\cat R_\omega$, so that \Cref{thrm:agreement} provides a homeomorphism of spectra
\[
\Spch(\mathrm R^\omega_\cat K)\cong \Spc(\cat R_\omega)
\] 
which identifies the respective comparison maps. Then:

\begin{cor} \label{cor:central-locs-agreement}
There is a canonical restriction-extension bijection between the multiplicative systems $S$ (containing all invertible homogeneous elements) in the $\mathbb Z$-graded ring~$\mathrm R^\omega_\cat K$ and the homogeneous multiplicative systems $\tilde S$ of~$\cat R_\omega$. 
Moreover, for every such system the central localization of $\cat K$ in the sense of \cite{balmer:sss} (where each $\omega$-graded Hom gets localized as an $\mathrm R^\omega_\cat K$-module) coincides with the central localization in the sense of~\cite{MR3163513}:
\[
S^{-1} \cat K \cong \tilde S^{-1}_\cat K \cat K.
\]
\end{cor}

\begin{proof}
The correspondence of multiplicative systems works in fact for any tightening $(R,G,\pi, \{g_\gamma\}, \{\varphi_x\})$ of any central 2-ring $\cat R_\cat M$ of~$\cat K$, in a similar way to the correspondence of ideals in \Cref{thrm:agreement} but with an easier proof we now sketch.
If $S\subseteq R^\hmg$ is a multiplicative system of the graded ring (that is $(R^\hmg)^\times\subseteq S$ and $SS\subseteq S$), we let $\mathrm e(S)$ be the homogeneous system of $\cat R_\cat M$ generated by $\varphi(S)$ (with $\varphi:=\sqcup_{x} \varphi_{x\in G}\colon R^\hmg \to \cat R_\cat M$ as before).
Thus $\mathrm e(S)$ is the smallest set of morphisms of $\cat R$ containing~$\varphi(S)$, containing all isomorphisms, and stable under composition and under twists by any objects of~$\cat R_\cat M$. 

We claim that $\mathrm e(S)$ is equal to the closure of $\varphi(S)$ under translation in~$\cat R_\cat M$.
Indeed, denote the latter closure by~$S'$. We clearly have $S'\subseteq \mathrm e(S)$ by definition, and to prove $\mathrm e(S)\subseteq S'$ (and hence the claim) it suffices to show that $S'$ is closed under composition. 
Thus let $a\colon g\to h$ and $b\colon h\to k$ be two composable maps in~$S'$. 
We must show that $ba$ is a translate of some element of~$\varphi(S)$, and by applying $g^{-1}\otimes -$ to the problem we may assume that $g=\unit$.
Recall (\Cref{rmk:translate}) that $a$ and $b$ may be written as follows for some $r,s\in S$, some objects $\ell,\ell' \in \cat R_\cat M$, and some isomorphisms $u,v,u',v'$ (where $\gamma = \pi(|r|)$ and $\delta=\pi(|s|)$):
\[
\xymatrix{
\ell \otimes \unit 
	\ar[r]^-{\ell \otimes \varphi(r)} &
\ell \otimes g_\gamma 
	\ar[d]^\simeq_v
	\ar@{-->}@/^4ex/[dd]^(.3){=:\,w} & \\
\unit 
	\ar[r]^-a
	\ar[u]^u_\simeq
	\ar@{-->}[ur]
	\ar@{-->}[dr] &
h 
	\ar[r]^(.6){b}
	\ar[d]_{u'}^\simeq &
k \\
& 
\ell' \otimes \unit
	\ar[r]^-{\ell'\otimes \varphi(s)} &
\ell' \otimes g_\delta
	\ar[u]^\simeq_{v'} 
}
\]
In this diagram we can moreover assume $\ell=\unit$ (because of the isomorphism~$u$), hence that $\ell' = g_\gamma$ (using $u' \circ v$).
With these hypotheses $w:= u' v$ is an automorphism of~$g_\gamma$, hence we may apply \Cref{lem:magic} to the above triangle of segmented maps in order to obtain the next commutative diagram with $\tilde w=g_\gamma^{-1}\otimes w$:
\[
\xymatrix@C=12pt{
\unit
	\ar[rrr]^-a
	\ar[d]_{\tilde w}^\simeq &&& g_\gamma \ar[rrr]^-b &&& k \\
\unit 
	\ar[rr]^-u_-\simeq &&
\unit
	\ar[rr]^-{\varphi(r)} &&
g_\delta
	\ar[rr]^-{g_\gamma \otimes \varphi(s)} &&  
g_\gamma \otimes g_\delta 
	\ar[u]_{v'}^\simeq
}
\]
This diagram displays $ba$ as a translate of $(g_\gamma \otimes \varphi(s))\circ \varphi(r)$, which is a translate of $\varphi(sr)$ by definition of a tightening; hence by transitivity~$ba$ is a translate of $\varphi(sr)$ which lies in $\varphi(S)$, as claimed.

In the other direction, if $\tilde S \subseteq \cat R_\cat M$ is a homogeneous multiplicative system we simply set $\mathrm r(\tilde S) := \varphi^{-1}( \tilde S)$,  
which is clearly a multiplicative system of~$R$ containing~$(R^\hmg)^\times$. 
Using the above claim and the fact that each $\tilde S$ is translation-closed, it is now straightforward to check that~$\mathrm r$ and~$\mathrm e$ are mutually inverse bijections.

Finally, the agreement of localized categories (the ``moreover'' part) follows from the isomorphism $\tilde S^{-1}_\cat K \cat K \cong \cat K/\cat J_{\tilde S}$ of \Cref{thrm:central-loc} and the isomorphism $S^{-1}\cat K \cong \cat K/\cat J_S$ of \cite[Thm.\,3.6]{balmer:sss}, together with the immediate observation that the thick tensor ideals $\cat J_S$ and $\cat J_{\tilde S}$ are identical.  
\end{proof}

\section{Periodic locus for derived permutation modules}
\label{sec:peri-locus-perm}

\sectionauthor{Martin Gallauer and Sam K.\,Miller} 
The goal of this appendix is to supply a proof of \Cref{sta:dperm-periodic-locus}.

    \begin{rmk}
    Fix a finite group~$G$ and a commutative ring~$k$.
    Any $G$-space $X$ defines a functor on the orbit category,
    \[
    G/H\mapsto \operatorname{Map}_G(G/H,X)=X^H.
    \]
    When $X$ is a $G$-CW complex this can be improved to a functor taking values in CW complexes.
    Composing with the ``cellular chain complex'' functor it yields a coefficient system
    \[
    k[X^?]\colon \Orop{G}\to \Ch_k
    \]
    valued in chain complexes of~$k$-modules.
    We denote by $\tilde{k}[X^?]$ its reduced version.
    \end{rmk}
    
    \begin{rmk}
    Observe that the inclusion $B(\Weyl{G}{H})\simeq B(\Weyl{G}{H})^{\operatorname{op}}\to \Orop{G}$ of the full subcategory on the orbit $G/H$ endows the value of any coefficient system at~$G/H$ with a $k$-linear $\Weyl{G}{H}$-action.
    In fact, with this action $\tilde{k}[X^H]$ is a complex of $\Weyl{G}{H}$-\emph{permutation} modules, by construction.
    \end{rmk}
    We recall the following (special case of a) construction of ``forerunners'' from \cite{miller2025permutationtwistedcohomologyremixed} for $p$-groups and \cite{GM26} for all finite groups.
    \begin{cns}
    Let $G$ and $k$ as above.
    A finite-dimensional real $G$-representation~$V$ comes with an \emph{$h$-mark} function
    \[
    h_V\colon H\mapsto \dim_{\RR}(V^H)
    \]
    from subgroups of~$G$ to~$\ZZ$.
    It is monotone.
    We call $H$ a \emph{pivot} for~$V$ if it is maximal among subgroups~$H'$ with $\dim(V^{H'})=\dim(V^H)$.
    We call $V$ \emph{$k$-orientable} if the action of $N_G(H)$ on the reduced $k$-homology of the (fixed-points) representation sphere $(S^{V})^H=S^{V^H}$ is trivial for every pivot~$H$.

    Fix a $k$-orientable $V$ and a pivot~$H$.
    The associated representation sphere $S^V$ has a (finite) $G$-CW complex structure and we therefore obtain the coefficient system $\tilde{k}[(S^{V})^?]$.
    Our assumption provides a quasi-isomorphism of chain complexes
    \begin{equation}
    \label{eq:incl-top}
    \Sigma^{h_V(H)}k\to \tilde{k}[(S^V)^H]
    \end{equation}
    of permutation $N_G(H)$-modules.
    Let $G/e\to G/H$ be the morphism sending $g$ to $gH$.
    Thus a morphism
    \begin{equation}
    \label{eq:incl-subgroup}
    \tilde{k}[(S^V)^H]\to \tilde{k}[(S^V)^e]
    \end{equation}
    which, by functoriality, is $N_G(H)$-equivariant.
    Composing~\eqref{eq:incl-top} and~\eqref{eq:incl-subgroup} we get a morphism of chain complexes of permutation $N_G(H)$-modules
    \[
    \iota'^H_{V}\colon \Sigma^{h_V(H)}k\to \tilde{k}[(S^V)^e]
    \]
    and we make it $G$-equivariant by setting
    \[
    \iota^H_V:=\operatorname{tr}^G_{N_G(H)}\iota'^H_V=\sum_{g\in G/N_G(H)}g\cdot \iota'^H_{V}\ \colon\ \Sigma^{h_V(H)}k\to \tilde{k}[(S^V)^e]
    \]
    (using the fact that the domain has trivial $G$-action).
    \end{cns}

    \begin{lem}
    \label{sta:Psi-iota}
    Keep the notation as above.
    Assume, moreover, that $H\leq G$ is a $p$-subgroup and $k$ is a field of characteristic~$p$.
    Applying modular fixed points $\Psi^H$ to $\iota^H_V$ yields a quasi-isomorphism
    \[
    \Psi^H(\iota^H_V)\ \colon\ \Sigma^{h_V(H)}k\to \tilde{k}[(S^V)^H].
    \]
    \end{lem}
    \begin{proof}
    Let $n=h_V(H)$.
    Consider an $n$-cell $D^n\times G/L$ of~$S^V$.
    If  $H\leq_G L$ then necessarily $H=_GL$ since $H$ is a pivot.
    If $H\not\leq_G L$ then the cell has no $H$-fixed points.
    It follows that the $N_G(H)$-equivariant map $\iota'^H_{V}$ in degree~$n$ is given by
    \[
    k\to \oplus_i k(\Weyl{G}{H})
    \]
    determined by a matrix $(\lambda_i\eta'_i)$ with $\lambda_i\in k$ (not all~$0$), $\eta'_i\colon k\to k(\Weyl{G}{H})$ the unit for the restriction-induction adjunction (associated with $H\to N_G(H)$).
    It then follows that $\iota^H_V$ is given in degree~$n$ by
    \[
    (\lambda_i\eta_i)\colon k\to \oplus_ik(G/H)\to \tilde{k}[(S^V)^e]_n
    \]
    with $\eta_i$ similarly the unit for the adjunction, this time for the inclusion $H\to G$.
    Under the identification
    \[
    \Psi^H(\tilde{k}[(S^V)^e]_n)=\oplus k((G/H)^H)=\oplus k(N_G(H)/H)\subseteq \oplus k(G/H)
    \]
    it is clear that $\Psi^H(\iota_V^H)$ is the quasi-isomorphism~\eqref{eq:incl-top} we started with.
    \end{proof}

    \begin{exa}
    \label{exa:doubled-regular}
    Since $(S^V)^H=S^{(V^H)}$ is a sphere of dimension $n:=h_V(H)$ and $N_G(H)$ acts on it through homeomorphisms, we see that the action on its reduced homology $k=\tilde{\operatorname{H}}_n((S^V)^H;k)$ corresponds to a homomorphism $N_G(H)\to \{\pm 1\}$.
    It follows that any representation of the form $V=W^{\oplus 2}$ is $k$-orientable.

    We will use $V=(\RR G)^{\oplus 2}$, two copies of the regular representation.
    Observe that $(\RR G)^H$ has a basis given by the $H$-orbits of~$G/e$ hence is of dimension $[G:H]$.
    It follows that $h_V(H)=2[G:H]$ and in particular, every $H\leq G$ is a pivot.
    \end{exa}

    \begin{prop}
    \label{sta:DPerm-per-locus-bound}
    Let $G$ be a finite group, $k$ a field of characteristic~$p>0$.
    Let $\cP\in\Spec(\DPerm(G;k)^\omega)$ be a non-closed point.
    Then it has non-zero period.
    \end{prop}
    \begin{proof}
    By~\cite[Remark~16.18]{MR4946248}, there are two $p$-subgroups $H_1,H_2\leq G$ with $|H_1|\neq |H_2|$ and such that $\cP$ specializes to both closed points $\gm(H_i)$.
    Recall that
    \[
    \gm(H_i)=\ker
    \left(
    \DPerm(G;k)^\omega\xrightarrow{\Psi^{H_i}}\DPerm(\Weyl{G}{H_i};k)^\omega\xrightarrow{\Upsilon} \operatorname{D}^b(k(\Weyl{G}{H_i}))
    \right)
    \]
    so that $\cP$ contains both of these kernels.
    For $V=(\RR G)^{\oplus 2}$ as in \Cref{exa:doubled-regular} we have two morphisms
    \[
    \iota^{H_i}_V\colon \Sigma^{2[G:H_i]}k\to \tilde{k}[(S^V)^e]
    \]
    in $\DPerm(G;k)^\omega$, which by \Cref{sta:Psi-iota}, become isomorphisms in $\DPerm(G;k)^\omega/\cP$.
    It follows that in the latter category we have
    \[
    \Sigma^{2[G:H_1]}k\cong \Sigma^{2[G:H_2]}k
    \]
    and this concludes the proof.
    \end{proof}

    \begin{rmk}
    The proof of \Cref{sta:DPerm-per-locus-bound} suggests the possibility that the local periods of~$\DPerm(G;k)^\omega$ are determined by the real representations and specialization data.
    One may show that this is indeed true for elementary abelian $p$-groups, $Q_8$, and $D_8$ but not true in general.
    We plan on taking up this question elsewhere.
    \end{rmk}

    \begin{proof}[Proof of \Cref{sta:dperm-periodic-locus}]
    Recall that $\gm(H)$ is the image of the unique closed point in $\cV_{\Weyl{G}{H}}$ under the map~$\check{\psi}^H:=\Spc(\check{\Psi}^H)$.
    By \Cref{exa:rep}, this closed point has local period~$0$.
    It then follows from \Cref{sta:d-periodic-basics-geometric} that we have $\per_{\DPerm(G;k)^{\omega}}(\gm(H))=0$ as well.
    This shows that the closed points~$\gm(H)$ are not in the periodic locus.

    The reverse inclusion is \Cref{sta:DPerm-per-locus-bound}.
    \end{proof}

    \printbibliography
\enlargethispage{1.0\baselineskip}
\end{document}